\documentclass[11pt,reqno]{amsart}
\usepackage[dvipdfm,a4paper,top=30mm,bottom=25mm,left=25mm,right=25mm]{geometry}
\usepackage{amsmath,amssymb,amsthm,mathtools,ulem}
\usepackage{comment,mathrsfs,bm,cancel}

\usepackage[colorlinks=true,linkcolor=blue,citecolor=blue]{hyperref}

\def\eqref#1{\textcolor{blue}{(\ref{#1})}}

\theoremstyle{definition}
\newtheorem{theorem}{Theorem}[section]
\newtheorem{proposition}[theorem]{Proposition}
\newtheorem{corollary}[theorem]{Corollary}
\newtheorem{lemma}[theorem]{Lemma}
\newtheorem{definition}[theorem]{Definition}

\newtheorem{remark}[theorem]{Remark}

\newtheorem{question}{Question}

\renewcommand{\Re}{\operatorname{Re}}

\def\({\left(}
\def\){\right)}
\def\<{\left\langle}
\def\>{\right\rangle}
\newcommand{\norm}[1]{\lVert #1\rVert}

\numberwithin{equation}{section} %

\allowdisplaybreaks[4]

\def\rnum#1{\expandafter{\romannumeral #1}} 
\def\Rnum#1{\uppercase\expandafter{\romannumeral #1}}

\makeatletter
    
    \@addtoreset{equation}{section}
  \makeatother


\title[mass-subcritical nonlinear Schr\"odinger system]{A sharp scattering threshold level for mass-subcritical nonlinear Schr\"odinger system}

\author[M. Hamano and S. Masaki]{Masaru Hamano and Satoshi Masaki}
\address[Masaru Hamano]{Department of Mathematics, Graduate School of Science and Engineering Saitama University, Shimo-Okubo 255, Sakura-ku, Saitama-shi, Saitama 338-8570, Japan}
\email{m.hamano.733@ms.saitama-u.ac.jp}

\address[Satoshi Masaki]{Department of Systems Innovation, Graduate School of Engineering Sciences, Toyonaka, Osaka, Japan}
\email{masaki@sigmath.es.osaka-u.ac.jp}

\begin{document}

\maketitle

\begin{abstract}
 In this paper, we consider the quadratic nonlinear Schr\"odinger system in three space dimensions.
% \begin{equation}
% \notag
% \begin{cases}
% \hspace{-0.4cm}&\displaystyle{i\partial_tu+\Delta u=-2v\bar{u},}\\
% \hspace{-0.4cm}&\displaystyle{i\partial_tv+\frac{1}{2}\Delta v=-u^2.}
% \end{cases}
% \end{equation}
Our aim is to obtain sharp scattering criteria.
Because of the mass-subcritical nature, it is difficult to do so in terms of conserved quantities.
The corresponding single equation is studied by the second author and a sharp scattering criteria is established by 
introducing a distance from a trivial scattering solution, the zero solution.
By the structure of the nonlinearity we are dealing with, the system admits a scattering solution which is a pair of zero solution and a linear solution.
Taking this fact into account, we introduce a new optimizing quantity and give a sharp scattering criterion in terms of it.
\end{abstract}

% \tableofcontents 

%Introduction

\section{Introduction}

We consider the following quadratic Schr\"odinger system in three space dimensions:
\begin{equation}
\tag{NLS}\label{NLS}
\begin{cases}
\hspace{-0.4cm}&\displaystyle{i\partial_tu+\Delta u=-2v\bar{u}\qquad(t,x)\in\mathbb{R}\times\mathbb{R}^3,}\\
\hspace{-0.4cm}&\displaystyle{i\partial_tv+\frac{1}{2}\Delta v=-u^2\qquad(t,x)\in\mathbb{R}\times\mathbb{R}^3,}
%\hspace{-0.4cm}&\displaystyle{u(x,0)=u_0(x),\ v(x,0)=v_0(x)\qquad x\in\mathbb{R}^3,}
\end{cases}
\end{equation}
where $i=\sqrt{-1}$, $u,v:\mathbb{R}\times\mathbb{R}^3\longrightarrow\mathbb{C}$ are unknown functions, $\Delta=\sum_{j=1}^3\frac{\partial^2}{\partial x_j^2}$, and $\overline{u}$ is the complex conjugate of $u$.
%$u_0,v_0:\mathbb{R}^3\longrightarrow\mathbb{C}$ are given functions,
%
In this paper, we study \eqref{NLS} in a weighted space. For $s\in[0,\frac{3}{2})$, we define the homogeneous weighted space $\mathcal{F}\dot{H}^s(\mathbb{R}^3)$ by the norm
\begin{align*}
\|f\|_{\mathcal{F}\dot{H}^s(\mathbb{R}^3)}:=\|\mathcal{F}f\|_{\dot{H}^s(\mathbb{R}^3)}=\||x|^sf\|_{L_x^2(\mathbb{R}^3)}.
\end{align*}
Here, $\mathcal{F}$ denotes the Fourier transform on $\mathbb{R}^3$, that is,
\begin{align*}
\mathcal{F}f(\xi):=\widehat{f}(\xi):=(2\pi)^{-\frac{3}{2}}\int_{\mathbb{R}^3}e^{-ix\cdot\xi}f(x)dx.
\end{align*}
We consider the system \eqref{NLS} with the initial condition
\begin{align}
(u(0),v(0))=(u_0,v_0)\in\mathcal{F}\dot{H}^\frac{1}{2}\times\mathcal{F}\dot{H}^\frac{1}{2}.\label{003}
\end{align}

At least formally, the system \eqref{NLS} has the following two conserved quantities: One is \textit{mass}
\begin{equation}\label{D:mass}
	M[u,v] := \int_{\mathbb{R}^3} (|u(x)|^2 + 2|v(x)|^2 ) dx
\end{equation}
and the other is \textit{energy}
\begin{equation}\label{D:energy}
	E[u,v] := \int_{\mathbb{R}^3} \(|\nabla u(x)|^2 + \frac12 |\nabla v(x)|^2  - 2 \Re (u(x)^2\overline{v(x)}) \) dx.
\end{equation}

Let us make the notion of solution clear. 
We need a slight modification of the notion compared with $L^2$ or $H^1$ solutions because the Schr\"odinger flow is not unitary
in the homogeneous weighted space $\mathcal{F}\dot{H}^\frac{1}{2}$. 
% the treatment of the nonlinear Schr\"odinger equation in a homogeneous weighted space needs a care.
The corresponding single equation is studied in this space
in \cite{Mas15} (see also \cite{GinOzaVel94,NakOza02} and references therein).

\begin{definition}[Solution]\label{D:sol}
Let $I\subset\mathbb{R}$ be a nonempty time interval.
We say that a pair of functions $(u,v) : I\times\mathbb{R}^3\rightarrow\mathbb{C}^2$ is a solution to \eqref{NLS} on $I$ if $(e^{-it\Delta}u(t),e^{-\frac{1}{2}it\Delta}v(t))\in(C(I;\mathcal{F}\dot{H}^\frac{1}{2}))^2$ and the Duhamel formula
\begin{equation*}
\begin{cases}
\hspace{-0.4cm}&\displaystyle{e^{-it\Delta}u(t)=e^{-i\tau\Delta}u(\tau)+2i\int_{\tau}^te^{-is\Delta}(v\overline{u})(s)ds},\\
\hspace{-0.4cm}&\displaystyle{e^{-\frac{1}{2}it\Delta}v(t)=e^{-\frac{1}{2}i\tau\Delta}v(\tau)+i\int_{\tau}^te^{-\frac{1}{2}is\Delta}(u^2)(s)ds}
\end{cases}
\end{equation*}
holds  in $\mathcal{F}\dot{H}^\frac{1}{2}$ for any $t,\tau\in I$, where $e^{it\Delta}:=\mathcal{F}^{-1}e^{-it|\xi|^2}\mathcal{F}$ is the Schr\"odinger group.
We express the maximal interval of existence of $(u,v)$ by $I_{\text{max}}=(T_{\text{min}},T_{\text{max}})$.
We say $(u,v)$ is forward time-global (resp.~backward time-global) if $T_{\text{max}}=\infty$ (resp.~if $T_{\text{min}}=-\infty$).
\end{definition}

This definition of solutions is not time-translation invariant.
That is, if $(u,v)$ is a solution to \eqref{NLS}, then $(u(\cdot+\tau),v(\cdot+\tau))$ is not necessarily a solution for $\tau\in\mathbb{R}$.
On the other hand, solutions to \eqref{NLS} remains invariant under the rescaling
\begin{align*}
(u(t,x),v(t,x))\mapsto (u_{[\lambda]}(t,x),v_{[\lambda]}(t,x)):=(\lambda^2u(\lambda^2t,\lambda x),\lambda^2v(\lambda^2t,\lambda x))
\end{align*}
for $\lambda>0$.
Corresponding transform of initial data is as follows:
\begin{align}
(\phi(x),\psi(x))\mapsto(\phi_{\{\lambda\}}(x),\psi_{\{\lambda\}}(x)):=(\lambda^2\phi(\lambda x),\lambda^2\psi(\lambda x))\label{002}
\end{align}
for $\lambda>0$.
The $\dot{H}_x^{-\frac{1}{2}}$-norm is invariant under the above scaling transformation.
In other words, the equation \eqref{NLS} with initial condition \eqref{003} is scaling critical.
In this sense our problem is mass-subcritical.
Remark that the equation \eqref{NLS} with initial condition \eqref{003} is also a critical problem, in the sense that the $\mathcal{F}\dot{H}^\frac{1}{2}$-norm is invariant under the scaling \eqref{002}.
% By persistence of regularity type argument, one can see that if $(u_0,v_0) \in \mathcal{F} {H}^\frac12:= \mathcal{F} \dot{H}^\frac12 \cap L^2$ then the solution satisfies
% $(e^{-it\Delta}u(t),e^{-\frac{1}{2}it\Delta}v(t))\in(C(I;\mathcal{F}{H}^\frac{1}{2}))^2$

To state the local well-posedness for \eqref{NLS}, we introduce the function spaces $\dot{X}_m^{s,r}(t)$, $W_1$, and $W_2$ defined by norms
\begin{align*}
	\|f\|_{\dot{X}_m^{s,r}(t)}:=\Bigl\|\Bigl(-\frac{t^2}{m^2}\Delta\Bigr)^\frac{s}{2}e^{-\frac{im|x|^2}{2t}}f\Bigr\|_{L_x^r},\ \ \ 
	\|f\|_{W_j}:=\|f\|_{L_t^{6,2}\dot{X}_{2^{{}^{j-2}}}^{\frac{1}{2},\frac{18}{7}}}.
\end{align*}
For a space $\dot{X}_m^{s,r}(t)$, we omit the second exponent when $r=2$, that is, $\dot{X}_m^s(t)=\dot{X}_m^{s,2}(t)$.
We discuss these function spaces in more detail in Subsection \ref{subsec:Function spaces} and Subsection \ref{ss:space}, below.
The following is our result on the local well-psedness.
More detailed version is given later as Theorem \ref{Th. Local well-posedness2}.

\begin{theorem}[Local well-posedness]\label{Th. Local well-posedness}
For any %$t_0\in\mathbb{R}$ and 
$(u_0,v_0)\in \mathcal{F}\dot{H}^\frac{1}{2}\times \mathcal{F}\dot{H}^\frac{1}{2}$
there exist an open interval $I \ni 0$ and a unique solution $(u,v)\in(C(I;\dot{X}_{1/2}^{1/2})\cap W_1(I)) \times (C(I;\dot{X}_1^{1/2})\cap W_2(I))$ to \eqref{NLS} with a initial condition $(u(0),v(0))=(u_0,v_0)$.
% Moreover, there exists a universal constant $\delta>0$ such that if the data satisfies
% \begin{align*}
% \|(e^{it\Delta}u_0,e^{\frac{1}{2}it\Delta}v_0)\|_{W_1(I)\times W_2(I)}\leq\delta,
% \end{align*}
% then the solution satisfies
% \begin{align*}
% \|(u,v)\|_{W_1(I)\times W_2(I)}\lesssim\|(e^{it\Delta}u_0,e^{\frac{1}{2}it\Delta}v_0)\|_{W_1(I)\times W_2(I)}.
% \end{align*}
\end{theorem}

Now, we turn to the large time behavior of solutions to \eqref{NLS}. % given in Theorem \ref{Th. Local well-posedness}.
There are various types of possible behavior. %depending on the choice of the initial data.
In this paper, we are interested in scattering solutions defined as follows:
\begin{definition}
We say that a solution $(u,v)$ scatters forward in time (resp. backward in time) if $(u,v)$ is forward time-global (resp. backward time-global) and there exists $(u_+,v_+)\in\mathcal{F}\dot{H}^\frac{1}{2}\times\mathcal{F}\dot{H}^\frac{1}{2}$ (resp.~$(u_-,v_-)\in\mathcal{F}\dot{H}^\frac{1}{2}\times\mathcal{F}\dot{H}^\frac{1}{2}$) such that
\begin{align}\label{E:definition of scattering}
	\lim_{t\rightarrow+\infty}\|(e^{-it\Delta}u(t),e^{-\frac{1}{2}it\Delta}v(t))-(u_+,v_+)\|_{\mathcal{F}\dot{H}^\frac{1}{2}\times\mathcal{F}\dot{H}^\frac{1}{2}}=0\hspace{0.28cm}&\\
	\left(\text{resp.~} \lim_{t\rightarrow-\infty}\|(e^{-it\Delta}u(t),e^{-\frac{1}{2}it\Delta}v(t))-(u_-,v_-)\|_{\mathcal{F}\dot{H}^\frac{1}{2}\times\mathcal{F}\dot{H}^\frac{1}{2}}=0\right)&.\notag
\end{align}
We say that a solution $(u,v)$ scatters when it scatters both forward and backward in time.
\end{definition}
One has several equivalent characterizations of the scattering.
For example, the forward-in-time scattering is equivalent
 to the boundedness of $\norm{u}_{W_1((\tau, T_{\max}))}$
or of $\norm{v}_{W_2((\tau, T_{\max}))}$ for some $\tau \in I_{\max}$.
% We have more equivalent characterizations of scattering.
See Proposition \ref{Scattering criterion} for the details.

\subsection{Criterion for scattering by conserved quantities}

We are interested in obtaining a sharp condition for scattering.
This subject is recently extensively studied based on a concentration compactness/rigidity type argument after Kenig and Merle \cite{KenMer06}.
As for the Schr\"odinger system \eqref{NLS}, the first author treated five dimensions \cite{Ham18} and Inui, Kishimoto, and Nishimura treated four dimensions \cite{InuKisNis18}.
In these results, a sharp condition for scattering is given in terms of conserved quantities.
For example, in the four dimensions, the equation is mass-critical and we have the following simple criterion in terms of the mass:
a solution scatters for both time direction if the mass of a solution is smaller than that of the ground state solution.
This is a natural extension of the single equation case \cite{Dod15,KilTaoVis09,KilVisZha08}.
In five dimensions, condition for blowup is also studied.

Remark that if a solution $(u,v)$ scatters then any scaled solution $(u_{[\lambda]},v_{[\lambda]})$ scatters also.
Consequently, any criterion for scattering is scaling invariant.\footnote{Suppose that we have a criteria ``if $(u,v)$ 
satisfies a condition $P$ then $(u,v)$ scatters." Then, it actually reads as ``if there exists $\lambda>0$ such that $(u_{[\lambda]},v_{[\lambda]})$ 
satisfies a condition $P$ then $(u,v)$ scatters." The latter criteria is scaling invariant in such a sense that the validity of its assumption is left invariant under the scaling.}
Hence, if we look for a criteria given in terms of some characteristic quantity of a solution, it is natural that the quantity is scaling invariant.
Recall that mass $M$ is scaling invariant in four dimensions, and the product of two quantities $ME$ is a scaling invariant in five dimensions.
In the previous results \cite{Ham18,InuKisNis18}, these quantities play a crucial role in criterion there.

However, in the three dimensional case, it would be difficult to give a criteria in terms of the mass and the energy.
%  which make sense for suitable subclass of solutions,
They are not scaling invariant:
\[
	M[\phi_{\{\lambda\}},\psi_{\{\lambda\}}] = \lambda M[\phi,\psi], \quad
	E[\phi_{\{\lambda\}},\psi_{\{\lambda\}}] = \lambda^3 E[\phi,\psi]
\]
for $\lambda>0$. 
Furthermore, the both right hands sides has the positive power of $\lambda$.
It means that one can make the both magnitudes of $M$ and of $E$ small or large at the same time by scaling.
This is a feature of the mass-subcritical case.
Thus, we may not have a criteria similar to those in four or five dimensional cases, as one may not construct a scaling invariant quantity by a combination or product of (positive powers of) these quantities.

Hence, in the sequel, we look for a criteria which is not given in terms of the conserved quantities, as in \cite{Mas15,Mas17,Mas16,MasSeg18}.

\subsection{Trivial scattering set and minimization of non-scattering solutions}
It can be said that
the main purpose of this paper is to investigate transition phenomena between scattering solutions near the trivial scattering set and
other solutions.
For comparison, let us recall the single NLS equation with the gauge invariant quadratic nonlinearity:
\begin{align}\label{E:NLSsingle}
	i \partial_t w + \Delta w = |w|w, \qquad(t,x)\in\mathbb{R}\times\mathbb{R}^3
\end{align}
with $w(0)=w_0 \in \mathcal{F} \dot{H}^{1/2}$. Let $\mathcal{S}_{+,\mathrm{single}}$ be the set of initial data for which the corresponding solution scatters forward in time.
Since the zero solution is a scattering solution, we have
\begin{align*}
	\{ 0 \} \subset \mathcal{S}_{+,\mathrm{single}}.
\end{align*}
Then, it is natural to define a ``size" of a solution by the distance from the zero solution.
This fact leads us to a study of a quantity like
\begin{align}\label{E:singlelc}
	 \inf \left\{ \| w_0 \|_{\mathcal{F} \dot{H}^\frac12}  : w_0 \not\in \mathcal{S}_{+,\mathrm{single}}  \right\}.
\end{align}
It is known \cite{Mas15,Mas17} that this infimum value is strictly smaller than the size of the ground state solution for \eqref{E:NLSsingle}
and further that
there exists a minimizer to this quantity.

Let us go back to the system case.
We define $\mathcal{S}_{+}$ as the set of initial data $(u_0,v_0) \in \mathcal{F} \dot{H}^{\frac12} \times \mathcal{F} \dot{H}^{\frac12}$
for which the corresponding solution scatters forward in time.
A straightforward generalization of the quantity \eqref{E:singlelc} is
\begin{align}\label{E:systemlc}
	 \inf \left\{ (\| u_0 \|_{\mathcal{F} \dot{H}^\frac12}^2 + \alpha \| v_0 \|_{\mathcal{F} \dot{H}^\frac12}^2)^{1/2}  : (u_0,v_0) \not\in \mathcal{S}_{+}  \right\}
\end{align}
with some constant $\alpha>0$.
However, there may not be a strong motivation to study this quantity other than the similarity to \eqref{E:singlelc}
because the quantity is not relevant to conserved quantities of \eqref{NLS}.
Hence, we want to find a different way of sizing which is based on a system nature.
To this end, we look at
the fact 
that not only the zero solution $(0,0)$ but also all solutions of the form $(0,e^{\frac12it\Delta}v_0)$ can be also regarded
as a trivial scattering solution for arbitrary $v_0 \in \mathcal{F} \dot{H}^{\frac12}$.
Taking this fact into account, one natural choice of the ``size" of an initial data would be the distance from the set $\{0\} \times \mathcal{F} \dot{H}^{\frac12}$.
This choice leads us to consider the following optimization problem:
\begin{equation}\label{D:ellv0}
\ell_{v_0}:=\inf \{\|u_0\|_{\mathcal{F}\dot{H}^\frac{1}{2}}: (u_0,v_0)\not\in \mathcal{S}_+ \} \in (0,\infty].
\end{equation}
By using a stability type argument, we will show that $\ell_{v_0}>0$ for any $v_0 \in \mathcal{F} \dot{H}^{\frac12}$ (see, Proposition \ref{Small data scattering}).
The following criteria is obvious by the definition of $\ell_{v_0}$.

\begin{proposition}[Sharp small data scattering]%\label{Main result}
\label{P:SSDS}
Let $(u_0,v_0)\in\mathcal{F}\dot{H}^\frac{1}{2}\times\mathcal{F}\dot{H}^\frac{1}{2}$ and
let $(u,v):I_{\text{max}}\times\mathbb{R}^3\longrightarrow\mathbb{C}^2$ be the solution to \eqref{NLS} with the initial condition \eqref{003}.
If $\|u_0\|_{\mathcal{F}\dot{H}^\frac{1}{2}}<\ell_{v_0}$ then $(u,v)$ scatters forward in time.
\end{proposition}

The above criteria ``$\|u_0\|_{\mathcal{F}\dot{H}^\frac{1}{2}}<\ell_{v_0}$" is sharp in such a sense that
the number $\ell_{v_0}$ may not be replaced with any larger number.
The questions which we address in this paper are the following two:
(a) to obtain a condition which implies $\ell_{v_0}$ is finite;
(b) to show the existence of a minimizer to $\ell_{v_0}$ (when $\ell_{v_0}$ is finite).

\subsection{Stability of ground state}

A characteristic property of the mass-subcritical case is that the ground state is orbitally stable in $H^1\times H^1$
\cite{Din20}.
The ground state solution for \eqref{NLS} is a solution of the form
\[
	(e^{i\omega t} Q_{1,\omega}(x), e^{2i \omega t} Q_{2,\omega}(x)),
\]
where $\omega>0$ and $(Q_{1,\omega}(x), Q_{2,\omega}(x))$ is a positive radial solution to the elliptic equation
\[
	- \Delta Q_{1,\omega} + \omega Q_{1,\omega} = 2 Q_{1,\omega} Q_{2,\omega},\quad
	- \frac12 \Delta Q_{2,\omega} + 2\omega Q_{2,\omega} = Q_{1,\omega}^2.
\]
The orbital stability implies that there exists an open neighborhood $\mathcal{N} \subset \mathbb{R}^2$ of $(1,1)$ such that 
\[
	(c_1 Q_{1,\omega}(x), c_2 Q_{2,\omega}(x)) \notin S_+ 
\]
for all $(c_1,c_2) \in \mathcal{N}$. (This also follows from Theorem \ref{Nonpositive energy}, below.)
Hence, the ground state solutions are not optimizers to our problems.
In particular, $\ell_{Q_{2,\omega}}$ is strictly smaller than $\norm{Q_{1,\omega}}_{\mathcal{F} \dot{H}^\frac{1}{2}}$.
Similarly, $(Q_{1,\omega}(x), Q_{2,\omega}(x))$ is not a solution to \eqref{E:systemlc} for any choice of $\alpha>0$.

\subsection{Main results}
It will turn out that the following quantity $\ell_{v_0}^\dagger$ plays an important role in the analysis of $\ell_{v_0}$.
% In what follows, 
% we are basically interested in the case $\ell_{v_0}<\infty$ since otherwise $v(0)=v_0$ implies that the corresponding solution
% scatters forward in time.  
\begin{definition}
For $v_0\in\mathcal{F}\dot{H}^\frac{1}{2}$ and $0\leq \ell<\infty$, we let
\begin{equation*}
L_{v_0}(\ell):=\sup\left\{\|(u,v)\|_{W_1([0,T_{\max}))\times W_2([0,T_{\max}))}\middle|
\begin{array}{l}
(u,v)\text{ is the solution to \eqref{NLS} on }[0,T_{\max}),\\[0.1cm]
v(0)=v_0,\,\,
\|u(0)\|_{\mathcal{F}\dot{H}^\frac{1}{2}}\leq\ell,\ u(0)\in\mathcal{F}\dot{H}^\frac{1}{2}
\end{array}
\right\},
\end{equation*}
where $W_j([0,T_{\max}))$ $(j=1,2)$ is a Strichartz-like function space defined in Subsection \ref{ss:space}, below.
Further, define
\begin{equation}\label{D:ellv0d}
\ell_{v_0}^\dagger:=\sup\{\ell:L_{v_0}(\ell)<\infty\} \in (0,\infty].
\end{equation}
\end{definition}

We have $\ell_{v_0}^\dagger \le \ell_{v_0}$ by the definition of the each quantities (see Lemma \ref{Comparison of ell and ell^dagger} for more detail).
Intuitively, this can be seen by noticing that if $\|u_0\|_{\mathcal{F}\dot{H}^\frac{1}{2}}<\ell_{v_0}^\dagger$ then not only $(u,v)$ scatters forward in time but also 
we have a priori bound $\|(u,v)\|_{W_1([0,\infty))\times W_2([0,\infty))}\le L_{v_0}(\|u_0\|_{\mathcal{F}\dot{H}^\frac{1}{2}}) <\infty$.
As for the single-equation \eqref{E:NLSsingle}, it is known that these two kinds of quantities coincide each other (see \cite{Mas16}).
Our first result is as follows.

\begin{theorem}\label{T:l0}
$\ell_{v_0}^\dagger=\min (\ell_0,\ell_{v_0})$ is true for any $v_0 \in \mathcal{F} \dot{H}^{\frac12}$, including the case where the both sides are infinite.
In particular, $\ell_{0}^\dagger=\ell_0$ holds. 
% Moreover, $ \ell_{v_0}^\dagger  \le \ell_0$ holds for any $v_0 \in \mathcal{F}\dot{H}^{\frac12}$.
\end{theorem}
It is worth noting that $\ell_{v_0}^\dagger=\infty$ guarantees $\ell_{v_0}=\infty$ but the inverse is not necessarily true.
Our interest in the sequel is to see what happens when $\ell_{v_0}^\dagger<\infty$.

In the case $v_0=0$, we have $\ell_{0}^\dagger = \ell_0$, including the case both are infinite, as seen in Theorem \ref{T:l0}.
% The situation is similar to the single equation case:
\begin{theorem}\label{T:case0}
If 
$\ell_{0}^\dagger<\infty$ then there exists a minimizer $(u^{(0)}(t),v^{(0)}(t))$ to $\ell_{0}(=\ell_0^\dagger)$ such that
\begin{enumerate}
\item $v^{(0)}(0)=0$ and $\norm{u^{(0)}(0)}_{\mathcal{F} \dot{H}^{\frac12}} = \ell_{0}$;
\item $(u^{(0)}(t),v^{(0)}(t))$ does not scatter forward in time.
\end{enumerate}
\end{theorem}
So far, we do not know whether $\ell_0^\dagger<\infty$ or not.
It will turn out that this question is important to understand the attainability of $\ell_{v_0}$ for all non-zero $v_0$.
One quick conclusion of $\ell_0^\dagger=\infty$ is $\ell_{v_0}=\ell_{v_0}^\dagger$ for all $v_0$, which follows from Theorem \ref{T:l0}.
We will resume this subject later.

Let us move on to the case $v_0\neq 0$.
Suppose $\ell_{v_0}^\dagger <\infty$.
% For each fixed $v_0\neq 0$, we have either 
Then, we have either
\begin{equation}\label{E:CaseAB}
\ell_{v_0}^\dagger=\ell_{v_0}%(\le \ell_0)
\quad \text{ or } \quad
%(\ell_0=)
\ell_{v_0}^\dagger< \ell_{v_0}.
\end{equation}
% Further, Theorem \ref{T:l0} implies that we have $\ell_{v_0}^\dagger = \ell_0$ in the latter case.
% Remark that we do not know whether $\ell_0$ is finite.
The following Theorem \ref{T:case1} is about the first case and Theorem \ref{T:case2} is about the second case, respectively.

% We now state our result in the first case. 
% In view of Theorem \ref{T:l0}, we have 
\begin{theorem}\label{T:case1}
Fix $v_0 \in \mathcal{F} \dot{H}^{\frac12}\setminus \{0\}$.
Suppose that $\ell_{v_0}^\dagger=\ell_{v_0}<\ell_0$.
Then, there exists a minimizer $(u^{(v_0)}(t),v^{(v_0)}(t))$ to $\ell_{v_0}$ such that
\begin{enumerate}
\item $v^{(v_0)}(0)=v_0$ and $\norm{u^{(v_0)}(0)}_{\mathcal{F} \dot{H}^{\frac12}} = \ell_{v_0}$;
\item $(u^{(v_0)}(t),v^{(v_0)}(t))$ does not scatter forward in time.
\end{enumerate}
\end{theorem}
The case $\ell_{v_0}^\dagger=\ell_{v_0}=\ell_0<\infty$ is excluded in the above theorem.
We consider this exceptional case in Remark \ref{R:exception}, below.

Let us consider the second case of \eqref{E:CaseAB}.
In this case, the following strange thing occurs:
% Fix $v_0\in \mathcal{F}\dot{H}^{\frac12}$.
Take $u_0 \in \mathcal{F}\dot{H}^{\frac12}$ with $\norm{u_0}_{\mathcal{F}\dot{H}^{\frac12}}= \ell_{v_0}^\dagger$
and consider the corresponding solution $(u(t),v(t))$ with the data $(u_0,v_0)$. 
Then, in one hand, the solution $(u(t),v(t))$ scatters forward in time for any choice of such $u_0$
since $\norm{u_0}_{\mathcal{F}\dot{H}^{\frac12}}<\ell_{v_0}$.
However, on the other hand,
for arbitrarily large number $N>0$, one can choose $u_0 \in \mathcal{F}\dot{H}^{\frac12}$ so that the
corresponding solution $(u(t),v(t))$ satisfies
\begin{align*}
	\norm{(u,v)}_{W_1([0,\infty)) \times W_2([0,\infty))} \ge N.
\end{align*}
The next theorem tells us how this is ``attained".
Notice that the second case of \eqref{E:CaseAB} occurs only when $\ell_0=\ell_{v_0}^\dagger <\infty$, thanks to Theorem \ref{T:l0}.
Consequently, there is a minimizer to $\ell_0$ in this case, by means of Theorem \ref{T:case0}.

\begin{theorem}\label{T:case2}
Fix $v_0 \in \mathcal{F} \dot{H}^{\frac12}\setminus\{0\}$.
Suppose that $\ell_{v_0}^\dagger < \ell_{v_0}$.
Pick a sequence $\{ u_{0,n} \}_{n} \subset \mathcal{F} \dot{H}^{\frac12}$ satisfying
$\norm{u_{0,n}}_{\mathcal{F} \dot{H}^{\frac12}}< \ell_{v_0}^\dagger$ for all $n\ge1$,
\begin{align*}
	\lim_{n\to\infty} \norm{u_{0,n}}_{\mathcal{F} \dot{H}^{\frac12}} = \ell_{v_0}^\dagger,
\end{align*}
and
\begin{align*}
	\lim_{n\to\infty} \norm{(u_{n},v_{n})}_{W_1([0,\infty)) \times W_2([0,\infty))} = \infty,
\end{align*}
where $(u_n(t),v_n(t))$ is a solution with the initial data $(u_n(0),v_n(0))=(u_{0,n},v_0)$.
Then, there exist a subsequence of $n$, a minimizer $(u^{(0)},v^{(0)})$ to $\ell_0$, and two sequences $\{\xi_n\}_n \subset \mathbb{R}^3$ and
$\{h_n\}_n \subset 2^\mathbb{Z}$ such that
\begin{align*}
	|\log h_n| + |\xi_n| \longrightarrow \infty
\end{align*}
and
\begin{align*}
	e^{-ix\cdot h_n \xi_n }(u_{0,n})_{\{h_n\}} \longrightarrow u^{(0)}(0) \quad \text{in }\mathcal{F} \dot{H}^{\frac12}
\end{align*}
hold along the subsequence.
In particular, along the same subsequence, it holds for any $\tau \in (0, T_{\max}(u^{(0)},v^{(0)}))$ that
\begin{align*}
(u_n(t),v_n(t))
	&=\(e^{-it|\xi_n|^2+ix\cdot\xi_n}u^{(0)}_{[h_n^{-1}]}\(t,x-2\xi_nt\),e^{-2it|\xi_n|^2+2ix\cdot\xi_n}v^{(0)}_{[h_n^{-1}]}\(t,x-2\xi_nt\)\) \\
	&\hspace{8.0cm} +(0,e^{\frac{1}{2}it\Delta}v_0) + o_{\dot{X}_\frac{1}{2}^{\frac{1}{2}}(t)\times \dot{X}_1^{\frac{1}{2}}(t)}(1)
\end{align*}
for $0 \le t \le \tau h_n^2$.
\end{theorem}

\begin{remark}\label{R:exception}
The special case $\ell_{v_0}^\dagger = \ell_0 = \ell_{v_0}<\infty$ ($v_0\neq 0$) is not included in the above two theorems.
In this exceptional case, the conclusion of Theorem \ref{T:case1} and/or Theorem \ref{T:case2} holds.
Namely, if there does not exist a minimizer to $\ell_{v_0}$ as in Theorem \ref{T:case1},
then the conclusion of Theorem \ref{T:case2} is true.
\end{remark}

Let us summarize the above results.
Let $v_0 \in \mathcal{F} \dot{H}^{\frac12}$ be a given function.
If $\ell_{v_0}^\dagger=\infty$ then we have $\ell_{v_0}=\infty$ (Theorem \ref{T:l0}) and hence any solution satisfies $v(0)=v_0$ scatters forward in time (Proposition \ref{P:SSDS}).
On the other hand, if $\ell_{v_0}^\dagger<\infty$ and $v_0\neq0$ then we have either Theorem \ref{T:case1} or Theorem \ref{T:case2} according to the dichotomy \eqref{E:CaseAB}. Remark that the first case in \eqref{E:CaseAB} contains an exceptional case discussed in Remark \ref{R:exception}.
When $v_0=0$ then we do not have the dichotomy, we have $\ell_0=\ell^\dagger_0$ (Theorem \ref{T:l0}).
If $\ell_0$ is finite then there exists a minimizer (Theorem \ref{T:case0}).

The question whether $\ell_0=\infty$ or not would be an interesting question to the system \eqref{NLS}.
We do not have the answer yet. Let us formulate the problem without using our terminology:
\begin{question}
In \eqref{NLS}, does $v_0=0$ implies scattering of the corresponding solution for any $u_0$?
\end{question}
If it were true, that is, if $\ell_0=\ell_0^\dagger=\infty$ then Theorem \ref{T:l0} tells us that
$\ell_{v_0}^\dagger=\ell_{v_0}$ is true for any $v_0$, as mentioned above.

\smallbreak

Although we do not know the exact value of $\ell_{v_0}$, we are able to have a condition which implies the finiteness of $\ell_{v_0}$
and to give an upper bound for $\ell_{v_0}$.
A simple one is a condition in terms of the energy.

\begin{theorem}\label{Nonpositive energy}
Fix nontrivial $u_0, v_0 \in \mathcal{F} \dot{H}^{\frac12} \cap H^1$.
If $E[u_0,v_0]\le 0$ and then the corresponding solution $(u,v)$ does not scatter.
In particular, $\ell_{v_0} \le \norm{u_0}_{\mathcal{F}\dot{H}^{\frac12}}$.
\end{theorem}

In our context, we want to find a condition which is stated in terms of $v_0$ only. %which implies the finiteness of $\ell_{v_0}$.
We give two criteria in this direction.
The first one is for large data case:
\begin{corollary}\label{Large data nonscattering}
For any $v_0 \in \mathcal{F} \dot{H}^{\frac12} \cap H^1$ with $v_0\neq0$, %and for any $\varepsilon>0$
% If there exists
% $u_0 \in \mathcal{F} \dot{H}^{\frac12} \cap H^1$ such that
% \[
% 	\int_{\mathrm{R}^3} \overline{u_0}^2 v_0 dx \neq 0,
% \]
% then 
there exists 
$c_0>0$ such that the estimate $\ell_{cv_0} \lesssim_{v_0} c^{\frac12}$
holds for any $c \ge c_0$.
\end{corollary}

% Our last result is a sufficient condition to have $\ell_{v_0}<\infty$ for given $v_0$.

The second one is criterion for a specific $v_0$:

\begin{corollary}\label{Sufficient condition of finite ell}
Pick $v_0 \in \mathcal{F} \dot{H}^{\frac12} \cap H^1$.
If there exists $\theta \in \mathbb{R}$ such that a Schr\"odinger operator $-\Delta - 2 \text{Re}(e^{i\theta} v_0)$ has 
a negative eigenvalue then $\ell_{v_0}<\infty$. 
Moreover, if $\varphi \in \mathcal{F}\dot{H}^\frac{1}{2}\cap H^1$
is a real-valued eigenfunction 
associated with a negative eigenvalue $\tilde{e}<0$ of $-\Delta - 2 \text{Re}(e^{i\theta} v_0)$
then the estimate
\[
	\ell_{v_0} \le \frac{\norm{\varphi}_{\mathcal{F}\dot{H}^{\frac12}}}{\sqrt{2|\tilde{e}|} \norm{\varphi}_{L^2}} \norm{\nabla v_0}_{L^2}
\]
is true.
\end{corollary}

\begin{remark}
The estimate given in
Corollary \ref{Sufficient condition of finite ell} is scaling invariant.
Indeed, if $\varphi(x)$ is an eigenfunction of $-\Delta - 2\Re (e^{i\theta} v_0 )$
associated with a negative eigenvalue $\tilde{e}$ then
$\varphi_{\{\lambda\}}(x)$ is an eigenfunction of $-\Delta - 2\Re (e^{i\theta} (v_0)_{\{\lambda\}} )$
and the corresponding eigenvalue is $\lambda^2 \tilde{e}$, where $f_{\{\lambda\}}$ denotes the scaling of $f$ defined in \eqref{002}.
\end{remark}

\begin{remark}
It is possible to study the optimizing problem \eqref{E:systemlc}.
Let us formulate in an abstract setting.
Let $f(x,y)$ be a function on $[0,\infty)\times [0,\infty)$ such that it is continuous and strictly increasing with respect to the both variables and that $f(0,0)=0$. 
We define
\begin{equation}\label{E:definition of lf}
	\ell_f := \inf \{ f(\norm{u_0}_{\mathcal{F} \dot{H}^{\frac12} },\norm{v_0}_{\mathcal{F} \dot{H}^{\frac12} }) \ :\ (u_0,v_0) \not \in \mathcal{S}_+ \}.
\end{equation}
Then, we have
the relation 
\[
	\ell_f = \inf_{v_0 \in \mathcal{F} \dot{H}^{\frac12}} f(\ell_{v_0}, \norm{v_0}_{\mathcal{F} \dot{H}^{\frac12}})
	=\inf_{v_0 \in \mathcal{F} \dot{H}^{\frac12}} f(\ell_{v_0}^\dagger, \norm{v_0}_{\mathcal{F} \dot{H}^{\frac12}}).
\]
% (Remark that the first identity is rather obvious by definition, however the second is not in view of Theorem \ref{T:l0}.)
Moreover,
there exists a minimizer, say $(u_{0,f}, v_{0,f})$, to $\ell_f$.
The minimizer satisfies $\ell_{v_{0,f}}^\dagger = \ell_{v_{0,f}}$,
and $u_{0,f}$ is a minimizer to $\ell_{v_{0,f}}$, i.e., $\ell_{v_{0,f}} = \norm{u_{0,f}}_{\mathcal{F} \dot{H}^{\frac12} }$.
(See Theorem \ref{T:ellfminimizer} for the details.)
Intuitively, this implies the following: Let us consider the level set $\Omega_r := \{(u_0,v_0): f(\norm{u_0}_{\mathcal{F} \dot{H}^{\frac12} },\norm{v_0}_{\mathcal{F} \dot{H}^{\frac12} }) \le r\}$. Then $\Omega_0 = \{(0,0)\}$ and, as the ``radius" $r$ increases from zero, $\Omega_{r}$ first contains a non-scattering initial data
exactly at it touches the curve $v_0 \mapsto \ell_{v_0}^\dagger$.
And it touches the curve $v_0 \mapsto \ell_{v_0}$ at the same point.
We would emphasize that it is true for any choice of $f$.
It can be said that a function $v_0$ such that $\ell_{v_0}^\dagger < \ell_{v_0}$ is, if exists, never found as a minimizer to a minimizing problem of the type \eqref{E:definition of lf}.
Further, for any choice of $f$, the ground state solution is not a minimizer, as mentioned above.
\end{remark}

The rest of the paper is organized as follows.
In Section \ref{Preliminary}, we collect some notations and inequalities.
Then, we define function spaces.
In Section \ref{Local well-posedness and Stability}, we prove local well-posedness for \eqref{NLS} and give a necessary and sufficient condition for scattering.
Then, we check that the solutions to \eqref{NLS} with nonpositive energy does not scatter (Theorem \ref{Nonpositive energy}).
In Section \ref{Properties of L and l}, we investigate properties of $L_{v_0}^\dagger$ and $\ell_{v_0}^\dagger$.
In Section \ref{Sec:Linear profile decomposition}, we obtain linear profile decomposition theorem (Theorem \ref{Linear profile decomposition}).
In Section \ref{Proof of main theorem}, we prove Theorem \ref{T:l0}, Theorem \ref{T:case1}, and Theorem \ref{T:case2} and consider the optimizing problem $\ell_f$.
In Scetion \ref{Proof of corollaries}, we prove corollaries of Theorem \ref{Nonpositive energy}.

\section{Preliminary}\label{Preliminary}

In this section, we prepare some notations and estimates used throughout the paper.

\subsection{Notations}

For non-negative $X$ and $Y$, we write $X\lesssim Y$ to denote $X\leq CY$ for a constant $C>0$.
If $X\lesssim Y\lesssim X$, we write $X\sim Y$.
The dependence of implicit constants on parameters will be indicated by subscripts when necessary, e.g. $X\lesssim_uY$ denotes $X\leq CY$ for some $C=C(u)$.
We write $a'\in[1,\infty]$ to denote the H\"older conjugate to $a\in[1,\infty]$, that is, $\frac{1}{a}+\frac{1}{a'}=1$ holds.
For $s\in\mathbb{R}$, the operator $|\nabla|^s$ is defined as the Fourier multiplier operator with multiplier $|\xi|^s$, that is, $|\nabla|^s=\mathcal{F}^{-1}|\xi|^s\mathcal{F}$.
For a set $A \in \mathbb{R}^d$, ${\bf 1}_A(x)$ stands for the characteristic function of $A$. 

We recall the standard Littlewood-Paley projection operators.
Let $\phi$ be a radial cut-off function satisfies ${\bf 1}_{\{ |\xi| \le 4/3\}} \le \phi \le {\bf 1}_{\{ |\xi| \le 5/3\}}$.
% \begin{equation}
% \phi(\xi):= \label{073}
% \begin{cases}
% &\hspace{0.05cm}1\hspace{0.98cm}(|\xi|\leq \frac{4}{3}),\\
% &\hspace{-0.4cm}\text{smooth}\quad{(\frac{4}{3}\leq |\xi|\leq \frac{5}{3})},\\
% &\hspace{0.05cm}0\hspace{0.98cm}(\frac{5}{3}\leq |\xi|).
% \end{cases}
% \end{equation}
For $N\in2^\mathbb{Z}$, the operators $P_N$ is defined as
\begin{equation*}
%\begin{split}
%	&\widehat{P_{\leq N}f}(\xi):=\widehat{f_{\leq N}}(\xi):=\phi_N(\xi)\widehat{f}(\xi),
%	\ \ \ \widehat{P_{>N}f}(\xi):=\widehat{f_{>N}}(\xi):=\left(1\textendash\phi_N(\xi)\right)\widehat{f}(\xi),\\
	\widehat{P_Nf}(\xi):=\widehat{f_N}(\xi):=\psi_N(\xi)\widehat{f}(\xi),
%\end{split}
\end{equation*}
where $\phi_N(x) = \phi(x/N)$ and
\begin{align}
\psi_N(x)=\phi_N(x)-\phi_{N/2}(x). \label{069}
\end{align}

\subsection{The Galilean transform and the Galilean operator}

%The Schr\"odinger group $e^{it\Delta}=\mathcal{F}^{-1}e^{-it|\xi|^2}\mathcal{F}$ is written in physical space as
%\begin{align*}
%[e^{it\Delta}f](x)
%	=(4\pi it)^{-\frac{3}{2}}\int_{\mathbb{R}^3}e^\frac{i|x-y|^2}{4t}f(y)dy,\ \ \ \ t\neq0.
%\end{align*}
%This description deduces the $L_x^1\rightarrow L_x^\infty$ dispersive estimate. 
%Since $e^{it\Delta}$ is unitary on $L_x^2$, interpolation yields the following general dispersive estimates:
%\begin{align*}
%\|e^{it\Delta}f\|_{L_x^r(\mathbb{R}^3)}
%	\lesssim|t|^{-(\frac{3}{2}-\frac{3}{r})}\|f\|_{L_x^{r'}(\mathbb{R}^3)}\ \ \text{ for all }\ \ r\geq2.
%\end{align*}
The identities
\begin{equation}
\begin{split}
[e^{it\Delta}(e^{ix\cdot\xi_0}f)](x)
	&=e^{-it|\xi_0|^2+ix\cdot\xi_0}(e^{it\Delta}f)(x-2t\xi_0),\\[0.1cm]
[e^{\frac{1}{2}it\Delta}(e^{2ix\cdot\xi_0}f)](x)
	&=e^{-2it|\xi_0|^2+2ix\cdot\xi_0}(e^{\frac{1}{2}it\Delta}f)(x-2t\xi_0) \label{008}
\end{split}
\end{equation}
%\textcolor{cyan}{
%\begin{align*}
%[e^{it\Delta}(e^{ix\cdot\xi_0}f)](x)
%	&=\mathcal{F}^{-1}e^{-it|\xi|^2}\mathcal{F}[e^{ix\cdot\xi_0}f](x)
%	=\mathcal{F}^{-1}e^{-it|\xi|^2}[\mathcal{F}f](\xi-\xi_0)\\
%	&=\int_{\mathbb{R}^3}e^{ix\cdot\xi}e^{-it|\xi|^2}\mathcal{F}f(\xi-\xi_0)d\xi
%	=\int_{\mathbb{R}^3}e^{ix\cdot(\xi+\xi_0)}e^{-it|\xi+\xi_0|^2}\mathcal{F}f(\xi)d\xi\\
%	&=e^{-it|\xi_0|^2+ix\cdot\xi_0}\int_{\mathbb{R}^3}e^{i(x-2t\xi_0)\cdot\xi}e^{-it|\xi|^2}\mathcal{F}f(\xi)d\xi\\
%	&=e^{-it|\xi_0|^2+ix\cdot\xi_0}\mathcal{F}^{-1}\left[e^{-it|\xi|^2}\mathcal{F}f\right](x-2t\xi_0)
%	=e^{-it|\xi_0|^2+ix\cdot\xi_0}(e^{it\Delta}f)(x-2t\xi_0),
%\end{align*}
%\begin{align*}
%[e^{\frac{1}{2}it\Delta}(e^{2ix\cdot\xi_0}f)](x)
%	&=\mathcal{F}^{-1}e^{-\frac{1}{2}it|\xi|^2}\mathcal{F}[e^{2ix\cdot\xi_0}f](x)
%	=\mathcal{F}^{-1}e^{-\frac{1}{2}it|\xi|^2}[\mathcal{F}f](\xi-2\xi_0)\\
%	&=\int_{\mathbb{R}^3}e^{ix\cdot\xi}e^{-\frac{1}{2}it|\xi|^2}\mathcal{F}f(\xi-2\xi_0)d\xi
%	=\int_{\mathbb{R}^3}e^{ix\cdot(\xi+2\xi_0)}e^{-\frac{1}{2}it|\xi+2\xi_0|^2}\mathcal{F}f(\xi)d\xi\\
%	&=e^{-2it|\xi_0|^2+2ix\cdot\xi_0}\int_{\mathbb{R}^3}e^{i(x-2t\xi_0)\cdot\xi}e^{-\frac{1}{2}it|\xi|^2}\mathcal{F}f(\xi)d\xi\\
%	&=e^{-2it|\xi_0|^2+2ix\cdot\xi_0}\mathcal{F}^{-1}\left[e^{-\frac{1}{2}it|\xi|^2}\mathcal{F}f\right](x-2t\xi_0)
%	=e^{-2it|\xi_0|^2+2ix\cdot\xi_0}(e^{\frac{1}{2}it\Delta}f)(x-2t\xi_0).
%\end{align*}
%}
imply that the class of solutions to the linear Schr\"odinger equation is invariant under Galilean transform:
\begin{align}
(u(t,x),v(t,x))\mapsto(e^{-it|\xi_0|^2+ix\cdot\xi_0}u(t,x-2\xi_0t),e^{-2it|\xi_0|^2+2ix\cdot\xi_0}v(t,x-2\xi_0t)),\ \ \ \xi_0\in\mathbb{R}^3.\label{009}
\end{align}
% Since the nonlinear equation \eqref{NLS} is invariant under gauge transform:
The invariance is inherited in the nonlinear equation \eqref{NLS}.
% \begin{align*}
% (u(t,x),v(t,x))\mapsto(e^{ix\cdot\xi_0}u(t,x),e^{2ix\cdot\xi_0}v(t,x)),\ \ \ \xi_0\in\mathbb{R}^3,
% \end{align*}
% it follows that \eqref{009} is a symmetry for \eqref{NLS}, as well.

The Galilean operator 
\[
	J_m(t):=x+i\frac{t}{m}\nabla,
\]
which is a multiple of the infinitesimal operator for transforms appearing in \eqref{009},
plays an important role in the scattering theory for mass-subcritical nonlinear Schr\"odinger equation.

We define the multiplication operator
\begin{align}
[\mathcal{M}_m(t)f](x):=e^\frac{im|x|^2}{2t}f(x)\ \ (t\neq0) \label{071}
\end{align}
and the dilation operator
\begin{align}
[\mathcal{D}(t)f](x):=(2it)^{-\frac{3}{2}}f\left(\frac{x}{2t}\right)\ \ (t\neq0). \label{072}
\end{align}
It is well known that the Schr\"odinger group is factorized as $e^{it\Delta}=\mathcal{M}_\frac{1}{2}(t)\mathcal{D}(t)\mathcal{F}\mathcal{M}_\frac{1}{2}(t)$ by using these operators.
This factorization deduces the identity
\begin{align}
e^{it\Delta}\Phi(x)e^{-it\Delta}
	=\mathcal{M}_\frac{1}{2}(t)\Phi(2it\nabla)\mathcal{M}_\frac{1}{2}(-t) \label{007}
\end{align}
for suitable multiplier $\Phi$, where $\Phi(i\nabla)$ denotes the Fourier multiplier operator with multiplier $\Phi(\xi)$, that is, $\Phi(i\nabla):=\mathcal{F}^{-1}\Phi(\xi)\mathcal{F}$.
The Galilean operator is written as follows:
\begin{align*}
J_m(t)
	=e^{\frac{1}{2m}it\Delta}xe^{-\frac{1}{2m}it\Delta}
	=\mathcal{M}_m(t)i\frac{t}{m}\nabla \mathcal{M}_m(-t),
\end{align*}
where the second equality holds for $t\neq0$.
We define a fractional power of $J_m$ by
\begin{align}
J_m^s(t):=e^{\frac{1}{2m}it\Delta}|x|^se^{-\frac{1}{2m}it\Delta}=\mathcal{M}_m(t)\left(-\frac{t^2}{m^2}\Delta\right)^\frac{s}{2}\mathcal{M}_m(-t)\ \ \text{ for }\ \ s\in\mathbb{R}. \label{070}
\end{align}
Remark that the second formula is valid for $t\neq0$.
%and localized Galilean operator $J_m^s(t)$ by
%\begin{align*}
%J_{\leq N}^s(t):=e^{\frac{1}{2m}it\Delta}\phi\left(\frac{x}{N}\right)|x|^se^{-\frac{1}{2m}it\Delta}
%	=\mathcal{M}_m(t)P_{\leq\frac{N}{|t|}}\left(-\frac{t^2}{m^2}\Delta\right)^\frac{s}{2}\mathcal{M}_m(-t),
%\end{align*}
%where $\phi$ is given in \eqref{073} and $P_{\leq\frac{N}{|t|}}$ is given in \eqref{074}.
%$J_N^s$ and $J_{>N}^s$ are defined similarly.

%Combining \eqref{008} and scaling, we get useful formula
%\begin{align*}
%\left[e^{it\Delta}\left(e^{ix\cdot\xi_0}f\left(\frac{\cdot}{\lambda}\right)\right)\right](x)
%	=e^{-it|\xi_0|^2+ix\cdot\xi_0}(e^{it\lambda^{-2}\Delta}f)\left(\frac{x-2t\xi_0}{\lambda}\right).
%\end{align*}

\subsection{Function spaces}\label{subsec:Function spaces}

We define a time-dependent spaces $\dot{X}_m^{s,r}=\dot{X}_m^{s,r}(t)$ by using the norm
\begin{align}
\|f\|_{\dot{X}_m^{s,r}}:=\|J_m^s(t)f\|_{L_x^r(\mathbb{R}^3)}\sim\||t|^s|\nabla|^s\mathcal{M}_m(-t)f\|_{L_x^r(\mathbb{R}^3)}. \label{010}
\end{align}
When $r=2$, we omit it, that is, $\dot{X}_m^s=\dot{X}_m^{s,2}$.
We can see immediately by the definition of $J_m^s$ that the equivalence of norms in \eqref{010} for $t\neq0$.
It is natural to write
\begin{align*}
f\in e^{\frac{1}{2m}it \Delta}\mathcal{F}\dot{H}^s\Longleftrightarrow e^{-\frac{1}{2m}it \Delta}f\in\mathcal{F}\dot{H}^s.
\end{align*}
Then, we have a change of notation\,:\,$e^{\frac{1}{2m}it\Delta}\mathcal{F}\dot{H}^s=\dot{X}_m^s(t)$ by using this description.

We use Lorentz-modified space-time norms.
For an interval $I$, $1\leq q<\infty$, and $1\leq\alpha\leq\infty$, the Lorentz space $L_t^{q,\alpha}(I)$ is defined by using the quasi-norm
\begin{align*}
	\|f\|_{L_t^{q,\alpha}(I)}
		:=\|\lambda|\{t\in I:|f(t)|>\lambda\}|^\frac{1}{q}\|_{L^\alpha((0,\infty),\frac{d\lambda}{\lambda})}.
\end{align*}
For a Banach space $X$, $L_t^{q,\alpha}(I;X)$ is defined as the whole of functions $u:I\times\mathbb{R}^3\longrightarrow\mathbb{C}$ satisfying
\begin{align*}
\|u\|_{L_t^{q,\alpha}(I;X)}:=\left\|\|u(t)\|_{X}\right\|_{L_t^{q,\alpha}(I)}<\infty.
\end{align*}
The following equivalence is useful:
\begin{align*}
\|f\|_{L_t^{q,\alpha}(I)}\sim\|\|f\cdot1_{\{2^{k-1}\leq|f|\leq2^k\}}\|_{L_t^q(I)}\|_{\ell^\alpha(k\in\mathbb{Z})}.
\end{align*}

\subsection{Strichartz estimates}
The standard Strichartz estimates for $e^{it\Delta}$ were proved in \cite{GinVel92,KeeTao98,Str77}.
We also need Strichartz estimates for the spaces $L_t^{q,\alpha}\dot{X}_m^{s,r}$, which were proved in \cite{Mas16,NakOza02}.

\begin{definition}
If a pair $(q,r)$ satisfies
\begin{align*}
2<q<\infty,\ \ \ 2<r<6,\ \ \text{ and }\ \ \frac{2}{q}+\frac{3}{r}=\frac{3}{2},
\end{align*}
then $(q,r)$ is an admissible pair.
\end{definition}
Remark that we do not included two end points $(\infty,2)$ and $(2,6)$ to admissible pairs.
It is because they require exceptional treatments sometimes.

\begin{proposition}[Strichartz estimates, \cite{NakOza02}]\label{Strichartz estimates}
Let $s\geq0$ and $t_0\in I\subset\mathbb{R}$.
\begin{itemize}
\item[(1)] For any admissible pair $(q,r)$, we have
\begin{align*}
\|e^{it\Delta}f\|_{L_t^\infty\dot{X}_m^s\cap L_t^{q,2}\dot{X}_m^{s,r}}\lesssim\|f\|_{\mathcal{F}\dot{H}^s}.
\end{align*}
\item[(ii)] For any admissible pairs $(q,r)$ and $(\alpha,\beta)$, we have
\begin{align*}
\left\|\int_{t_0}^te^{i(t-s)\Delta}F(s)ds\right\|_{L_t^\infty(I;\dot{X}_m^s)\cap L_t^{q,2}(I;\dot{X}_m^{s,r})}\lesssim\|F\|_{L_t^{\alpha',2}(I;\dot{X}_m^{s,\beta'})}.
\end{align*}
\end{itemize}
\end{proposition}

\subsection{Specific function spaces}\label{ss:space}

Throughout this paper, we use the following concrete choice of function spaces.
The same exponents were used in \cite{KilMasMurVis17,Mas15}.

We define
\begin{align*}
\left(\frac{1}{q_1},\frac{1}{r_1}\right)
	:=\left(\frac{1}{6},\frac{7}{18}\right),\ \ \ 
\left(\frac{1}{\widetilde{q}},\frac{1}{\widetilde{r}}\right)
	:=\left(\frac{2}{3},\frac{2}{9}\right).
\end{align*}
The pair $(q_1,r_1)$ is admissible. The pair $(\tilde{q},\tilde{r})$ satisfies the critical scaling relation
$\frac{2}{\widetilde{q}}+\frac{3}{\widetilde{r}}=2$, and is not a admissible pair.
These exponents satisfy the following relations.
\begin{align}
\frac{1}{q_1'}=\frac{1}{\widetilde{q}}+\frac{1}{q_1},\ \ \ \frac{1}{r_1'}=\frac{1}{\widetilde{r}}+\frac{1}{r_1},\ \ \ \frac{1}{\widetilde{q}}-\frac{1}{q_1}=\frac{3}{r_1}-\frac{3}{\widetilde{r}}=\frac{1}{2}.\label{011}
\end{align}
We define the spaces
\begin{align*}
	S^\text{weak}
	:=L_t^{\widetilde{q},\infty}L_x^{\widetilde{r}}=L_t^{\frac{3}{2},\infty}L_x^\frac{9}{2},\ \ \ 
	S
	:=L_t^{\widetilde{q},2}L_x^{\widetilde{r}}=L_t^{\frac{3}{2},2}L_x^\frac{9}{2},\ \ \ 
	W_j
	:=L_t^{q_1,2}\dot{X}_{2^{{\,}^{j-2}}}^{\frac{1}{2},r_1}=L_t^{6,2}\dot{X}_{2^{{\,}^{j-2}}}^{\frac{1}{2},\frac{18}{7}}
\end{align*}
for the solutions and the spaces
\begin{align*}
N_j
	:=L_t^{q_1',2}\dot{X}_{2^{{\,}^{j-2}}}^{\frac{1}{2},r_1'}
	=L_t^{\frac{6}{5},2}\dot{X}_{2^{{\,}^{j-2}}}^{\frac{1}{2},\frac{18}{11}}
\end{align*}
for the nonlinear terms.
We use a notation $S^\text{weak}(I)$ to indicate that the norm is taken over the space-time slab $I\times\mathbb{R}^3$, and similarly for the other spaces.

\subsection{Some estimates} In this section, we collect some estimates.
They easily follow as in \cite{Mas16} (see also \cite{KilMasMurVis17}).

\begin{lemma}[Embeddings]\label{Embeddings}
The following inequalities hold.
\begin{align*}
\|u\|_{S^\text{weak}}\lesssim\|u\|_{S}\lesssim_j \|u\|_{W_j},%,\,\|u\|_{W_2}.
\end{align*}
where $j=1,2$.
\end{lemma}

% \begin{proof}
% The first estimate holds by a property $L_t^{\frac{3}{2},2}\hookrightarrow L_t^{\frac{3}{2},\infty}$ of Lorentz spaces.
% We use Sobolev embedding, \eqref{011}, Lemma \ref{Holder in Lorentz spaces}, and \eqref{010} to get the second inequality. 
% \begin{align*}
% \|u\|_{S}
% 	%&=\|u\|_{L_t^{\frac{3}{2},2}L_x^\frac{9}{2}}
% 	&= \|\mathcal{M}_m(-t)u\|_{L_t^{\frac{3}{2},2}L_x^\frac{9}{2}}
% 	\lesssim\||\nabla|^\frac{1}{2}\mathcal{M}_m(-t)u\|_{L_t^{\frac{3}{2},2}L_x^\frac{18}{7}}
% 	\lesssim \||t|^{-\frac{1}{2}}\|_{L_t^{2,\infty}}\||t|^\frac{1}{2}|\nabla|^\frac{1}{2}\mathcal{M}_m(-t)u\|_{L_t^{6,2}L_x^\frac{18}{7}}\\
% 	&\lesssim \||t|^{-\frac{1}{2}}\|_{L_t^{2,\infty}}\||t|^\frac{1}{2}|\nabla|^\frac{1}{2}\mathcal{M}_m(-t)u\|_{L_t^{6,2}L_x^\frac{18}{7}}
% 	\sim \||t|^{-\frac{1}{2}}\|_{L_t^{2,\infty}}\|u\|_{L_t^{6,2}\dot{X}_m^{\frac{1}{2},\frac{18}{7}}}.
% \end{align*}
% Since $\||t|^{-\frac{1}{2}}\|_{L_t^{2,\infty}}<\infty$, the desired result follows.
% \end{proof}

\begin{lemma}[Nonlinear estimates]\label{Nonlinear estimates}
The following inequalities hold.
\begin{align*}
\|v\overline{u}\|_{N_1}\lesssim\|u\|_{S^\text{weak}}\|v\|_{W_2}+\|u\|_{W_1}\|v\|_{S^\text{weak}}\lesssim\|u\|_{W_1}\|v\|_{W_2},
\end{align*}
\begin{align*}
\|u_1u_2\|_{N_2}\lesssim\|u_1\|_{W_1}\|u_2\|_{S^\text{weak}}+\|u_1\|_{S^\text{weak}}\|u_2\|_{W_1}\lesssim\|u_1\|_{W_1}\|u_2\|_{W_1},
\end{align*}
\begin{gather*}
	\|v\overline{u}\|_{L_t^{\frac{3}{2},2}\dot{X}_\frac{1}{2}^{\frac{1}{2},\frac{18}{13}}}\lesssim\|v\|_{L_t^\infty\dot{X}_1^\frac{1}{2}}\|u\|_{S}+\|v\|_{S}\|u\|_{L_t^\infty\dot{X}_\frac{1}{2}^\frac{1}{2}},\\
	\|u_1u_2\|_{L_t^{\frac{3}{2},2}\dot{X}_1^{\frac{1}{2},\frac{18}{13}}}\lesssim\|u_1\|_{L_t^{\infty}\dot{X}_\frac{1}{2}^\frac{1}{2}}\|u_2\|_{S}+\|u_1\|_{S}\|u_2\|_{L_t^{\infty}\dot{X}_\frac{1}{2}^\frac{1}{2}},
\end{gather*}
\begin{align*}
\left\| \int_{\tau}^t e^{i(t-s)\Delta}(v\overline{u})(s) ds\right\|_{S}\lesssim
\|u\|_{S}\|v\|_{S},
\end{align*}
\begin{align*}
\left\| \int_{\tau}^t e^{\frac{1}2i(t-s)\Delta}(u_1 u_2 )(s) ds\right\|_{S}\lesssim
\|u_1\|_{S}\|u_2\|_{S}.
\end{align*}
\end{lemma}

Remark that the last two are consequences of inhomogeneous Strichartz estimate for non-admissible pairs by Kato \cite{Kat94}.

\begin{lemma}[Interpolation in $\dot{X}_m^{s,r}$]\label{Interpolation in X}
Let $I\subset\mathbb{R}$.
The following inequality holds.
\begin{align*}
\|f\|_{L^{\rho,\gamma}(I;\dot{X}_m^{s,r})}\lesssim\|f\|_{L^{\rho_1,\gamma_1}(I;\dot{X}_m^{s_1,r_1})}^{1-\theta}\|f\|_{L^{\rho_2,\gamma_2}(I;\dot{X}_m^{s_2,r_2})}^\theta
\end{align*}
for $1\leq\rho, \rho_1, \rho_2<\infty$, $1\leq\gamma, \gamma_1, \gamma_2\leq\infty$, $1<r, r_1, r_2<\infty$, $0<\theta<1$ with $\frac{1}{\rho}=\frac{1-\theta}{\rho_1}+\frac{\theta}{\rho_2}$, $\frac{1}{\gamma}=\frac{1-\theta}{\gamma_1}+\frac{\theta}{\gamma_2}$, $s=(1-\theta)s_1+\theta s_2$, and $\frac{1}{r}=\frac{1-\theta}{r_1}+\frac{\theta}{r_2}$.
\end{lemma}

% \begin{proof}
% Since
% \begin{align*}
% \|f\|_{\dot{H}_x^{s,r}}\lesssim\|f\|_{\dot{H}_x^{s_1,r_1}}^{1-\theta}\|f\|_{\dot{H}_x^{s_2,r_2}}^\theta
% \end{align*}
% (see \cite{BerLof76}), we have
% \begin{align*}
% \|f\|_{\dot{X}_m^{s,r}}\sim|t|^s\|\mathcal{M}_m(-t)f\|_{\dot{H}^{s,r}}&\lesssim(|t|^{s_1}\|\mathcal{M}_m(-t)f\|_{\dot{H}^{s_1,r_1}})^{1-\theta}(|t|^{s_2}\|\mathcal{M}_m(-t)f\|_{\dot{H}^{s_2,r_2}})^\theta\\
% &\sim\|f\|_{\dot{X}_m^{s_1,r_1}}^{1-\theta}\|f\|_{\dot{X}_m^{s_2,r_2}}^\theta.
% \end{align*}
% Therefore, Lemma \ref{Holder in Lorentz spaces} deduces
% \begin{align*}
% \|f\|_{L^{\rho,\gamma}(I;\dot{X}_m^{s,r})}&\lesssim\left\|\|f\|_{\dot{X}_m^{s_1,r_1}}^{1-\theta}\|f\|_{\dot{X}_m^{s_2,r_2}}^\theta\right\|_{L^{\rho,\gamma}(I)}\lesssim\|f\|_{L^{\rho_1,\gamma_1}(I;\dot{X}_m^{s_1,r_1})}^{1-\theta}\|f\|_{L^{\rho_2,\gamma_2}(I;\dot{X}_m^{s_2,r_2})}^\theta.
% \end{align*}
% \end{proof}

\begin{lemma}\label{Bounded multiplication operator}
Let $1\leq r<\infty$, $0<s<\frac{3}{r}$ and let $\chi\in \mathcal{S}(\mathbb{R}^3)$, where $\mathcal{S}(\mathbb{R}^3)$ denotes Schwartz space.
Then, a multiplication operator $\chi\times$ is bounded on $\dot{X}_{m}^{s,r}$.
\end{lemma}

% \begin{proof}
% We take $(r_1,r_2)$ with $\frac{1}{r_1}=\frac{1}{r}-\frac{s}{3}$, $\frac{1}{r}=\frac{1}{r_1}+\frac{1}{r_2}$, Using Sobolev's embbeding, we have
% \begin{align*}
% \|\chi f\|_{\dot{X}_m^{s,r}}
% 	&\sim\||t|^s|\nabla|^s\mathcal{M}_m(-t)\chi f\|_{L_x^r}\\
% 	&\lesssim\||t|^s|\nabla|^s\mathcal{M}_m(-t)f\|_{L_x^r}\|\chi\|_{L_x^\infty}+\||t|^s\mathcal{M}_m(-t)f\|_{L_x^{r_1}}\||\nabla|^s\chi\|_{L_x^{r_2}}\\
% 	&\lesssim\||t|^s|\nabla|^s\mathcal{M}_m(-t)f\|_{L_x^r}\|\chi\|_{L_x^\infty}+\||t|^s|\nabla|^s\mathcal{M}_m(-t)f\|_{L_x^r}\||\nabla|^s\chi\|_{L_x^{r_2}}\\
% 	&\lesssim\|f\|_{\dot{X}_m^{s,r}}.
% \end{align*}
% \end{proof}

%\begin{lemma}[Bernstein]\label{Bernstein}
%For $1\leq r\leq q\leq\infty$, the following estimates hold.
%\begin{gather*}
%	\||\nabla|^sf_N\|_{L_x^r}\sim N^s\|f_N\|_{L_x^r}\ \ \text{ for }\ \ s\in\mathbb{R},\\[0.1cm]
%	\||\nabla|^sf_{\leq N}\|_{L_x^r}\lesssim N^s\|f_{\leq N}\|_{L_x^r}\ \ \text{ for }\ \ s\geq0,\\
%	\|f_{\leq N}\|_{L_x^q}\lesssim N^{\frac{d}{r}-\frac{d}{q}}\|f_{\leq N}\|_{L_x^r}.
%\end{gather*}
%\end{lemma}

\begin{lemma}[Square function estimate]\label{Square function estimate}
For $0\leq s\leq 2$ and $1<p<\infty$, we have
\begin{align*}
\||\nabla|^sf\|_{L_x^p}\sim\Bigl\|\Bigl(\sum_{N\in2^\mathbb{Z}}\left|P_N|\nabla|^sf\right|^2\Bigr)^\frac{1}{2}\Bigr\|_{L_x^p}.
\end{align*}
\end{lemma}

The following H\"older's inequality for Lorentz spaces holds.

\begin{lemma}[H\"older in Lorentz spaces, \cite{Hun66,One63}]\label{Holder in Lorentz spaces}
Let $1\leq q,q_1,q_2<\infty$ and $1\leq\alpha,\alpha_1,\alpha_2\leq\infty$ satisfy
\begin{align*}
\frac{1}{q}=\frac{1}{q_1}+\frac{1}{q_2}\ \ \text{ and }\ \ \frac{1}{\alpha}=\frac{1}{\alpha_1}+\frac{1}{\alpha_2}.
\end{align*}
Then, the following estimate holds.
\begin{align*}
\|fg\|_{L_t^{q,\alpha}}\lesssim\|f\|_{L_t^{q_1,\alpha_1}}\|g\|_{L_t^{q_2,\alpha_2}}.
\end{align*}
\end{lemma}

% The following lemma is cited in \cite[Lemma 2.26]{Mas15}.

% \begin{lemma}\label{Norm of bounded domain}
% Let $\mathcal{B}$ be a bounded subset of $\mathbb{R}\!\times\!\mathbb{R}^3$.
% Let $f\in\mathcal{F}\dot{H}^\frac{1}{2}$ and $1\leq p,q\leq\infty$.
% Then, it follows that for any $\varepsilon>0$, there exists a constant $C_\varepsilon>0$ such that
% \begin{align*}
% \|e^{it\Delta}|x|^\frac{1}{2}f\|_{L^2(\mathcal{B})}\leq\varepsilon\|f\|_{\mathcal{F}\dot{H}^\frac{1}{2}}+C_\varepsilon\|e^{it\Delta}f\|_{L_t^{p,\infty}L_x^q}.
% \end{align*}
% \end{lemma}

%Local well-posedness and Stability

\section{Local well-posedness and Stability}\label{Local well-posedness and Stability}

\subsection{Local well-posedness}%\label{Local well-posedness}
In this subsection, we establish a local theory in $(C_t\dot{X}_{1/2}^{1/2}\cap W_1)\!\times\!(C_t\dot{X}_1^{1/2}\cap W_2)$ for \eqref{NLS}.
The result is given as a consequence of Strichartz estimate (Proposition \ref{Strichartz estimates}) and the estimates of the previous subsection (Lemma \ref{Embeddings} and Lemma \ref{Nonlinear estimates}).

Let us first establish a weak version of the local well-posedness.
\begin{proposition}\label{P:Slwp}
There exists a universal constant $\delta>0$ with the following property:
Let $\tau \in \mathbb{R}$. If a pair of functions $(u_\tau,v_\tau) \in \mathcal{S}'(\mathbb{R}^3)^2$ satisfies
\[
	\norm{ (e^{i(t-\tau)\Delta}u_\tau,e^{\frac{1}2i(t-\tau)\Delta}v_\tau) }_{S(I) \times S(I)} \le \delta
\]
for some interval $I \ni \tau$ then there exists a unique pair of functions $(u,v) \in S(I) \times S(I)$ such that
\begin{equation*}
\begin{cases}
\hspace{-0.4cm}&\displaystyle{u(t)=e^{i(t-\tau)\Delta}u_\tau+2i\int_{\tau}^te^{i(t-s)\Delta}(v\overline{u})(s)ds},\\
\hspace{-0.4cm}&\displaystyle{v(t)=e^{\frac{1}{2}i(t-\tau)\Delta}v_\tau+i\int_{\tau}^te^{\frac{1}{2}i(t-s)\Delta}(u^2)(s)ds}
\end{cases}
\end{equation*}
holds in $S(I)\times S(I)$ sense and satisfies
\[
	\norm{(u,v)}_{S(I)\times S(I)} \le 2 \norm{ (e^{i(t-\tau)\Delta}u_\tau,e^{\frac{1}2i(t-\tau)\Delta}v_\tau) }_{S(I) \times S(I)}.
\]
\end{proposition}
This follows by a standard contraction mapping argument with the last two estimates of Lemma \ref{Nonlinear estimates}.
We refer a pair of functions $(u,v)$ in this proposition to as an $S$-solution to \eqref{NLS} on $I$.

Now we are able to establish the following version of the local well-posedness result.
Theorem \ref{Th. Local well-posedness} corresponds to the case $t_0=0$.
\begin{theorem}[Local well-posedness]\label{Th. Local well-posedness2}
For any $t_0\in\mathbb{R}$ and 
$(u_0,v_0)\in\dot{X}_{1/2}^{1/2}(t_0)\times\dot{X}_1^{1/2}(t_0)$
there exist an open interval $I \ni t_0$ and a unique solution $(u,v)\in(C_t(I;\dot{X}_{1/2}^{1/2})\cap W_1(I))\!\times\!(C_t(I;\dot{X}_1^{1/2})\cap W_2(I))$ to \eqref{NLS} with a initial condition $(u(t_0),v(t_0))=(u_0,v_0)$.
Moreover, there exists a universal constant $\delta>0$ such that if the data satisfies
\begin{align*}
\|(e^{i(t-t_0)\Delta}u_0,e^{\frac{1}{2}i(t-t_0)\Delta}v_0)\|_{W_1(I)\times W_2(I)}\leq\delta,
\end{align*}
then the solution satisfies
\begin{align*}
\|(u,v)\|_{W_1(I)\times W_2(I)}\lesssim\|(e^{i(t-t_0)\Delta}u_0,e^{\frac{1}{2}i(t-t_0)\Delta}v_0)\|_{W_1(I)\times W_2(I)}.
\end{align*}
\end{theorem}
\begin{proof}%[Proof of Theorem \ref{Th. Local well-posedness}]
The strategy of the proof is as follows: We first obtain a $S$-solution. Then, we show it is a solution in the sense of Definition \ref{D:sol} by a persistence-of-regularity type argument.

By Lemma \ref{Embeddings} and Proposition \ref{Strichartz estimates}, we have
\begin{align*}
	\norm{(e^{i(t-t_0)\Delta}u_0,e^{\frac{1}2i(t-t_0)\Delta}v_0)}_{S(\mathbb{R})\times S(\mathbb{R})}
	&{}\lesssim \norm{(e^{i(t-t_0)\Delta}u_0,e^{\frac{1}2i(t-t_0)\Delta}v_0)}_{W_1(\mathbb{R}) \times W_2(\mathbb{R})}\\
	&{}\lesssim \norm{(e^{-it_0\Delta}u_0,e^{-\frac{1}2it_0\Delta}v_0)}_{\mathcal{F}\dot{H}^\frac12 \times \mathcal{F}\dot{H}^\frac12}<\infty.
\end{align*}
Hence, we can chose an open interval $I\ni t_0$ so that
\[
	\norm{(e^{i(t-t_0)\Delta}u_0,e^{\frac{1}2i(t-t_0)\Delta}v_0)}_{S(I)\times S(I)} \le \delta.
\]
For this interval, we have a unique $S$-solution $(u,v) \in S(I)\times S(I)$ by Proposition \ref{P:Slwp}.

We shall show this is a solution in the sense of Definition \ref{D:sol}.
By Proposition \ref{Strichartz estimates} and Lemma \ref{Nonlinear estimates}, one has
\begin{align*}
	\|(u,v)\|_{W_1(I)\times W_2(I)}
%		&\leq \bigl\|\bigl(e^{i(t-t_0)\Delta}u_0,e^{\frac{1}{2}i(t-t_0)\Delta}v_0\bigr)\bigr\|_{W_1(I)\times W_2(I)}+\left\|\left(2\int_{t_0}^te^{i(t-s)\Delta}(v\overline{u})(s)ds,\int_{t_0}^te^{\frac{1}{2}i(t-s)\Delta}(u)^2(s)ds\right)\right\|_{W_1(I)\times W_2(I)}\\
		&\leq \bigl\|\bigl(e^{i(t-t_0)\Delta}u_0,e^{\frac{1}{2}i(t-t_0)\Delta}v_0\bigr)\bigr\|_{W_1(I)\times W_2(I)}+c\|v\overline{u}\|_{N_1(I)}+c\|u^2\|_{N_2(I)}\\
		&\leq \bigl\|\bigl(e^{i(t-t_0)\Delta}u_0,e^{\frac{1}{2}i(t-t_0)\Delta}v_0\bigr)\bigr\|_{W_1(I)\times W_2(I)}+c\|(u,v)\|_{W_1(I)\times W_2(I)}\|(u,v)\|_{S(I)\times S(I)}.
\end{align*}
By subdividing the interval $I$ into $\cup_{j=0}^J I_j$ so that we have $c\|(u,v)\|_{S(I_j)\times S(I_j)} \le \frac12$ in each interval.
Suppose $t_0 \in I_0$. We have
\[
	\|(u,v)\|_{W_1(I_0)\times W_2(I_0)} \leq 2\bigl\|\bigl(e^{i(t-t_0)\Delta}u_0,e^{\frac{1}{2}i(t-t_0)\Delta}v_0\bigr)\bigr\|_{W_1(I_0)\times W_2(I_0)}.
\]
Another use of Strichartz' estimate then shows
\[
	\|(u,v)\|_{L_t^\infty(I_0; \dot{X}_{1/2}^{1/2})\times L_t^\infty(I_0;\dot{X}_1^{1/2})}
		\lesssim \|(u_0,v_0)\|_{\dot{X}_{1/2}^{1/2}(t_0)\times \dot{X}_1^{1/2}(t_0)}.
\]
Repeat the argument to obtain $(u,v) \in (L_t^\infty(I; \dot{X}^{1/2}_{1/2}) \cap W_1(I))\times (L_t^\infty(I;\dot{X}_1^{1/2})\cap W_2(I))$.
The latter statement is obvious. We omit the details.
\end{proof}

\subsection{Scattering criterion} In this subsection, we derive a necessary and sufficient condition for scattering (Proposition \ref{Scattering criterion}). 
We also give a scattering result for small data (Proposition \ref{Small data scattering}).

\begin{proposition}[Scattering criterion]\label{Scattering criterion}
% Let $t_0=0$ and $(u_0,v_0)\in\mathcal{F}\dot{H}^\frac{1}{2}\times\mathcal{F}\dot{H}^\frac{1}{2}$ and 
Let $(u(t),v(t))$ be a unique solution to \eqref{NLS} given in Theorem \ref{Th. Local well-posedness}.
Let $I_{\max}=(T_\text{min},T_{\max})$ be the maximal interval of $(u(t),v(t))$. Then, the following seven statements are equivalent.
\begin{itemize}
\item[(1)] $(u,v)$ scatters forward in time;
\item[(2)] %$T_{\max}=\infty$ and 
$\|(u,v)\|_{W_1([\tau,T_{\max}))\times W_2([\tau,T_{\max}))}<\infty$, $\exists \tau \in I_{\max}$;
% \item[(2)] $T_{\max}=\infty$ and $\|(u,v)\|_{S([0,\infty))\times S([0,\infty))}+\|(u,v)\|_{W_1([0,\infty))\times W_2([0,\infty))}<\infty$;
\item[(3)] %$T_{\max}=\infty$ and 
$\|(u,v)\|_{S([\tau,T_{\max}))\times S([\tau,T_{\max}))}<\infty$, $\exists \tau \in I_{\max}$;
\item[(4)] $\|u\|_{W_1([\tau,T_{\max}))}<\infty$, $\exists \tau \in I_{\max}$;
\item[(5)] $\|v\|_{W_2([\tau,T_{\max}))}<\infty$, $\exists \tau \in I_{\max}$;
\item[(6)] $\|u\|_{S([\tau,T_{\max}))}<\infty$, $\exists \tau \in I_{\max}$;
\item[(7)] $\|v\|_{S([\tau,T_{\max}))}<\infty$, $\exists \tau \in I_{\max}$.
\end{itemize}
Similar criterion holds for the backward-in-time scattering.
% by replacing $[\tau,T_{\max})$ with $(T_{\min},\tau]$ in (2)--(7).
\end{proposition}

\begin{proof}
$(1)\Leftrightarrow(2)\Leftrightarrow(3)$ is standard (see, for instance, \cite{Mas16}).
Let us prove they are also equivalent to from (4) to (7).
To this end, it suffices to show that $(6)$ and $(7)$ are equivalent.
It is because (3) is equivalent to ``$(6)$ and $(7)$." 
Further, once the above equivalence is established, the rest $(4)$ and $(5)$ are handled easily:
We have $(4)\Rightarrow(6)\Leftrightarrow(2)\Rightarrow(4)$ and $(5)\Rightarrow(7)\Leftrightarrow(2)\Rightarrow(5)$.

Suppose (6). Then, for any $T \in (0,T_{\max})$, one deduces from Proposition \ref{Strichartz estimates} and Lemma \ref{Nonlinear estimates} that
\[
	\norm{v}_{S([0,T))} \lesssim \norm{v_0}_{\mathcal{F}\dot{H}^\frac{1}{2}} + \norm{u}_{S([0,T))}^2.
\]
Here the implicit constant is independent of $T$. Hence, by letting $T \uparrow T_{\max}$ we obtain (7).

Next, suppose (7). Take $\tau \in (T_{\min},T_{\max})$ to be chosen later. For any $T \in (\tau,T_{\max})$, we see
\[
	\norm{u}_{S((\tau,T))} \leq C \norm{u(\tau)}_{\dot{X}^{1/2}_{1/2}(\tau)} + C\norm{u}_{S((\tau,T))}  \norm{v}_{S((\tau,T))},
\]
where the constant $C$ is independent of $\tau$ and $T$.
We now choose $\tau$ so that $C \norm{v}_{S((\tau,T_{\max}))}\le \frac{1}{2}$.
This is possible by the property (7).
Then, the above inequality implies that
\[
	\norm{u}_{S((\tau,T))} \leq 2C \norm{u(\tau)}_{\dot{X}^{1/2}_{1/2}(\tau)}.
\]
Since $T \in (\tau,T_{\max})$ is arbitrary, we obtain the result.
\end{proof}

We turn to a sufficient condition for scattering.
One of the simplest is by the smallness of the data.

\begin{proposition}[Small data scattering]\label{Small data scattering}
Let $(u_0,v_0)\in\mathcal{F}\dot{H}^\frac{1}{2}\times\mathcal{F}\dot{H}^\frac{1}{2}$ and let $(u,v)$ be a corresponding unique solution given in Theorem \ref{Th. Local well-posedness}. Then, we have the following.
\begin{itemize}
\item[(1)] There exists $\eta_1>0$ such that if $\|(e^{it\Delta}u_0,e^{\frac{1}{2}it\Delta}v_0)\|_{S\times S}\leq\eta_1$, then $(u,v)$ scatters.
\item[(2)] There exists $\eta_2>0$ such that if $\|(e^{it\Delta}u_0,e^{\frac{1}{2}it\Delta}v_0)\|_{W_1\times W_2}\leq\eta_2$, then $(u,v)$ scatters.
\item[(3)] There exists $\eta_3>0$ such that if $\|(u_0,v_0)\|_{\mathcal{F}\dot{H}^\frac{1}{2}\times\mathcal{F}\dot{H}^\frac{1}{2}}\leq\eta_3$, then $(u,v)$ scatters.
\end{itemize}
\end{proposition}

This follows from Proposition \ref{P:Slwp}, Proposition \ref{Scattering criterion}, and Proposition \ref{Strichartz estimates}.
% We omit the details.

% \begin{proof}
% We prove part (1). If $\eta_1$ is sufficiently small, then
% \begin{align*}
% \|(u,v)\|_{W_1\times W_2}\lesssim\|(e^{it\Delta}u_0,e^{\frac{1}{2}it\Delta}v_0)\|_{W_1\times W_2}\leq\eta_1
% \end{align*}
% by Theorem \ref{Th. Local well-posedness}. Applying Proposition \ref{Scattering criterion}, $(u,v)$ scatters. We prove part (2). Using Proposition \ref{Strichartz estimates}, we have
% \begin{align*}
% \|(e^{it\Delta}u_0,e^{\frac{1}{2}it\Delta}v_0)\|_{W_1\times W_2}\lesssim\|(u_0,v_0)\|_{\mathcal{F}\dot{H}^\frac{1}{2}\times\mathcal{F}\dot{H}^\frac{1}{2}}\leq\eta_2.
% \end{align*}
% If $\eta_2$ is sufficiently small, then $(u,v)$ scatters by part (1).
% \end{proof}

\subsection{Nonpositive energy implies failure of scattering}\label{subsec:Nonpositive energy}
In this subsection, we give
a proof of Theorem \ref{Nonpositive energy}.
To begin with, we will prove that if a data belongs $H^1 \times H^1$, in addition, then the corresponding solution given in Theorem \ref{Th. Local well-posedness} stays in $H^1 \times H^1$ and
the mass and the energy make sense and are conserved.
Furthermore, as is well-known, since our equation is mass-subcritical, the conservation of mass implies the solution is global.

\begin{proposition}\label{LWP H1}
For any $t_0\in\mathbb{R}$ and 
$(u_0,v_0)\in (\dot{X}_{1/2}^{1/2}(t_0) \cap H^1 )\times (\dot{X}_1^{1/2}(t_0)\cap H^1)$
there exists a unique time global solution $(u,v)\in(C_t(\mathbb{R};\dot{X}_{1/2}^{1/2}\cap H^1)\cap W_{1,\text{loc}}(\mathbb{R}))\!\times\!(C_t(\mathbb{R};\dot{X}_1^{1/2}\cap H^1)\cap W_{2,\text{loc}}(\mathbb{R}))$ to \eqref{NLS} with 
the initial condition $(u(t_0),v(t_0))=(u_0,v_0)$.
The solution have conserved mass and energy;
\[
	M[u(t),v(t)] = M[u_0,v_0],\quad E[u(t),v(t)] = E[u_0,v_0].
\]
Furthermore, if the solution scatters forward in time 
then  \eqref{E:definition of scattering} holds also in $H^1 \times H^1$
sense.
\end{proposition}
This is done by a persistence-of-regularity argument.
% We omit the proof.
Now, we prove Theorem \ref{Nonpositive energy}.
\begin{proof}[Proof of Theorem \ref{Nonpositive energy}]
Suppose that a solution $(u,v)$ given in Proposition \ref{LWP H1} scatters forward in time.
Then, the limit \eqref{E:definition of scattering} holds also in $H^1 \times H^1$ sense.
Moreover, one sees from Proposition \ref{Scattering criterion} that
$\norm{(u,v)}_{S([0,\infty))\times S([0,\infty))}<\infty$.
Hence, one finds a sequence $\{t_n\}\subset [0,\infty)$, $t_n\to\infty$ as $n\to\infty$, such that
\[
	\norm{u(t_n)}_{L^{\frac92}_x} + \norm{v(t_n)}_{L^{\frac92}_x} \to 0
\]
as $n\to\infty$. Combining with the mass conservation, the above $L^\frac{9}{2}_x$ can be replaced by any $L^p_x$ for $2<p\le \frac{9}{2}$.
In particular, $p=3$ is allowed.
Hence,
\[
	\left| 2 \int_{\mathbb{R}^3} \Re (u(t_n)^2\overline{v(t_n)})  dx \right| \le 2\norm{u(t_n)}_{L^{3}_x}^2 \norm{v(t_n)}_{L^{3}_x}  \longrightarrow 0
\]
as $n\to\infty$. We deduce that
\[
	E[u_0,v_0] = \lim_{n\to\infty} E[u(t_n),v(t_n)] = \norm{\nabla u_+}_{L_x^2}^2 + \frac12 \norm{\nabla v_+}_{L_x^2}^2 \ge 0.
\]
Further, $E[u_0,v_0]=0$ implies $(u_+,v_+)=(0,0)$. By \eqref{E:definition of scattering} in $H^1$ and the mass conservation implies
$(u_0,v_0)=(0,0)$
\end{proof}

\subsection{Stability} In this subsection, we establish a stability result. 
%(Proposition \ref{Short time perturbation} and Proposition \ref{Long time perturbation}).
Roughly speaking, the proposition implies that two solutions are also close each other 
if their initial data are close and the equations for them are close.

\begin{proposition}[Long time perturbation]\label{Long time perturbation}
Let $I$ be a time interval with $t_0\in I$ and $M>0$. Let $(\widetilde{u},\widetilde{v}):I\times\mathbb{R}^3\longrightarrow\mathbb{C}^2$ satisfy
\begin{equation}
\notag
\begin{cases}
\hspace{-0.4cm}&\displaystyle{i\partial_t\widetilde{u}+\Delta\widetilde{u}=-2\widetilde{v}\overline{\widetilde{u}}+e_1},\\
\hspace{-0.4cm}&\displaystyle{i\partial_t\widetilde{v}+\frac{1}{2}\Delta\widetilde{v}=-\widetilde{u}^2+e_2}
\end{cases}
\end{equation}
and
\begin{align*}
\|(\widetilde{u},\widetilde{v})\|_{W_1(I)\times W_2(I)}\leq M
\end{align*}
for some functions $e_1$, $e_2$. Let $(u_0,v_0)\in\dot{X}^{1/2}_{1/2}(t_0)\times\dot{X}^{1/2}_1(t_0)$ and let $(u,v)$ be a corresponding
solution to \eqref{NLS} with $(u(t_0),v(t_0))=(u_0,v_0)$ given by Theorem \ref{Th. Local well-posedness}.
% and satisfy
% \begin{equation}
% \notag
% \begin{cases}
% \hspace{-0.4cm}&\displaystyle{i\partial_tu+\Delta u=-2v\overline{u}},\\
% \hspace{-0.4cm}&\displaystyle{i\partial_tv+\frac{1}{2}\Delta v=-u^2},\\
% \hspace{-0.4cm}&\displaystyle{(u(t_0),v(t_0))=(u_0,v_0)}.
% \end{cases}
% \end{equation}
There exist $\varepsilon_1=\varepsilon_1(M)>0$ and $c=c(M)>0$ such that if
\begin{align*}
\|(\widetilde{u}(t_0),\widetilde{v}(t_0))-(u_0,v_0)\|_{\dot{X}_{1/2}^{1/2}(t_0)\times\dot{X}_1^{1/2}(t_0)}+\|(e_1,e_2)\|_{N_1(I)\times N_2(I)}\le \varepsilon
\end{align*}
for some $0\le \varepsilon<\varepsilon_1$ then the maximal existence interval of $(u,v)$ contains $I$ and 
the solution satisfies
\begin{align*}
	\|(u,v)-(\widetilde{u},\widetilde{v})\|_{(L_t^\infty(I;\dot{X}_{1/2}^{1/2})\cap W_1(I))\times (L_t^\infty(I;\dot{X}_1^{1/2})\cap W_2(I))}
		\leq c\varepsilon.
\end{align*}
\end{proposition}

For the proof, see \cite{Mas15}.

\section{Properties of $L_{v_0}$ and $\ell_{v_0}^\dagger$}\label{Properties of L and l}

In this section, we collect properties of $L_{v_0}$ and $\ell_{v_0}^\dagger$.

\begin{proposition}\label{Control of L}
For any $v_0\in\mathcal{F}\dot{H}^\frac{1}{2}$, there exist $\varepsilon_1>0$ and $\delta>0$ such that
\begin{align*}
L_{v_0}(\varepsilon)\leq\|e^{\frac{1}{2}it\Delta}v_0\|_{W_2([0,\infty))}+c\delta\varepsilon
\end{align*}
holds for $0\le \varepsilon<\varepsilon_1$.
Here, the constants $\varepsilon_1>0$ and $\delta>0$ depend only on
$\|e^{\frac{1}{2}it\Delta}v_0\|_{W_2([0,\infty))}$.
In particular, $\ell^{\dagger}_{v_0}>0$ for any $v_0 \in \mathcal{F} \dot{H}^{1/2}.$
% where $\varepsilon_1$ is given in Proposition \ref{Long time perturbation}.
\end{proposition}

% \begin{proof}
% We consider a initial data $(0,v_0)$. The solution $(0,e^{\frac{1}{2}it\Delta}v_0)$ to \eqref{NLS} with a data $(0,v_0)$ satisfies
% \begin{align*}
% \|(0,e^{\frac{1}{2}it\Delta}v_0)\|_{W_1\times W_2}=\|e^{\frac{1}{2}it\Delta}v_0\|_{W_2}.
% \end{align*}
% We consider a initial data $(u_0,v_0)$ with $\left\|u_0\right\|_{\mathcal{F}\dot{H}^\frac{1}{2}}\leq\varepsilon$. Then,
% \begin{align*}
% \left\|(u_0,v_0)-(0,v_0)\right\|_{\mathcal{F}\dot{H}^\frac{1}{2}\times\mathcal{F}\dot{H}^\frac{1}{2}}=\left\|u_0\right\|_{\mathcal{F}\dot{H}^\frac{1}{2}}\leq\varepsilon.
% \end{align*}
% Applying Proposition \ref{Long time perturbation},
% \begin{align*}
% \left\|(u,v)-(0,e^{\frac{1}{2}it\Delta}v_0)\right\|_{W_1\times W_2}\lesssim_{\|e^{\frac{1}{2}it\Delta}v_0\|_{W_2}}\varepsilon.
% \end{align*}
% Therefore,
% \begin{align*}
% \left\|(u,v)\right\|_{W_1\times W_2}\leq\|e^{\frac{1}{2}it\Delta}v_0\|_{W_2}+c\delta\varepsilon.
% \end{align*}
% Taking supremum on $\{u_0:\|u_0\|_{\mathcal{F}\dot{H}^\frac{1}{2}}\leq\varepsilon\}$, we have
% \begin{align*}
% L_{v_0}(\varepsilon)\leq\|e^{\frac{1}{2}it\Delta}v_0\|_{W_2}+c\delta\varepsilon.
% \end{align*}
% \end{proof}

\begin{proof}
Apply Proposition \ref{Long time perturbation} with $(\widetilde{u},\widetilde{v})=(0,e^{\frac{1}{2}it\Delta}v_0)$,
for which $e_1=e_2=0$.
\end{proof}

In the sequel, we call a function which values in an extended real numbers $\mathbb{R} \cup \{+\infty\}$ as an extended function.

\begin{definition}[Continuous extended function]\label{D:continuity}
We say that a non-decreasing extended function $f:\mathbb{R}\longrightarrow[0,\infty]$ is continuous at $x_0\in\mathbb{R}$ if $f$ satisfies the following:
\begin{itemize}
\item When $f(x_0)<\infty$, $f$ is continuous in the usual sense.
\item When $f(x_0)=\infty$ and $0<x_0<\infty$, either ``there exists $\varepsilon>0$ such that $f(x)=\infty$ for any $x\in(x_0-\varepsilon,x_0]$'' or ``$f(x)<\infty$ for any $x\in[0,x_0)$ and $\displaystyle\lim_{x\rightarrow x_0-0}f(x)=\infty$'' is true.
% \item When $f(x_0)=\infty$ and $x_0=0$, $f(x)=\infty$ for any $x\in[0,\infty)$.
\end{itemize}
Moreover, we say a non-decreasing extended function $f$ is continuous if $f(x)$ is continuous at any $x_0\in\mathbb{R}$ in the above sense.
Furthermore, we say a
non-decreasing extended function $f$ defined on $[0,\infty)$ is continuous if there exists 
a non-decreasing extended function $\tilde{f}$ (defined on $\mathbb{R}$) such that $f(x)=\tilde{f}(x)$ for all $x\ge0$.
\end{definition}

% \begin{remark}
% If a non-decreasing function $f:[0,\infty)\longrightarrow[0,\infty]$ satisfies $f(0)=\infty$, then $f$ is extended continuous by the its definition.
% \end{remark}

\begin{proposition}[Properties of $L_{v_0}$]\label{Continuity of L}
For each fixed $v_0\in\mathcal{F}\dot{H}^\frac{1}{2}$,
the function $L_{v_0}$ is a non-decreasing continuous extended function defined on $[0,\infty)$.
% \begin{itemize}
% \item[(i)] $L_{v_0}:[0,\infty)\longrightarrow[0,\infty]$,
% \item[(ii)] $L_{v_0}$ is non-decreasing,
% \item[(iii)] $L_{v_0}$ is a extended continuous function.
% \end{itemize}
\end{proposition}

\begin{proof}
% (i) holds clearly.
% (ii) holds by a property of supremum.
It is clear that $L_{v_0}$ is a non-decreasing extended function defined on $[0,\infty)$.

We prove the continuity in the sense of Definition \ref{D:continuity}.
% By Proposition \ref{Control of L},
% If $\norm{u(0)}_{\mathcal{F}\dot{H}^{\frac12}}=0$ then $u(0)=0$ and
% the only possible solution to \eqref{NLS} is $(0,e^{\frac{1}{2}it\Delta}v_0)$. and
% \begin{align*}
% \|(0,e^{\frac{1}{2}it\Delta}v_0)\|_{W_1\times W_2}\leq c\|v_0\|_{\mathcal{F}\dot{H}^\frac{1}{2}}<\infty,
% \end{align*}
% that is, $L_{v_0}(0)<\infty$.
It is obvious that
\[
	L_{v_0}(0) = \|e^{\frac{1}{2}it\Delta}v_0\|_{W_2([0,\infty))}<\infty.
\]
The continuity of $L_{v_0}(\ell)$ at $\ell=0$ holds by Proposition \ref{Control of L}.

Fix $\ell_0\in(0,\infty)$ such that $L_{v_0}(\ell_0)<\infty$.
Let us prove right continuity of $L_{v_0}(\ell)$ at $\ell=\ell_0$.
Pick $\varepsilon>0$.
Take $\delta>0$ so that $\delta<\varepsilon_1$ and $c\delta<\varepsilon$, where $\varepsilon_1=\varepsilon_1(L_{v_0}(\ell_0))$ and $c=c(L_{v_0}(\ell_0))$ are the constants given in Proposition \ref{Long time perturbation} with the choice $M=L_{v_0}(\ell_0)$.
Fix $\ell\in(\ell_0,\ell_0+\delta)$.
Then, for any $u_{0,1}\in\mathcal{F}\dot{H}^\frac{1}{2}$ satisfying $\|u_{0,1}\|_{\mathcal{F}\dot{H}^\frac{1}{2}}\leq\ell$, 
the function %$u_{0,2}\in\mathcal{F}\dot{H}^\frac{1}{2}$ such that
\[
u_{0,2}=\frac{\ell_0}{\ell_0+\delta} u_{0,1}
\]
satisfies $\|u_{0,2}\|_{\mathcal{F}\dot{H}^\frac{1}{2}}\leq\ell_0$ and $\|u_{0,1}-u_{0,2}\|_{\mathcal{F}\dot{H}^\frac{1}{2}}\leq\delta$.
% Indeed, if we consider $u_{0,2}=\Bigl(1-\frac{\delta}{\|u_{0,1}\|_{\mathcal{F}\dot{H}^\frac{1}{2}}}\Bigr)u_{0,1}$, then
% \begin{align*}
% \|u_{0,2}\|_{\mathcal{F}\dot{H}^\frac{1}{2}}=\|u_{0,1}\|_{\mathcal{F}\dot{H}^\frac{1}{2}}-\delta\leq\ell-\delta<\ell_0\ \ \text{ and }\ \ \|u_{0,1}-u_{0,2}\|_{\mathcal{F}\dot{H}^\frac{1}{2}}=\delta.
% \end{align*}
Let $(u_1,v_1)$ and $(u_2,v_2)$ be two solutions to \eqref{NLS} with initial data $(u_{0,1},v_0)$ and $(u_{0,2},v_0)$, respectively.
Note that
\begin{align*}
\|(u_2,v_2)\|_{W_1([0,\infty))\times W_2([0,\infty))}\leq L_{v_0}(\ell_0)
\end{align*}
since $\|u_{0,2}\|_{\mathcal{F}\dot{H}^\frac{1}{2}}\leq\ell_0$.
Hence, we have
\begin{align*}
\|(u_1,v_1)-(u_2,v_2)\|_{W_1([0,\infty))\times W_2([0,\infty))}\leq c\delta<\varepsilon
\end{align*}
by Proposition \ref{Long time perturbation}.
Thus, it follows that
\begin{align*}
\|(u_1,v_1)\|_{W_1([0,\infty))\times W_2([0,\infty))}<\|(u_2,v_2)\|_{W_1([0,\infty))\times W_2([0,\infty))}+\varepsilon\leq L_{v_0}(\ell_0)+\varepsilon.
\end{align*}
Taking the supremum over such $u_{0,1} \in \mathcal{F}\dot{H}^\frac{1}{2}$,
%  such that $\|u_{0,1}\|_{\mathcal{F}\dot{H}^\frac{1}{2}}\leq\ell$ with fixed $\ell\in(\ell_0,\ell_0+\delta)$,
we obtain
\begin{align*}
L_{v_0}(\ell) \le L_{v_0}(\ell_0)+\varepsilon
\end{align*}
for $\ell \in (\ell_0 , \ell_0 + \delta)$.
This shows the right continuity of $L_{v_0}(\ell)$ at $\ell = \ell_0$ together with non-decreasing property.
The left continuity is shown in a similar way. We omit the details.
% Let us show the left continuity.
% For $\varepsilon>0$, take the same $\delta>0$ as above.
% Without loss of generality, we may suppose that $\delta < \ell_0$.
% Let $\ell \in (\ell_0-\delta,\ell_0)$.
% For any $u_{0,1}\in\mathcal{F}\dot{H}^\frac{1}{2}$ with $\|u_{0,1}\|_{\mathcal{F}\dot{H}^\frac{1}{2}}\leq\ell_0$,
% define
% \[
% u_{0,2}=\frac{\ell_0-\delta}{\ell_0} u_{0,1}
% \]
% Then, $\|u_{0,2}\|_{\mathcal{F}\dot{H}^\frac{1}{2}}\leq\ell$ and $\|u_{0,1}-u_{0,2}\|_{\mathcal{F}\dot{H}^\frac{1}{2}}\leq\delta$.
% Let $(u_1,v_1)$ and $(u_2,v_2)$ be two solutions to \eqref{NLS} with initial data $(u_{0,1},v_0)$ and $(u_{0,2},v_0)$, respectively.
% Since
% \begin{align*}
% \|(u_2,v_2)\|_{W_1\times W_2}\leq L_{v_0}(\ell)\leq L_{v_0}(\ell_0),
% \end{align*}
% by non-decreasing of $L_{v_0}$, we have
% \begin{align*}
% \|(u_1,v_1)-(u_2,v_2)\|_{W_1\times W_2}\leq c\delta<\varepsilon
% \end{align*}
% by Proposition \ref{Long time perturbation}.
% Thus, it follows that
% \begin{align*}
% \|(u_1,v_1)\|_{W_1\times W_2}<\|(u_2,v_2)\|_{W_1\times W_2}+\varepsilon\leq L_{v_0}(\ell)+\varepsilon.
% \end{align*}
% Taking supremum over $\|u_{0,1}\|_{\mathcal{F}\dot{H}^\frac{1}{2}}\leq\ell_0$, we obtain
% \begin{align*}
% L_{v_0}(\ell_0)-\varepsilon <L_{v_0}(\ell).
% \end{align*}
% Combining the above inequality, it follows that for any $\varepsilon>0$, there exists $\delta>0$ such that
% \begin{align*}
% |L_{v_0}(\ell)-L_{v_0}(\ell_0)|<\varepsilon
% \end{align*}
% for any $\ell$ with $|\ell-\ell_0|<\delta$.

Let us move on to the case $L_{v_0}(\ell_0)=\infty$.
We may suppose that
$\ell_0:=\inf\{\ell:L_{v_0}(\ell)=\infty\}$
otherwise continuity is trivial by definition. Under this assumption, we prove that 
$L_{v_0}(\ell)$ goes to infinity as $\ell\uparrow\ell_0$.
Assume that 
\[
	\displaystyle C_0:=\sup_{\ell<\ell_0}L_{v_0}(\ell)<\infty
\]
 for contradiction.
Let $\varepsilon_1=\varepsilon_1(C_0)$ be the constant given in Proposition \ref{Long time perturbation}.
Fix $0<\varepsilon<1$ so that $\varepsilon\ell_0<\varepsilon_1$.
Then, for any fixed $u_{0,1}\in\mathcal{F}\dot{H}^\frac{1}{2}$ with $\|u_{0,1}\|_{\mathcal{F}\dot{H}^\frac{1}{2}}\leq\ell_0$,
the function
\[
	u_{0,2}:=(1-\varepsilon)u_{0,1}
\]
satisfies
$\|u_{0,2}\|_{\mathcal{F}\dot{H}^\frac{1}{2}}\leq(1-\varepsilon)\ell_0$.
Let $(u_1,v_1)$ and $(u_2,v_2)$ be two solutions to \eqref{NLS} with initial data $(u_{0,1},v_0)$ and $(u_{0,2},v_0)$, respectively.
One sees that
\begin{align*}
	\|(u_2,v_2)\|_{W_1([0,\infty))\times W_2([0,\infty))}
		\leq L_{v_0}\left((1-\varepsilon)\ell_0\right)
		\leq C_0
		<\infty.
\end{align*}
In addition, we have
\begin{align*}
	\|u_{0,1}-u_{0,2}\|_{\mathcal{F}\dot{H}^\frac{1}{2}}
		=\varepsilon\|u_{0,1}\|_{\mathcal{F}\dot{H}^\frac{1}{2}}
		\leq \varepsilon\ell_0.
\end{align*}
Applying Proposition \ref{Long time perturbation}, we obtain
\begin{align*}
	\|(u_1,v_1)\|_{W_1([0,\infty))\times W_2([0,\infty))}
		\leq \|(u_2,v_2)\|_{W_1([0,\infty))\times W_2([0,\infty))}+c\varepsilon \ell_0
		\leq L_{v_0}\left((1-\varepsilon)\ell_0\right)+c\varepsilon \ell_0
		<\infty,
\end{align*}
where $c=c(C_0)$ is a constant.
Taking supremum over $u_{0,1}$, it follows that
\begin{align*}
L_{v_0}(\ell_0)\leq L_{v_0}\left((1-\varepsilon)\ell_0\right)+c\varepsilon\ell_0<\infty.
\end{align*}
This is a contradiction. %Therefore, we obtain extended continuity of $L_{v_0}$.
\end{proof}

By using the non-decreasing property of $L_{v_0}$, we have the following:

\begin{proposition}[Another characterization of $\ell_{v_0}^\dagger$]\label{Another characterization}
The following identity holds.
\begin{align*}
\ell_{v_0}^\dagger=\inf\{\ell:L_{v_0}(\ell)=\infty\}
\end{align*}
for any $v_0\in\mathcal{F}\dot{H}^\frac{1}{2}$.
\end{proposition}

\begin{proof}
% We set $\widetilde{\ell}_{v_0}^\dagger:=\inf\{\ell:L_{v_0}(\ell)=\infty\}$. 
When $L_{v_0}(\ell)$ is finite for any $\ell>0$, we see that the both sides are infinite.
Otherwise, the two sets $\{\ell:L_{v_0}(\ell)<\infty\}$ and $\{\ell:L_{v_0}(\ell)=\infty\}$
give us a Dedekind cut of a totally ordered set $[0,\infty)$, by means of Propositions \ref{Control of L} and Proposition \ref{Continuity of L}.
\end{proof}

A consequence of the alternative characterization is that
\begin{equation}\label{E:Lvellv}
	L_{v_0}(\ell_{v_0}^\dagger) = \infty
\end{equation}
holds for any $v_0\in\mathcal{F}\dot{H}^\frac{1}{2}$. This follows from the continuity of $L_{v_0}$.
We also have the following:

\begin{lemma}\label{Comparison of ell and ell^dagger}
$\ell_{v_0} \ge \ell_{v_0}^\dagger$ for any $v_0 \in \mathcal{F} \dot{H}^\frac{1}{2}$.
\end{lemma}

\begin{proof}
If $\ell_{v_0}=\infty$, then Lemma \ref{Comparison of ell and ell^dagger} holds.
Let $\ell_{v_0}<\infty$.
By the definition of $\ell_{v_0}$, for any $\varepsilon>0$, there exists $u_0\in \mathcal{F}\dot{H}^\frac{1}{2}$ such that 
\begin{align*}
\|u_0\|_{\mathcal{F}\dot{H}^\frac{1}{2}}<\ell_{v_0}+\varepsilon
\end{align*}
holds and 
the corresponding solution $(u(t),v(t))$ with data $(u(0),v(0))=(u_0,v_0)$
does not scatter forward in time.
Since Proposition \ref{Scattering criterion} deduces $\|(u,v)\|_{W_1([0,T_{\max}))\times W_2([0,T_{\max}))}=\infty$ from
the failure of scattering, we obtain
\begin{align*}
L_{v_0}(\ell_{v_0}+\varepsilon)=\infty.
\end{align*}
This implies the relation $\ell_{v_0}^\dagger\leq\ell_{v_0}+\varepsilon$,
thanks to Proposition \ref{Another characterization}. Since $\varepsilon>0$ is arbitrary, we have the desired conclusion.
%Lemma \ref{Comparison of ell and ell^dagger}.
\end{proof}

The following is one of the key property to prove Theorem \ref{T:l0}.

\begin{proposition}\label{P:l0lv0}
$\ell_0^\dagger\geq\ell_{v_0}^\dagger$
for any $v_0\in\mathcal{F}\dot{H}^\frac{1}{2}$.
\end{proposition}

\begin{proof}
Fix $v_0\in\mathcal{F}\dot{H}^\frac{1}{2}$. We assume that $\ell_0^\dagger<\ell_{v_0}^\dagger$ for contradiction.
Then, we have
\[
	L_0 \Bigl(\frac{\ell_0^\dagger+\ell_{v_0}^\dagger}{2}\Bigr)=\infty ,\quad L_{v_0} \Bigl(\frac{\ell_0^\dagger+\ell_{v_0}^\dagger}{2}\Bigr)<\infty .
\]
Using the fact that $L_0 (\frac{\ell_0^\dagger+\ell_{v_0}^\dagger}{2})=\infty $ and the scaling argument,
one can take data $\{(U_{0,n},0)\}$ so that the corresponding solution $(U_n(t),V_n(t))$ to \eqref{NLS} satisfies
\begin{align}
\left\|U_{0,n}\right\|_{\mathcal{F}\dot{H}^\frac{1}{2}}\leq\frac{\ell_0^\dagger+\ell_{v_0}^\dagger}{2} \label{037}
\end{align}
and
\begin{equation}\label{E:l0pf1}
\left\|(U_n,V_n)\right\|_{W_1([0,n^{-1}])\times W_2([0,n^{-1}])}\geq n
\end{equation}
for all $n\ge1$.
Let $(u_{n},v_{n})$ be another solution  to \eqref{NLS} with the initial data $(U_{0,n},v_0)$. 
Since $L_{v_0} (\frac{\ell_0^\dagger+\ell_{v_0}^\dagger}{2})<\infty $,
one sees from \eqref{037} that $(u_{n},v_{n})$ is global in time and
\begin{align*}
\left\|(u_{n},v_{n})\right\|_{W_1([0,\infty))\times W_2([0,\infty))}\leq L_{v_0}\Bigl(\frac{\ell_0^\dagger+\ell_{v_0}^\dagger}{2}\Bigr)<\infty.
\end{align*}
We now set $(\widetilde{u}_{n},\widetilde{v}_{n})=(u_{n},v_{n})-(0,e^{\frac{1}{2}it\Delta}v_0)$. Then, $(\widetilde{u}_{n},\widetilde{v}_{n})$ solves
\begin{equation*}
\left\{
\begin{aligned}
&i\partial_t\widetilde{u}_{n}+\Delta\widetilde{u}_{n}+2\widetilde{v}_{n}\overline{\widetilde{u}_{n}}=-2(e^{\frac{1}{2}it\Delta}v_0)\overline{u_{n}},\\
&i\partial_t\widetilde{v}_{n}+\frac{1}{2}\Delta\widetilde{v}_{n}+\widetilde{u}_{n}^2=0,\\
&(\widetilde{u}_{n}(0),\widetilde{v}_{n}(0))= (U_{0,n},0)
\end{aligned}
\right.
\end{equation*}
and so it is an approximate solution to \eqref{NLS} with an error
\begin{align*}
	e_1 = -2(e^{\frac{1}{2}it\Delta}v_0)\overline{u_{n}},\quad e_2 = 0.
\end{align*}
Take $\tau>0$ and set $I=[0,\tau]$. We have
\begin{align*}
	\norm{(e_1,e_2)}_{N_1(I)\times N_2(I)} \lesssim \norm{e^{\frac{1}{2}it\Delta}v_0}_{W_2(I)} \norm{u_n}_{W_1(I)} \le
	\norm{e^{\frac{1}{2}it\Delta}v_0}_{W_2(I)} L_{v_0}\Bigl(\frac{\ell_0^\dagger+\ell_{v_0}^\dagger}{2}\Bigr).
\end{align*}
The right hand side is independent of $n$, and tends to zero as $\tau \downarrow 0$.

Now we apply the Proposition \ref{Long time perturbation} with $M=L_{v_0}(\frac{\ell_0^\dagger+\ell_{v_0}^\dagger}{2})+\|e^{\frac{1}{2}it\Delta}v_0\|_{W_2([0,\infty))}$.
Choose $\tau$ sufficiently small so that the above upper bound of the error becomes smaller than the corresponding $\varepsilon_1$.
Since $(U_n,V_n)$ is a solution with the same initial data as $(\widetilde{u}_{n},\widetilde{v}_{n})$,
we see from Proposition \ref{Long time perturbation} that $(U_n,V_n)$ extends up to time $\tau$ and obeys the bound
\begin{align*}
	\norm{(U_n,V_n)}_{W_1(I)\times W_2(I)}
	\le \norm{(\widetilde{u}_n,\widetilde{v}_n)}_{W_1(I)\times W_2(I)} + C\varepsilon_1
	\le L_{v_0}\Bigl(\frac{\ell_0^\dagger+\ell_{v_0}^\dagger}{2}\Bigr) + \|e^{\frac{1}{2}it\Delta}v_0\|_{W_2([0,\infty))} + C\varepsilon_1.
\end{align*}
However, this contradicts with \eqref{E:l0pf1} for large $n$.
\end{proof}

%Linear profile decomposition

\section{Linear profile decomposition}\label{Sec:Linear profile decomposition}

In this section, we obtain a linear profile decomposition (Theorem \ref{Linear profile decomposition}).

% \subsection{Deformations}

\subsection{Linear profile decomposition} 

% In this subsection, we introduce several operators.

Let us first introduce several operators and give a notion of deformation, which is a 
specific class of bounded operator.

\begin{definition}[Operators]
We define the following operators.
\begin{itemize}
\item[(1)] A dilation
\begin{align*}
(D(h)(f,g))(x)=(f_{\{h\}},g_{\{h\}})=(h^2f(hx),h^2g(hx)),\ \ \ h\in2^{\mathbb{Z}},
\end{align*}
\item[(2)] A translation in Fourier space
\begin{align*}
(T(\xi)(f,g))(x)=(e^{ix\cdot\xi}f(x),e^{2ix\cdot\xi}g(x)),\ \ \ \xi\in\mathbb{R}^3.
\end{align*}
\end{itemize}
\end{definition}

\begin{definition}
We say that a bounded operator
\begin{align*}
\mathcal{G}=(\mathcal{G}_{1},\mathcal{G}_{2})=T(\xi)D(h),\ \ \ (\xi,h)\in\mathbb{R}^3\!\times\!2^{\mathbb{Z}}
\end{align*}
on $\mathcal{F}\dot{H}^\frac{1}{2}\!\times\!\mathcal{F}\dot{H}^\frac{1}{2}$
is called a deformation in $\mathcal{F}\dot{H}^\frac{1}{2}\!\times\!\mathcal{F}\dot{H}^\frac{1}{2}$.
Let a set $G\subset\mathcal{L}(\mathcal{F}\dot{H}^\frac{1}{2}\!\times\!\mathcal{F}\dot{H}^\frac{1}{2})$ be composed of all deformations.
\end{definition}

\begin{remark}
$G$ is a group with the functional composition as a binary operation.
Id\,$=T(0)D(1)\in G$ is the identity element.
For $\mathcal{G}=T(\xi)D(h)$, the inverse element is $\mathcal{G}^{-1}=T(-\frac{\xi}{h})D(\frac{1}{h})\in G$.
\end{remark}

%We check that $G$ forms a group. For $\mathcal{G}^1$, $\mathcal{G}^2\in G$,
%\begin{align*}
%\mathcal{G}^1\mathcal{G}^2=T(\xi_1)D(h_1)T(\xi_2)D(h_2)=T(\xi_1)T(h_1\xi_2)D(h_1h_2)=T(\xi_1+h_1\xi_2)D(h_1h_2)
%\end{align*}
%implies $\mathcal{G}^1\mathcal{G}^2\in G$. For $\mathcal{G}^1$, $\mathcal{G}^2$, $\mathcal{G}^3\in G$,
%\begin{align*}
%(\mathcal{G}^1\mathcal{G}^2)\mathcal{G}^3=T(\xi_1+h_1\xi_2)D(h_1h_2)T(\xi_3)D(h_3)=T(\xi_1+h_1\xi_2+h_1h_2\xi_3)D(h_1h_2h_3)
%\end{align*}
%and
%\begin{align*}
%\mathcal{G}^1(\mathcal{G}^2\mathcal{G}^3)=T(\xi_1)D(h_1)T(\xi_2+h_2\xi_3)D(h_2h_3)=T(\xi_1+h_1\xi_2+h_1h_2\xi_3)D(h_1h_2h_3).
%\end{align*}
%Thus, we have
%\begin{align*}
%(\mathcal{G}^1\mathcal{G}^2)\mathcal{G}^3=\mathcal{G}^1(\mathcal{G}^2\mathcal{G}^3).
%\end{align*}
%$\text{Id}=T(0)D(1)\in G$ satisfies $\text{Id}\,\mathcal{G}=\mathcal{G}\,\text{Id}=\mathcal{G}$ for any $\mathcal{G}\in G$.
%For any $\mathcal{G}=T(\xi)D(h)$, $\mathcal{G}^{-1}=T\left(-\frac{\xi}{h}\right)D\left(\frac{1}{h}\right)\in G$ satisfies $\mathcal{G}\mathcal{G}^{-1}=\mathcal{G}^{-1}\mathcal{G}=\text{Id}$.
%Therefore, $G$ forms a group.

% \subsection{Orthogonality of families of deformations}\label{Orthogonality of families of deformations}

% In this subsection, collecting properties given in \S\ref{Orthogonality of families of deformations}, we obtain Linear profile decomposition (Theorem \ref{Linear profile decomposition}).

% In this subsection, we introduce several properties for families of deformation.

Next, we introduce a class of families of deformations.

\begin{definition}[A vanishing family]
We say that a family of deformations $\{\mathcal{G}_n=T(\xi_n)D(h_n)\}_n\subset G$ is vanishing if
$
|\xi_n|+|\log h_n|\longrightarrow\infty\ \text{ as }\ n\rightarrow\infty
$
holds.
\end{definition}

\begin{lemma}\label{Symmetry vanishing}
A family $\{\mathcal{G}_n\}_n\subset G$ is vanishing if and only if a family of inverse elements $\{\mathcal{G}_n^{-1}\}_n$ is vanishing.
\end{lemma}

\begin{proof}
It is clear from $(T(\xi_n)D(h_n))^{-1}=T(-\frac{\xi_n}{h_n})D(\frac{1}{h_n})$.
\end{proof}

The following characterization of the vanishing family is useful.

\begin{proposition}\label{Equivalent vanishing 2}
For a family $\{\mathcal{G}_n\}_n\subset G$ of deformations, the following three statements are equivalent.
\begin{itemize}
\item[(1)] $\{\mathcal{G}_n\}_n$ is vanishing.
\item[(2)] For any $(\phi,\psi)\in\mathcal{F}\dot{H}^\frac{1}{2}\!\times\!\mathcal{F}\dot{H}^\frac{1}{2}$, $\mathcal{G}_n(\phi,\psi)\xrightharpoonup[]{\hspace{0.4cm}}(0,0)$ weakly in $\mathcal{F}\dot{H}^\frac{1}{2}\!\times\!\mathcal{F}\dot{H}^\frac{1}{2}$ as $n\rightarrow\infty$.
\item[(3)] For any subsequence $n_k$ of $n$, there exists a sequence $\{(f_k,g_k)\}_k\subset\mathcal{F}\dot{H}^\frac{1}{2}\!\times\!\mathcal{F}\dot{H}^\frac{1}{2}$ and a subsequence $k_l$ of $k$ such that $(f_{k_l},g_{k_l})\xrightharpoonup[]{\hspace{0.4cm}}(0,0)$ and $\mathcal{G}_{n_{k_l}}^{-1}(f_{k_l},g_{k_l})\xrightharpoonup[]{\hspace{0.4cm}}(\phi,\psi)\neq(0,0)$ weakly in $\mathcal{F}\dot{H}^\frac{1}{2}\!\times\!\mathcal{F}\dot{H}^\frac{1}{2}$ as $l\rightarrow\infty$.
\end{itemize}
\end{proposition}

For the proof, see \cite{Mas16,MasSeg18}.

We now introduce a notion of orthogonality.

\begin{definition}[Orthogonality]
We say two families of deformations
 $\{\mathcal{G}_n\}$, $\{\widetilde{\mathcal{G}}_n\}\subset G$ are orthogonal
if $\{\mathcal{G}_n^{-1}\widetilde{\mathcal{G}}_n\}$ is vanishing.
%  then we say that $\{\mathcal{G}_n\}$ and $\{\widetilde{\mathcal{G}}_n\}$ .
\end{definition}

\begin{remark}
Let $\{\mathcal{G}_n^j=T(\xi_n^j)D(h_n^j)\}$ $\subset G$ $(j=1,2)$ be two families of deformations.
$\{\mathcal{G}_n^1\}$ and $\{\mathcal{G}_n^2\}$ are orthogonal if and only if
\begin{align*}
\frac{h_n^1}{h_n^2}+\frac{h_n^2}{h_n^1}+\frac{|\xi_n^1-\xi_n^2|}{h_n^1}\longrightarrow\infty
\end{align*}
as $n\rightarrow\infty$.
This equivalence holds from the identity
\begin{align*}
(\mathcal{G}_n^1)^{-1}\mathcal{G}_n^2
	%=T\left(-\frac{\xi_n^1}{h_n^1}\right)D\left(\frac{1}{h_n^1}\right)T(\xi_n^2)D(h_n^2)
	=T\left(\frac{\xi_n^2-\xi_n^1}{h_n^1}\right)D\left(\frac{h_n^2}{h_n^1}\right).
\end{align*}
\end{remark}

\begin{proposition}
Let $\{\mathcal{G}_n\}, \{\widetilde{\mathcal{G}}_n\}\subset G$.
Define the relation $\sim$ as follows:
If $\{\mathcal{G}_n\}$ and $\{\widetilde{\mathcal{G}}_n\}$ are not orthogonal then $\{\mathcal{G}_n\}\sim\{\widetilde{\mathcal{G}}_n\}$.
Then, $\sim$ is an equivalent relation.
\end{proposition}

For the proof, see \cite{Mas16,MasSeg18}.

Let us now state the linear profile decomposition result.

\begin{theorem}[Linear profile decomposition]\label{Linear profile decomposition}
Let $\{(\phi_n,\psi_n)\}$ be a bounded sequence in $\mathcal{F}\dot{H}^\frac{1}{2}\times\mathcal{F}\dot{H}^\frac{1}{2}$.
Passing to a sequence if necessary, there exist profile $\{(\phi^j,\psi^j)\}\subset \mathcal{F}\dot{H}^\frac{1}{2}\times\mathcal{F}\dot{H}^\frac{1}{2}$, $\{(\Phi_n^J,\Psi_n^J)\}\subset \mathcal{F}\dot{H}^\frac{1}{2}\times\mathcal{F}\dot{H}^\frac{1}{2}$, and pairwise orthogonal families of deformations $\{\mathcal{G}_n^j=T(\xi_n^j)D(h_n^j)\}_n\subset G$ $(j=1,2,\cdots)$ such that for each $J\geq 1$,
\begin{align*}
(\phi_n,\psi_n)=\sum_{j=1}^J\mathcal{G}_n^j(\phi^j,\psi^j)+(\Phi_n^J,\Psi_n^J)
\end{align*}
for any $n\geq1$.
Moreover, $\{(\Phi_n^J,\Psi_n^J)\}$ satisfies
\begin{equation}
\notag (\mathcal{G}_n^j)^{-1}(\Phi_n^J,\Psi_n^J)\xrightharpoonup[]{\hspace{0.4cm}}
\begin{cases}
&\hspace{-0.4cm}(\phi^j,\psi^j)\hspace{0.55cm}(J<j),\\
&\hspace{-0.2cm}(0,0)\hspace{0.75cm}(J\geq j)
\end{cases}
\end{equation}
in $\mathcal{F}\dot{H}^\frac{1}{2}\!\times\!\mathcal{F}\dot{H}^\frac{1}{2}$ as $n\rightarrow\infty$ for any $j\geq0$, where we use the convention $(\Phi_n^0,\Psi_n^0)=(\phi_n,\psi_n)$, and
\begin{align}
\limsup_{n\rightarrow\infty}\|(e^{it\Delta}\Phi_n^J,e^{\frac{1}{2}it\Delta}\Psi_n^J)\|_{L_t^{q,\infty}L_x^r\times L_t^{q,\infty}L_x^r}\longrightarrow0 \label{022}
\end{align}
as $J\rightarrow \infty$ for any $1<q,r<\infty$ such that $\frac{1}{q}\in(\frac{1}{2},1)$ and $\frac{2}{q}+\frac{3}{r}=2$.
Furthermore, we have Pythagorean decomposition:
\begin{align*}
\|\phi_n\|_{\mathcal{F}\dot{H}^\frac{1}{2}}^2=\sum_{j=1}^J\|\phi^j\|_{\mathcal{F}\dot{H}^\frac{1}{2}}^2+\|\Phi_n^J\|_{\mathcal{F}\dot{H}^\frac{1}{2}}^2+o_n(1),
\end{align*}
\abovedisplayskip=0cm
\begin{align*}
\|\psi_n\|_{\mathcal{F}\dot{H}^\frac{1}{2}}^2=\sum_{j=1}^J\|\psi^j\|_{\mathcal{F}\dot{H}^\frac{1}{2}}^2+\|\Psi_n^J\|_{\mathcal{F}\dot{H}^\frac{1}{2}}^2+o_n(1),
\end{align*}
where $o_n(1)$ goes to $0$ as $n\rightarrow\infty$.
\end{theorem}

\begin{proof}
We define
\begin{equation*}
\nu(\{(\phi_n,\psi_n)\}):=\left\{(\phi,\psi)\in\mathcal{F}\dot{H}^\frac{1}{2}\times\mathcal{F}\dot{H}^\frac{1}{2}\left|
\begin{array}{l}
\text{There exist }\xi_n\in\mathbb{R}^3\text{ and }h_n\in2^\mathbb{Z}\text{ such that}\\
(\mathcal{G}_n^j)^{-1}(\phi_n,\psi_n)\xrightharpoonup[]{\hspace{0.4cm}}(\phi,\psi)\text{ in }\mathcal{F}\dot{H}^\frac{1}{2}\times\mathcal{F}\dot{H}^\frac{1}{2}\\
\text{ as }n\rightarrow\infty,\text{ up to subsequence.}
\end{array}
\right.\right\}.
\end{equation*}
and
\begin{align*}
\eta(\{(\phi_n,\psi_n)\}):=\sup_{(\phi,\psi)\in\nu(\{(\phi_n,\psi_n)\})}\|(\phi,\psi)\|_{\mathcal{F}\dot{H}^\frac{1}{2}\times\mathcal{F}\dot{H}^\frac{1}{2}}.
\end{align*}
Then, a standard argument shows the theorem.
However, the smallness \eqref{022} is replaced by
$\eta(\{(\Phi_n^J,\Psi_n^J)\}) \to 0$
as $J\to\infty$.
The following Proposition \ref{P:cv} shows that this smallness is stronger.
\end{proof}

\subsection{Control of vanishing}

To complete the proof of Theorem \ref{Linear profile decomposition}, we show the following in this subsection.

\begin{proposition}[Control of vanishing]\label{P:cv}
If a sequence $\{(\Phi_n,\Psi_n)\}_n \subset \mathcal{F}\dot{H}^{\frac12} \times \mathcal{F}\dot{H}^{\frac12}$ satisfies
\begin{align*}
	\norm{(\Phi_n,\Psi_n)}_{\mathcal{F}\dot{H}^{\frac12} \times \mathcal{F}\dot{H}^{\frac12}} \le M
\end{align*}
and
\begin{align*}
	\|(e^{it\Delta}\Phi_n,e^{\frac{1}{2}it\Delta}\Psi_n)\|_{L_t^{q,\infty}L_x^r\times L_t^{q,\infty}L_x^r}\ge \varepsilon_0
\end{align*}
for some $M>0$, $\varepsilon_0>0$, and
$1<q,r<\infty$ with $\frac{1}{q}\in(\frac{1}{2},1)$ and $\frac{2}{q}+\frac{3}{r}=2$, then
\begin{align*}
	\eta(\{(\Phi_n,\Psi_n)\}) \gtrsim_{M,\varepsilon_0,q,r} 1.
\end{align*}
\end{proposition}
To prove the proposition, we will need the following.
\begin{lemma}[Improved Strichartz estimate]\label{L:iStr}
It holds that
\begin{align*}
	\norm{e^{it\Delta}f}_{L^3_t([0,\infty);L_x^3)} \lesssim \norm{f}_{\mathcal{F} \dot{H}^{\frac16}}^{\frac23} \sup_{N\in 2^\mathbb{Z}} \left( \norm{e^{it \Delta} \psi_N f}_{L^3_t([0,\infty);L_x^3)} \right)^{\frac13},
\end{align*}
where $\psi_N$ is defined as \eqref{069}.
\end{lemma}
\begin{proof}
% The proof boils down to the case $\theta=3$. Indeed, it suffices to use the embedding $L^{3,3}_t \hookrightarrow L^{3,\theta}_t$ if $\theta>3$ and
% to interpolate it with $\norm{e^{it\Delta}f}_{L^{3,2}L^{3}_{x}([0,\infty))} \lesssim \norm{f}_{\mathcal{F} \dot{H}^{\frac16}}$ if $\theta<3$.
%
By Lemma \ref{Square function estimate}, we have
\begin{align*}
	\norm{e^{it\Delta}f}_{L^{3}_x}
	= \norm{\mathcal{M}_{\frac{1}{2}}(-t)e^{it\Delta}f}_{L^{3}_x}
	\sim \Bigl\|\Bigl(\sum_{N\in 2^\mathbb{Z}}\left|P_\frac{N}{2t} \mathcal{M}_\frac{1}{2}(-t)e^{it\Delta}f\right|^2\Bigr)^{1/2}\Bigr\|_{L^{3}_x}
\end{align*}
for $t>0$, where the implicit constant is independent of $t$ by virtue of the scaling.
%\textcolor{cyan}{
%Since
%\begin{align*}
%P_\frac{N}{2t}[g(x,t)]
%	&=\mathcal{F}^{-1}\psi_\frac{N}{2t}\mathcal{F}[g(x,t)]\\
%	&=(2\pi)^{-\frac{3}{2}}\int_{\mathbb{R}^3}e^{ix\cdot\xi}\psi_\frac{N}{2t}(\xi)\mathcal{F}[g(x,t)](\xi)d\xi\\
%	&=(2t)^{-3}(2\pi)^{-\frac{3}{2}}\int_{\mathbb{R}^3}e^{ix\cdot\frac{\xi}{2t}}\psi_N(\xi)\mathcal{F}[g(x,t)]\left(\frac{\xi}{2t}\right)d\xi\ \ \left(\because\ \xi\rightarrow\frac{\xi}{2t}\right)\\
%	&=(2t)^{-3}\mathcal{F}^{-1}\left[\psi_N\mathcal{F}[g(x,t)]\left(\frac{\cdot}{2t}\right)\right]\left(\frac{x}{2t}\right)\\
%	&=(2t)^{-3}\mathcal{F}^{-1}\left[\psi_N(2\pi)^{-\frac{3}{2}}\int_{\mathbb{R}^3}e^{-ix\cdot\frac{\xi}{2t}}g(x,t)dx\right]\left(\frac{x}{2t}\right)\\
%	&=\mathcal{F}^{-1}\left[\psi_N(2\pi)^{-\frac{3}{2}}\int_{\mathbb{R}^3}e^{-ix\cdot\xi}g(2tx,t)dx\right]\left(\frac{x}{2t}\right)\ \ (\because\ x\rightarrow2tx)\\
%	&=\mathcal{F}^{-1}\left[\psi_N\mathcal{F}[g(2tx,t)]\right]\left(\frac{x}{2t}\right)\\
%	&=P_N[g(2tx,t)]\left(\frac{x}{2t}\right),
%\end{align*}
%we have
%\begin{align*}
%P_N[g(2tx,t)]=P_\frac{N}{2t}\left[g\left(x,t\right)\right](2tx).
%\end{align*}
%Applying this identity, we get
%\begin{align*}
%\|g(x,t)\|_{L_x^3}^3
%	&=(2t)^3\|g(2tx,t)\|_{L_x^3}^3\sim(2t)^3\left\|\left(\sum_{N\in2^\mathbb{Z}}\left|P_N[g(2tx,t)]\right|^2\right)^\frac{1}{2}\right\|_{L_x^3}^3\\
%	&=(2t)^3\left\|\left(\sum_{N\in2^\mathbb{Z}}\left|P_\frac{N}{2t}[g(x,t)](2tx)\right|^2\right)^\frac{1}{2}\right\|_{L_x^3}^3=\left\|\left(\sum_{N\in2^\mathbb{Z}}\left|P_\frac{N}{2t}[g(x,t)]\right|^2\right)^\frac{1}{2}\right\|_{L_x^3}^3.
%\end{align*}
%}
Denote $v_N = P_\frac{N}{2t} \mathcal{M}_\frac{1}{2}(-t)e^{it\Delta}f$ for simplicity. 
By a convexity argument, one has
\begin{align*}
	\Bigl\|\Bigl(\sum_{N\in 2^\mathbb{Z}}|v_N|^2\Bigr)^\frac{1}{2}\Bigr\|_{L^3_t([0,\infty);L_x^3)}^3
	&{}= \int_{[0,\infty) \times \mathbb{R}^3} \Bigl(\sum_{N\in 2^\mathbb{Z}}|v_N|^2\Bigr)^{\frac{3}{4}}
	\Bigl(\sum_{M\in 2^\mathbb{Z}}|v_M|^2\Bigr)^{\frac{3}{4}}dx\, dt\\
	&{}\lesssim {\sum_{M, N\in 2^\mathbb{Z}, N \le M}} \int_{[0,\infty)\times \mathbb{R}^3} |v_N|^{\frac32} |v_M|^{\frac32} dx\, dt,
\end{align*}
where we have used the symmetry in the last line to reduce the matter to the case $N\le M$.
Take $r_1$ and $r_2$ so that $\frac{8}{3} < r_1<3 < r_2<\frac{10}{3}$ and $\frac23=\frac1{r_1} + \frac1{r_2}$. By the H\"older inequality,
\begin{align*}
	\int_{[0,\infty) \times \mathbb{R}^3} |v_N|^{\frac32} |v_M|^{\frac32} dx\, dt
	\le \norm{v_N}_{L_t^{r_1}([0,\infty);L_x^{r_1})} \norm{v_N}_{L_t^3([0,\infty);L_x^3)}^\frac12 
	\norm{v_M}_{L_t^3([0,\infty);L_x^3)}^\frac12 \norm{v_M}_{L_t^{r_2}([0,\infty);L_x^{r_2})}.
\end{align*}
Hence,
\begin{align*}
	\norm{e^{it\Delta}f}_{L_t^3([0,\infty);L_x^3)}^3 &{}\lesssim 
	\Bigl(\sup_{N\in 2^\mathbb{Z}} \norm{v_N}_{L_t^3([0,\infty);L_x^3)}\Bigr) \sum_{M, N\in 2^\mathbb{Z}, N \le M} \norm{v_N}_{L_t^{r_1}([0,\infty);L_x^{r_1})}  \norm{v_M}_{L_t^{r_2}([0,\infty);L_x^{r_2})}
\end{align*}
Remark that
\begin{align*}
	v_N = \mathcal{F} \psi_\frac{N}{2t} \mathcal{F}^{-1} D(t) \mathcal{F}  \mathcal{M}_\frac{1}{2}(t) f
	= D(t) \mathcal{F} \mathcal{M}_\frac{1}{2}(t) \psi_N  f = \mathcal{M}_\frac{1}{2}(t)^{-1} e^{it\Delta} \psi_N f.
\end{align*}
%\textcolor{cyan}{
%Second identity:
%\begin{align*}
%\mathcal{F}\psi_{\frac{N}{2t}}\mathcal{F}^{-1}D(t)\mathcal{F}\mathcal{M}_{\frac{1}{2}}(t)f
%	&=\mathcal{F}\psi_\frac{N}{2t}\mathcal{F}^{-1}(2it)^{-\frac{3}{2}}\mathcal{F}\left[\mathcal{M}_\frac{1}{2}(t)f\right]\left(\frac{\xi}{2t}\right)\\
%	&=\mathcal{F}\psi_\frac{N}{2t}\mathcal{F}^{-1}(2it)^{-\frac{3}{2}}(2\pi)^{-\frac{3}{2}}\int_{\mathbb{R}^3}e^{-ix\cdot\frac{\xi}{2t}}e^{\frac{i|x|^2}{4t}}f(x)dx\\
%	&=\mathcal{F}\psi_\frac{N}{2t}\mathcal{F}^{-1}\left(\frac{i}{2t}\right)^{-\frac{3}{2}}(2\pi)^{-\frac{3}{2}}\int_{\mathbb{R}^3}e^{-ix\cdot\xi}e^{it|x|^2}f(2tx)dx\ \ (\because\ x\rightarrow2tx)\\
%	&=\mathcal{F}\psi_\frac{N}{2t}\mathcal{F}^{-1}\mathcal{F}\left[\left(\frac{i}{2t}\right)^{-\frac{3}{2}}e^{it|\cdot|^2}f(2t\cdot)\right]\\
%	&=\left(\frac{i}{2t}\right)^{-\frac{3}{2}}\mathcal{F}\psi_\frac{N}{2t}e^{it|\cdot|^2}f(2t\cdot)\\
%	&=\left(\frac{i}{2t}\right)^{-\frac{3}{2}}(2\pi)^{-\frac{3}{2}}\int_{\mathbb{R}^3}e^{-ix\cdot\xi}\psi_\frac{N}{2t}(x)e^{it|x|^2}f(2tx)dx\\
%	&=(2it)^{-\frac{3}{2}}(2\pi)^{-\frac{3}{2}}\int_{\mathbb{R}^3}e^{-ix\cdot\frac{\xi}{2t}}\psi_N(x)e^{\frac{i|x|^2}{4t}}f(x)dx\ \ \left(\because\ x\rightarrow\frac{x}{2t}\right)\\
%	&=(2it)^{-\frac{3}{2}}\mathcal{F}\left[\psi_N\mathcal{M}_\frac{1}{2}(t)f\right]\left(\frac{\xi}{2t}\right)\\
%	&=D(t)\mathcal{F}\mathcal{M}_\frac{1}{2}(t)\psi_Nf.
%\end{align*}
%}
By the Strichartz estimate,
\begin{align*}
\norm{v_N}_{L_t^r([0,\infty);L_x^r)} 
	&=\|\mathcal{M}_\frac{1}{2}(-t)e^{it\Delta}\psi_Nf\|_{L_t^r([0,\infty);L_x^r)}\lesssim\||\nabla|^\frac{10-3r}{2r}\mathcal{M}_\frac{1}{2}(-t)e^{it\Delta}\psi_Nf\|_{L_t^r([0,\infty);L_x^\frac{6r}{16-3r})}\\
	&\lesssim \||t|^{-\frac{10-3r}{2r}}\|_{L_t^{\frac{2r}{10-3r},\infty}}\||t|^\frac{10-3r}{2r}|\nabla|^\frac{10-3r}{2r}\mathcal{M}_\frac{1}{2}(-t)e^{it\Delta}\psi_Nf\|_{L_t^{\frac{2r}{3r-8},r}([0,\infty);L_x^\frac{6r}{16-3r})}\\
	&\lesssim \norm{e^{it\Delta} \psi_N f}_{L^{\frac{2r}{3r-8},r}_t([0,\infty);\dot{X}_\frac{1}{2}^{\frac{10-3r}{2r},\frac{6r}{16-3r}})}\lesssim \norm{e^{it\Delta} \psi_N f}_{L^{\frac{2r}{3r-8},2}_t([0,\infty);\dot{X}_\frac{1}{2}^{\frac{10-3r}{2r},\frac{6r}{16-3r}})}\\
	&\lesssim \|\psi_Nf\|_{\mathcal{F}\dot{H}^\frac{10-3r}{2r}}%=\||x|^\frac{10-3r}{2r}\psi_Nf\|_{L_x^2}
	=\||x|^{\frac{5}{r}-\frac{5}{3}}\psi_N|x|^\frac{1}{6}f\|_{L_x^2}
	\lesssim N^{\frac{5}{r}-\frac53} \norm{\psi_N |x|^{\frac16} f}_{L_x^2}.
\end{align*}
Thus,
\begin{align*}
	&\sum_{N \le M, M,N\in 2^\mathbb{Z}} \norm{v_N}_{L_t^{r_1}([0,\infty);L_x^{r_1})}  \norm{v_M}_{L_t^{r_2}([0,\infty);L_x^{r_2})}\\
	&{}\hspace{2.0cm} \leq \sum_{R\geq1}\sum_{N\in 2^\mathbb{Z}} \norm{v_N}_{L_t^{r_1}([0,\infty);L_x^{r_1})}  \norm{v_{NR}}_{L_t^{r_2}([0,\infty);L_x^{r_2})}\\
	&{}\hspace{2.0cm} \lesssim \sum_{R\ge1} R^{-\frac{5}{r_1}+\frac53} \sum_{N\in 2^\mathbb{Z}}\norm{\psi_N |x|^{\frac16} f}_{L_x^2} \norm{\psi_{NR} |x|^{\frac16} f}_{L_x^2}\\
	&\hspace{2.0cm} \leq \sum_{R\ge1} R^{-\frac{5}{r_1}+\frac53} \Bigl(\sum_{N\in 2^\mathbb{Z}}\norm{\psi_N |x|^{\frac16} f}_{L_x^2}^2\Bigr)^\frac{1}{2}\Bigl( \sum_{N\in2^\mathbb{Z}}\norm{\psi_{NR} |x|^{\frac16} f}_{L_x^2}^2\Bigr)^\frac{1}{2}\\
%	&{}\hspace{2.0cm} =\sum_{R\ge1} R^{-\frac{5}{r_1}+\frac53}\left(\sum_{N\in2^\mathbb{Z}}\int_{\mathbb{R}^3}\left|\psi_N(x)|x|^\frac{1}{6}f(x)\right|^2dx\right)^\frac{1}{2}\left(\sum_{N\in2^\mathbb{Z}}\int_{\mathbb{R}^3}\left|\psi_{NR}(x)|x|^\frac{1}{6}f(x)\right|^2dx\right)^\frac{1}{2}\\
	&{}\hspace{2.0cm} =\sum_{R\ge1} R^{-\frac{5}{r_1}+\frac53}\left(\int_{\mathbb{R}^3}\sum_{N\in2^\mathbb{Z}}\left|\psi_N(x)|x|^\frac{1}{6}f(x)\right|^2dx\right)^\frac{1}{2}\left(\int_{\mathbb{R}^3}\sum_{N\in2^\mathbb{Z}}\left|\psi_{NR}(x)|x|^\frac{1}{6}f(x)\right|^2dx\right)^\frac{1}{2}\\
	&{}\hspace{2.0cm} \lesssim \norm{f}_{\mathcal{F} \dot{H}^\frac16}^2 \sum_{R\ge1} R^{-\frac{5}{r_1}+\frac53}
	 \lesssim \norm{f}_{\mathcal{F} \dot{H}^\frac16}^2.
\end{align*}
%\textcolor{cyan}{
%We apply the following inequality to sixth inequality:
%\begin{align*}
%\sum_{N\in2^\mathbb{Z}}|\psi_N(x)|^2\leq 3.
%\end{align*}
%for any $x\in \mathbb{R}^3$. For each $x\in\mathbb{R}^3$, there exists $N_0\in2^\mathbb{Z}$ with $x\in\,$supp\,$\psi_{\frac{N_0}{2}}$, supp\,$\psi_{N_0}$, supp\,$\psi_{2N_0}$ and $x\notin$\,supp\,$\psi_N$ for any $N\neq \frac{N_0}{2}, N_0, 2N_0$.
%}
This completes the proof.
\end{proof}

\begin{proof}[Proof of Proposition \ref{P:cv}]
% We prove by contradiction. 
In what follows we denote various subsequences of $n$ again by $n$.
% Suppose that the conclusion is false, that is, there exists a subsequence of $n$ such that
% By assumption,
% \[
% 	\norm{e^{i t \Delta} \Phi_n}_{L^{q,\infty}L^r} + \norm{e^{i \frac{t}2 \Delta} \Psi_n}_{L^{q,\infty}L^r} \ge \varepsilon_0
% \]
% holds for some $\varepsilon_0>0$.
By the pigeon hole principle, 
\begin{align*}
	\norm{e^{i t \Delta} \Phi_n}_{L_t^{q,\infty}L_x^r}  \ge \frac{\varepsilon_0}2 \quad
\text{ or }\quad
	\norm{e^{\frac{1}{2}it \Delta} \Psi_n}_{L_t^{q,\infty}L_x^r} \ge \frac{\varepsilon_0}2
\end{align*}
holds for infinitely many $n$. We only consider the case where the former holds for infinitely many $n$.
The proof for the other case is similar.

By interpolation and boundedness lemma, there exists $\theta=\theta(q,r)>0$ such that
\begin{align*}
	\norm{e^{i t \Delta} \Phi_n}_{L_t^{q,\infty}L_x^r}  \lesssim \norm{e^{i t \Delta} \Phi_n}_{L_t^3([0,\infty);X^{\frac13,3}_{\frac12})}^\theta
	\norm{\Phi_n}_{\mathcal{F} \dot{H}^\frac12}^{1-\theta}.
\end{align*}
By means of Lemma \ref{L:iStr} and the assumption, we have
\begin{align*}
	\sup_{N\in 2^\mathbb{Z}} \norm{ e^{i t \Delta}  |x|^{\frac13} \psi_N \Phi_n}_{L^3_t([0,\infty);L_x^3)}
	\gtrsim_{M,\varepsilon_0,q,r} 1.
\end{align*}
%\textcolor{cyan}{
%Using \eqref{070}, we have $(-4t^2\Delta)^\frac{1}{6}\mathcal{M}_\frac{1}{2}(-t)e^{it\Delta}=\mathcal{M}_\frac{1}{2}(-t)e^{it\Delta}|x|^\frac{1}{3}$. Therefore, it follows that
%\begin{align*}
%\|e^{it\Delta}\Phi_n\|_{L_t^3([0,\infty);\dot{X}_\frac{1}{2}^{\frac{1}{3},3})}
%	&=\|(-4t^2\Delta)^\frac{1}{6}\mathcal{M}_\frac{1}{2}(-t)e^{it\Delta}\Phi_n\|_{L_t^3([0,\infty);L_x^3)}\\
%	&=\|\mathcal{M}_\frac{1}{2}(-t)e^{it\Delta}|x|^\frac{1}{3}\Phi_n\|_{L_t^3([0,\infty);L_x^3)}\\
%	&=\|e^{it\Delta}|x|^\frac{1}{3}\Phi_n\|_{L_t^3([0,\infty);L_x^3)}.
%\end{align*}
%Applying this identity, we obtain
%\begin{align*}
%\left(\frac{\varepsilon_0}{2}\right)^\frac{3}{\theta}
%	&\leq\|e^{it\Delta}\Phi_n\|_{L_t^{q,\infty}L_x^r}^\frac{3}{\theta}\lesssim\|e^{it\Delta}\Phi_n\|_{L_t^3([0,\infty);\dot{X}_\frac{1}{2}^{\frac{1}{3},3})}^3\|\Phi_n\|_{\mathcal{F}\dot{H}^\frac{1}{2}}^\frac{3(1-\theta)}{\theta}\leq M^\frac{3(1-\theta)}{\theta}\|e^{it\Delta}|x|^\frac{1}{3}\Phi_n\|_{L_t^3([0,\infty);L_x^3)}^3\\
%	&\lesssim M^\frac{3(1-\theta)}{\theta}\||x|^\frac{1}{3}\Phi_n\|_{\mathcal{F}\dot{H}^\frac{1}{6}}^2\sup_{N\in 2^\mathbb{Z}}\|e^{it\Delta}\psi_N|x|^\frac{1}{3}\Phi_n\|_{L_t^3([0,\infty);L_x^3)}\\
%	&=M^\frac{3(1-\theta)}{\theta}\|\Phi_n\|_{\mathcal{F}\dot{H}^\frac{1}{2}}^2\sup_{N\in 2^\mathbb{Z}}\|e^{it\Delta}\psi_N|x|^\frac{1}{3}\Phi_n\|_{L_t^3([0,\infty);L_x^3)}\\
%	&\leq M^{\frac{3(1-\theta)}{\theta}+2}\sup_{N\in 2^\mathbb{Z}}\|e^{it\Delta}\psi_N|x|^\frac{1}{3}\Phi_n\|_{L_t^3([0,\infty);L_x^3)}.
%\end{align*}
%}
One can choose a sequence $N_n$ so that
\begin{equation}\label{E:cvpf0}
	\norm{e^{i t \Delta}|x|^{\frac13}\psi_{N_n} \Phi_n}_{L_t^{3}([0,\infty);L_x^3)}
	\gtrsim 1.
\end{equation}
Since the scaling property and Strichartz' estimate give us
\begin{align*}
	\norm{e^{i t \Delta}|x|^{\frac13}\psi_{N_n} \Phi_n}_{L_t^{3}([0,\tau N_n^2];L_x^3)}
	&{}= N_n^{\frac53}\norm{e^{i t \Delta} (N_n|x|)^{\frac13}\psi \Phi_n(N_n \cdot)}_{L_t^{3}([0,\tau];L_x^3)}\\
	&{}\leq N_n^\frac{5}{3} \|1\|_{L_t^{12}([0,\tau])}\norm{e^{i t \Delta} (N_n|x|)^{\frac13}\psi \Phi_n(N_n \cdot)}_{L^4_t([0,\tau];L_x^3)}\\
%	&{}= N_n^{\frac53} \tau^{\frac1{12}} \norm{e^{i t \Delta} (N_n|x|)^{\frac13}\psi \Phi_n(N_n \cdot)}_{L^4_t([0,\tau];L_x^3)}\\	
	&{}\lesssim N_n^{\frac53} \tau^{\frac1{12}} \norm{(N_n|x|)^{\frac13}\psi \Phi_n(N_n \cdot)}_{L_x^2}\\
	&\lesssim \tau^{\frac1{12}} \norm{\Phi_n}_{\mathcal{F} \dot{H}^{\frac12}}
	\lesssim \tau^{\frac1{12}},
\end{align*}
one can choose $\tau_0=\tau_0(M,\varepsilon_0,q,r)>0$ small so that \eqref{E:cvpf0} is improved as
\begin{align*}
	\norm{e^{i t \Delta}|x|^{\frac13}\psi_{N_n} \Phi_n}_{L_t^{3}([\tau_0N_n^2,\infty);L_x^3)}
	\gtrsim 1
\end{align*}
for all $n\ge1$.
%\textcolor{cyan}{
%Since
%\begin{align*}
%\norm{e^{i t \Delta}|x|^{\frac13}\psi_{N_n} \Phi_n}_{L_t^{3}([\tau N_n^2,\infty);L_x^3)}
%	&\geq \norm{e^{i t \Delta}|x|^{\frac13}\psi_{N_n} \Phi_n}_{L_t^3L_x^3}-\norm{e^{i t \Delta}|x|^{\frac13}\psi_{N_n} \Phi_n}_{L_t^{3}([0,\tau N_n^2);L_x^3)}\\
%	&\geq c_1-c_2\tau^\frac{1}{12},
%\end{align*}
%if we take sufficiently small $\tau_0>0$ sach as $c_2\tau_0^\frac{1}{12}<\frac{c_1}{2}$, then we have
%\begin{align*}
%\norm{e^{i t \Delta}|x|^{\frac13}\psi_{N_n} \Phi_n}_{L_t^{3}([\tau_0 N_n^2,\infty);L_x^3)}>\frac{c_1}{2}.
%\end{align*}
%}
H\"older's inequality gives us
\begin{align*}
	&\norm{e^{i t \Delta}|x|^{\frac13}\psi_{N_n} \Phi_n}_{L_t^3([\tau_0N_n^2,\infty);L_x^3)}\\
	&\hspace{2.0cm} \le \norm{ |t|^{\frac32} e^{i t \Delta}|x|^{\frac13}\psi_{N_n} \Phi_n}_{L_t^{\infty}([\tau_0N_n^2,\infty);L_x^\infty)}^{\frac1{18}}
	\norm{|t|^{-\frac{3}{34}}e^{i t \Delta}|x|^{\frac13}\psi_{N_n} \Phi_n}_{L_t^{\frac{17}6}([\tau_0N_n^2,\infty);L_x^\frac{17}{6})}^{\frac{17}{18}}.
\end{align*}
Using the estimate
\begin{align*}
	\norm{|t|^{-\frac{3}{34}}e^{i t \Delta}|x|^{\frac13}\psi_{N_n} \Phi_n}_{L_t^{\frac{17}6}L_x^\frac{17}{6}}
	&{}\lesssim \||t|^{-\frac{3}{34}}\|_{L^{\frac{34}{3},\infty}}\norm{e^{i t \Delta}|x|^{\frac13}\psi_{N_n} \Phi_n}_{L^{\frac{34}{9},\frac{17}6}_t L^{\frac{17}6}_{x}} \\
	&\lesssim \norm{|\nabla|^\frac{3}{34}\mathcal{M}_\frac{1}{2}(-t)e^{i t \Delta}|x|^{\frac13}\psi_{N_n} \Phi_n}_{L^{\frac{34}{9},\frac{17}6}_t L^{\frac{34}{13}}_{x}} \\
	&{}\lesssim \||t|^{-\frac{3}{34}}\|_{L^{\frac{34}{3},\infty}}\norm{e^{i t \Delta}|x|^{\frac13}\psi_{N_n} \Phi_n}_{L^{\frac{17}{3},\frac{17}{6}}_t \dot{X}^{\frac3{34},\frac{34}{13}}_{\frac12}} \\
	&{}\lesssim \norm{e^{i t \Delta}|x|^{\frac13}\psi_{N_n} \Phi_n}_{L^{\frac{17}{3},2}_t \dot{X}^{\frac3{34},\frac{34}{13}}_{\frac12}}\\
	&{}\lesssim \norm{|x|^{\frac13+\frac3{34}}\psi_{N_n} \Phi_n}_{L_x^2}\\
	&{}\lesssim N_n^{-\frac4{51}},
\end{align*}
%\textcolor{cyan}{
%Last inequality:
%\begin{align*}
%\norm{|x|^{\frac13+\frac3{34}}\psi_{N_n} \Phi_n}_{L_x^2}^2
%	&=\int_{\frac{2}{3}N_n\leq|x|\leq\frac{5}{3}N_n}|x|^{\frac{2}{3}+\frac{3}{17}}|\psi_{N_n}(x)\Phi_n(x)|^2dx\leq \int_{\frac{2}{3}N_n\leq|x|\leq\frac{5}{3}N_n}|x|^{-\frac{8}{51}}|x||\Phi_n(x)|^2dx\\
%	&\lesssim N_n^{-\frac{8}{51}}\|\Phi_n\|_{\mathcal{F}\dot{H}^\frac{1}{2}}\lesssim N_n^{-\frac{8}{51}}.
%\end{align*}
%}
we reach to the estimate
\begin{align*}
	N_n^{-\frac43} \norm{ |t|^{\frac32} e^{i t \Delta}|x|^{\frac13}\psi_{N_n} \Phi_n}_{L_t^{\infty}([\tau_0N_n^2,\infty);L_x^\infty)}
	\gtrsim 1
\end{align*}
for all $n\ge1$.
%\textcolor{cyan}{
%\begin{align*}
%	&N_n^{-\frac43} \norm{ |t|^{\frac32} e^{i t \Delta}|x|^{\frac13}\psi_{N_n} \Phi_n}_{L_t^{\infty}([\tau_0N_n^2,\infty);L_x^\infty)}\\
%	&\hspace{1.5cm} \gtrsim N_n^{-\frac{4}{3}}\norm{e^{i t \Delta}|x|^{\frac13}\psi_{N_n} \Phi_n}_{L_t^3([\tau_0N_n^2,\infty);L_x^3)}^{18}\norm{|t|^{-\frac{3}{34}}e^{i t \Delta}|x|^{\frac13}\psi_{N_n} \Phi_n}_{L_t^{\frac{17}6}([\tau_0N_n^2,\infty);L_x^\frac{17}{6})}^{-17}\\
%	&\hspace{1.5cm} \gtrsim N_n^{-\frac{4}{3}}\cdot1\cdot N_n^\frac{4}{3}\sim1.
%\end{align*}
%}
There exist $t_n \ge \tau_0 N_n^2$ and $y_n \in \mathbb{R}^3$ such that
\begin{align}
	N_n^{-\frac43} | t_n^{\frac32} e^{i t_n \Delta}(|x|^{\frac13}\psi_{N_n} \Phi_n)(y_n)|
	\gtrsim 1. \label{E:cvpf1}
\end{align}
By the integral representation of the Schr\"odinger group, we obtain
\begin{align}
	&N_n^{-\frac43}|t_n^{\frac32} e^{i t_n \Delta}(|x|^{\frac13}\psi_{N_n} \Phi_n)(y_n)| \notag \\
	&{}\hspace{3.0cm} =N_n^{-\frac{4}{3}}\left|t_n^\frac{3}{2}(4\pi it_n)^{-\frac{3}{2}}\int_{\mathbb{R}^3}e^{\frac{i|x-y_n|^2}{4t_n}}|x|^\frac{1}{3}\psi_{N_n}(x)\Phi_n(x)dx\right| \notag \\
	&{}\hspace{3.0cm} \lesssim N_n^{-\frac32}\left| \int_{\mathbb{R}^3} e^{-i\frac{y_n}{2t_n}\cdot x} e^{i\frac{|x|^2}{4t_n}} (N_n^{\frac16}|x|^{-\frac16} \psi_{N_n}(x)) |x|^{\frac12}\Phi_n(x) dx
	\right| \notag \\
	&{}\hspace{3.0cm} = \left| \int_{\mathbb{R}^3} e^{i\frac{N_n^2|x|^2}{4t_n}} (|x|^{-\frac16} \psi(x)) |x|^{\frac12} (e^{-i\frac{N_ny_n}{2t_n}\cdot x}N_n^{2}\Phi_n(N_n x))dx \label{E:cvpf2}
	\right| .
\end{align}
Let
\begin{align*}
	\xi_n := - \frac{N_n y_n}{2t_n} \in \mathbb{R}^3, \quad h_n := N_n \in 2^{\mathbb{Z}}.
\end{align*}
Define a deformation $\mathcal{G}_n \in G$ so that $\mathcal{G}_n^{-1} = T(\xi_n)D(h_n)$.

Since $\{\mathcal{G}_n (\Phi_n,\Psi_n)\}_n$ is a bounded sequence in $\mathcal{F}\dot{H}^{\frac12} \times \mathcal{F}\dot{H}^{\frac12}$,
it weakly converges to a pair $(\widetilde{\Phi}, \widetilde{\Psi})\in \mathcal{F}\dot{H}^{\frac12} \times \mathcal{F}\dot{H}^{\frac12}$
along a subsequence.
It is obvious that $(\widetilde{\Phi}, \widetilde{\Psi}) \in \nu (\{(\Phi_n,\Psi_n)\})$.
Notice that $0<\frac{N_n^2}{4t_n} \le \frac{1}{4\tau_0}$. Hence, by extracting a subsequence if necessary, one has
\begin{align*}
	\int_{\mathbb{R}^3} e^{i\frac{N_n^2|x|^2}{4t_n}} (|x|^{-\frac16} \psi(x)) |x|^{\frac12} (e^{-i\frac{N_ny_n}{2t_n}\cdot x}N_n^{2}\Phi_n(N_n x))  dx
	\longrightarrow 
	\int_{\mathbb{R}^3} e^{ia |x|^2 } (|x|^{-\frac16} \psi(x)) |x|^{\frac12} \widetilde{\Phi}(x)  dx
\end{align*}
as $n\to\infty$, where $a\in \mathbb{R}$ is the limit of $\frac{N_n^2}{4t_n}$ along the (sub)sequence.
Plugging this with \eqref{E:cvpf1} and \eqref{E:cvpf2}, we conclude that
\begin{align*}
	1 \lesssim \left| \int_{\mathbb{R}^3} e^{ia |x|^2 } (|x|^{-\frac16} \psi(x)) |x|^{\frac12} \widetilde{\Phi}(x)  dx\right|
	\lesssim_\psi \norm{\widetilde{\Phi}}_{\mathcal{F} \dot{H}^{\frac12}}
	\le \norm{(\widetilde{\Phi},\widetilde{\Psi})}_{\mathcal{F} \dot{H}^{\frac12}\times \mathcal{F} \dot{H}^{\frac12}} 
	\le \eta (\{(\Phi_n,\Psi_n)\}).
\end{align*}
This is the desired estimate.
\end{proof}

%Proof of Main theorem

\section{Proof of Theorem \ref{T:l0}, Theorem \ref{T:case1}, and Theorem \ref{T:case2}}\label{Proof of main theorem}

In this section, we prove Theorem \ref{T:l0}, Theorem \ref{T:case1}, and Theorem \ref{T:case2}.
% We may see that Theorem \ref{T:case1} and Theorem \ref{T:case2} hold by establishing Theorem \ref{T:l0}.
The following proof shows all these theorems. 
% More precisely,Theorem \ref{T:case1} corresponds with the case of $\ell_{v_0}^\dagger=\ell_{v_0}<\ell_0$ in Theorem \ref{T:l0} and Theorem \ref{T:case2} corresponds with the case of $\ell_{v_0}^\dagger=\ell_0<\ell_{v_0}$ in Theorem \ref{T:l0}.

\begin{proof}[Proof of Theorems \ref{T:l0}, Theorems \ref{T:case1}, and Theorem \ref{T:case2}]
Fix $v_0\in\mathcal{F}\dot{H}^\frac{1}{2}$.
First, we consider the case $\ell_{v_0}^\dagger=\infty$.
In this case, we can obtain $\ell_{v_0}^\dagger=\ell_{v_0}=\ell_0=\infty$.
Indeed, we have $\infty=\ell_{v_0}^\dagger\leq\ell_{v_0}$ by Lemma \ref{Comparison of ell and ell^dagger}.
On the other hand, we have $\infty=\ell_{v_0}^\dagger\leq \ell_0^\dagger\leq \ell_0$ by Lemma \ref{P:l0lv0} and Lemma \ref{Comparison of ell and ell^dagger}.

From now on, we assume $\ell_{v_0}^\dagger<\infty$.
By definition of $\ell_{v_0}^\dagger$, we have $L_{v_0}(\ell_{v_0}^\dagger-\frac{1}{n})<\infty$ for each $n\in \mathbb{N}$, that is,
\begin{equation*}
\sup\left\{\|(u,v)\|_{W_1([0,\infty))\times W_2([0,\infty))}\middle|
\begin{array}{l}
(u,v)\text{ is the solution to \eqref{NLS} on }[0,\infty),\\[0.1cm]
v(0)=v_0,\,\,
\|u(0)\|_{\mathcal{F}\dot{H}^\frac{1}{2}}\leq\ell_{v_0}^\dagger-\frac{1}{n}
\end{array}
\right\}<\infty.
\end{equation*}
We note that $T_{\max}=\infty$ because of Proposition \ref{Scattering criterion}.
Since $L_{v_0}(\ell)<\infty$ for any $0\leq \ell<\ell_{v_0}^\dagger$, $L_{v_0}(\ell_{v_0}^\dagger)=\infty$, and $L_{v_0}$ is non-decreasing, we can take a sequence $\{m_n\}$ of $\mathbb{N}$ such that
\begin{align*}
L_{v_0}\Bigl(\ell_{v_0}^\dagger-\frac{1}{m_n}\Bigr)
	<L_{v_0}\Bigl(\ell_{v_0}^\dagger-\frac{1}{m_{n+1}}\Bigr)
\end{align*}
for each $n\in \mathbb{N}$.
We take a sequence $\{u_{0,n}\}\in\mathcal{F}\dot{H}^\frac{1}{2}$ satisfying
\begin{align}
\ell^\dagger_{v_0}-\frac{1}{m_n}
	< \|u_{0,n}\|_{\mathcal{F}\dot{H}^\frac{1}{2}}
	\leq \ell^\dagger_{v_0}-\frac{1}{m_{n+1}} \label{023}
\end{align}
and
\begin{align*}
L_{v_0}\Bigl(\ell_{v_0}^\dagger-\frac{1}{m_n}\Bigr)
	<\|(u_n,v_n)\|_{W_1([0,\infty))\times W_2([0,\infty))}
	\leq L_{v_0}\Bigl(\ell_{v_0}^\dagger-\frac{1}{m_{n+1}}\Bigr),
\end{align*}
where $(u_n,v_n)$ is the solution to \eqref{NLS} with the initial data $(u_{0,n},v_0)$.
Since $\{(u_{0,n},v_0)\}\subset \mathcal{F}\dot{H}^\frac{1}{2}\!\times\!\mathcal{F}\dot{H}^\frac{1}{2}$ is a bounded sequence, we apply Theorem \ref{Linear profile decomposition} to this sequence.
Then, there exists profile $\{(\phi^j,\psi^j)\}\subset\mathcal{F}\dot{H}^\frac{1}{2}\!\times\!\mathcal{F}\dot{H}^\frac{1}{2}$, remainder $\{(R_n^J,L_n^J)\}\subset \mathcal{F}\dot{H}^\frac{1}{2}\!\times\!\mathcal{F}\dot{H}^\frac{1}{2}$, and pairwise orthogonal families of deformations $\{\mathcal{G}_n^j=T(\xi_n^j)D(h_n^j)\}_n\subset G$ $(j=1,2,\ldots)$ such that
\begin{align}
(u_{0,n},v_0)=\sum_{j=1}^J\mathcal{G}_n^j(\phi^j,\psi^j)+(R_n^J,L_n^J) \label{026}
\end{align}
for any $J\geq 1$.
Since $v_0$ is independent of $n$, there exists unique $j_0$ such that $\psi^{j_0}=v_0$ and $\mathcal{G}_n^{j_0}=\text{Id}$.
Furthermore, the remainder for $v$-component is zero: $L_n^J=0$.
Rearranging the profile $(\phi^j,\psi^j)$, we may let $j_0=1$.
% That is, we have $\mathcal{G}_n^1=\text{Id}$ and
Then, the above decomposition reads as
\begin{align*}
(u_{0,n},v_0)
	%=\sum_{j=1}^J\mathcal{G}_n^j(\phi^j,\psi^j)+(R_n^J,L_n^J)
	=(\phi^1,v_0)+\sum_{j=2}^J\mathcal{G}_n^j(\phi^j,0)+(R_n^J,0).
\end{align*}
From Theorem \ref{Linear profile decomposition}, we have Pythagorean decomposition:
\begin{align}
\|u_{0,n}\|_{\mathcal{F}\dot{H}^\frac{1}{2}}^2=\sum_{j=1}^J\|\phi^j\|_{\mathcal{F}\dot{H}^\frac{1}{2}}^2+\|R_n^J\|_{\mathcal{F}\dot{H}^\frac{1}{2}}^2+o_n(1) \label{027}
\end{align}
for each $J\geq 1$.
The parameters are asymptotically orthogonal: if $j\neq k$, then
\begin{align}
\frac{h_n^j}{h_n^k}+\frac{h_n^k}{h_n^j}+\frac{|\xi_n^j-\xi_n^k|}{h_n^j}\longrightarrow\infty\ \ \text{ as }\ \ n\rightarrow\infty. \label{028}
\end{align}
The remainders satisfy
\begin{align*}
(\mathcal{G}_n^j)^{-1}R_n^J\xrightharpoonup[]{\hspace{0.4cm}}0\ \ \text{ in }\ \ \mathcal{F}\dot{H}^\frac{1}{2}\ \ \text{ as }\ \ n\rightarrow\infty
\end{align*}
for any $1\leq j\leq J$.
\begin{align}
\lim_{J\rightarrow\infty}\limsup_{n\rightarrow\infty}\|e^{it\Delta}R_n^J\|_{L_t^{q,\infty}L_x^r}=0 \label{030}
\end{align}
for any $1<q,r<\infty$ with $\frac{1}{q}\in(\frac{1}{2},1)$ and $\frac{2}{q}+\frac{3}{r}=2$.

We will prove that there exists only one $j_1$ satisfying $\phi^{j_1}\neq0$, and it satisfies
$\|\phi^{j_1}\|_{\mathcal{F}\dot{H}^\frac{1}{2}}=\ell_{v_0}^\dagger$. 
% and, after the rearrangement so that $j_1=2$ (when $j_1\neq1$),
% \begin{align*}
% \|R_n^2\|_{\mathcal{F}\dot{H}^\frac{1}{2}}\longrightarrow0\ \ \text{ as }\ \ n\rightarrow\infty.
% \end{align*}
From \eqref{027}, we have
\begin{align}
\sum_{j=1}^J\|\phi^j\|_{\mathcal{F}\dot{H}^\frac{1}{2}}^2\leq (\ell_{v_0}^\dagger)^2, \label{031}
\end{align}
and hence, $\|\phi^j\|_{\mathcal{F}\dot{H}^\frac{1}{2}}\leq\ell_{v_0}^\dagger$ holds for any $j\geq 1$.
Let $(\Phi_j,\Psi_j)$ be the solution to \eqref{NLS} with a initial data $(\phi^j,\psi^j)$.
We assume for contradiction that all $(\Phi_j,\Psi_j)$ scatter forward in time, that is,
\begin{align*}
\|(\Phi_j,\Psi_j)\|_{W_1([0,\infty))\times W_2([0,\infty))}<\infty
\end{align*}
is true for any $ j \ge 1$.
We set
\begin{align*}
	(\widetilde{w}_n^J,\widetilde{z}_n^J)
		:=\sum_{j=1}^J\left((\Phi_j)_{[h_n^j,\xi_n^j]}(t,x),(\Psi_j)_{[h_n^j,\xi_n^j]}(t,x)\right)
\end{align*}
and
\begin{align*}
	(\widetilde{u}_n^J,\widetilde{v}_n^J)
		:=(\widetilde{w}_n^J,\widetilde{z}_n^J)+(e^{it\Delta}R_n^J,0),
\end{align*}
where
\begin{align*}
	(\Phi_j)_{[h_n^j,\xi_n^j]}(t,x)
		:={h_n^j}^2e^{ix\cdot\xi_n^j}e^{-it|\xi_n^j|^2}\Phi_j({h_n^j}^2t,h_n^j(x-2t\xi_n^j)),
\end{align*}
\begin{align*}
(\Psi_j)_{[h_n^j,\xi_n^j]}(t,x):={h_n^j}^2e^{2ix\cdot\xi_n^j}e^{-2it|\xi_n^j|^2}\Psi_j({h_n^j}^2t,h_n^j(x-2t\xi_n^j)).
\end{align*}
We note that $((\Phi_j)_{[h_n^j,\xi_n^j]},(\Psi_j)_{[h_n^j,\xi_n^j]})$ is a solution to \eqref{NLS} with a initial data $\mathcal{G}_n^j(\phi^j,\psi^j)$.
Then, $(\widetilde{u}_n^J,\widetilde{v}_n^J)$ solves
\begin{gather*}
	i\partial_t\widetilde{u}_n^J+\Delta\widetilde{u}_n^J
		=\sum_{j=1}^J\left(i\partial_t(\Phi_j)_{[h_n^j,\xi_n^j]}+\Delta(\Phi_j)_{[h_n^j,\xi_n^j]}\right)
		=-2\sum_{j=1}^J(\Psi_j)_{[h_n^j,\xi_n^j]}\overline{(\Phi_j)_{[h_n^j,\xi_n^j]}},\\
	i\partial_t\widetilde{v}_n^J+\frac{1}{2}\Delta\widetilde{v}_n^J
		=\sum_{j=1}^J\left(i\partial_t(\Psi_j)_{[h_n^j,\xi_n^j]}+\frac{1}{2}\Delta(\Psi_j)_{[h_n^j,\xi_n^j]}\right)
		=-\sum_{j=1}^J(\Phi_j)_{[h_n^j,\xi_n^j]}^2.
\end{gather*}
We also set
\begin{align*}
\widetilde{e}_{1,n}^J:=i\partial_t\widetilde{u}_n^J+\Delta\widetilde{u}_n^J+2\widetilde{v}_n^J\overline{\widetilde{u}_n^J},
\end{align*}
\begin{align*}
\widetilde{e}_{2,n}^J:=i\partial_t\widetilde{v}_n^J+\frac{1}{2}\Delta\widetilde{v}_n^J+(\widetilde{u}_n^J)^2.
\end{align*}
Here, we introduce the following two lemmas.

\begin{lemma}\label{Lemma1 for main result}
For any $\varepsilon>0$, there exists $J_0=J_0(\varepsilon)$ such that
\begin{align*}
\limsup_{n\rightarrow\infty}\|(\widetilde{w}_n^J,\widetilde{z}_n^J)-(\widetilde{w}_n^{J_0},\widetilde{z}_n^{J_0})\|_{W_1([0,\infty))\times W_2([0,\infty))}\leq\varepsilon
\end{align*}
for any $J\geq J_0$.
\end{lemma}

\begin{lemma}\label{Lemma2 for main result}
It follows that
\begin{align*}
\lim_{J\rightarrow\infty}\limsup_{n\rightarrow\infty}\|(\widetilde{e}_{1,n}^J,\widetilde{e}_{2,n}^J)\|_{N_1([0,\infty))\times N_2([0,\infty))}=0.
\end{align*}
\end{lemma}

These are shown as in \cite{Mas16}.
% We will prove these two lemmas later. 
% We check that they give the desired result.
Using Lemma \ref{Lemma1 for main result} with $\varepsilon=1$, it follows that there exists $J_0$ such that
\begin{align}
	&\|(\widetilde{u}_n^J,\widetilde{v}_n^J)\|_{W_1([0,\infty))\times W_2([0,\infty))} \notag \\
%	&\hspace{0cm}\leq\|(\widetilde{w}_n^J,\widetilde{z}_n^J)\|_{W_1([0,\infty))\times W_2([0,\infty))}+\|e^{it\Delta}R_n^J\|_{W_1([0,\infty))} \notag \\
	&\hspace{0cm}\leq\|(\widetilde{w}_n^{J_0},\widetilde{z}_n^{J_0})\|_{W_1([0,\infty))\times W_2([0,\infty))}+\|(\widetilde{w}_n^J,\widetilde{z}_n^J)-(\widetilde{w}_n^{J_0},\widetilde{z}_n^{J_0})\|_{W_1([0,\infty))\times W_2([0,\infty))}+\|e^{it\Delta}R_n^J\|_{W_1([0,\infty))} \notag \\
%	&\hspace{0cm}\leq\sum_{j=1}^{J_0}\left\|\left((\Phi_j)_{[h_n^j,\xi_n^j]},(\Psi_j)_{[h_n^j,\xi_n^j]}\right)\right\|_{W_1([0,\infty))\times W_2([0,\infty))}+\|e^{it\Delta}R_n^J\|_{W_1([0,\infty))}+1 \notag \\
	&\hspace{0cm}\leq\sum_{j=1}^{J_0}\left\|\left(\Phi_j,\Psi_j\right)\right\|_{W_1([0,\infty))\times W_2([0,\infty))}+c\|R_n^J\|_{\mathcal{F}\dot{H}^\frac{1}{2}}+1 \notag \\
	&\hspace{0cm}\leq\sum_{j=1}^{J_0}\left\|\left(\Phi_j,\Psi_j\right)\right\|_{W_1([0,\infty))\times W_2([0,\infty))}+c\ell_{v_0}^\dagger+1=:M \label{032}
\end{align}
for any $J\geq J_0$ and $n\geq1$.
Let $\varepsilon_1$ be given in Proposition \ref{Long time perturbation}.
Then,
\begin{align}
	&\left\|\left(u_{0,n}-\widetilde{u}_n^J(0),v_0-\widetilde{v}_n^J(0)\right)\right\|_{\mathcal{F}\dot{H}^\frac{1}{2}\times \mathcal{F}\dot{H}^\frac{1}{2}}
%	&\hspace{2.0cm}=\left\|\left(e^{it\Delta}(u_{0,n}-\sum_{j=1}^Je^{ix\cdot\xi_n^j}\phi^j_{\{h_n^j\}}-R_n^J),e^{\frac{1}{2}it\Delta}(v_0-\psi_{\{h_n^1\}}^1)\right)\right\|_{W_1([0,\infty))\times W_2([0,\infty))} \notag \\
	=0. \label{033}
\end{align}
Lemma \ref{Lemma2 for main result} implies that there exists $J_1$ such that
\begin{align*}
\limsup_{n\rightarrow\infty}\left\|\left(\widetilde{e}_{1,n}^J,\widetilde{e}_{2,n}^J\right)\right\|_{N_1([0,\infty))\times N_2([0,\infty))}\leq\frac{\varepsilon_1}{4}.
\end{align*}
for any $J\geq J_1$.
Choose $J$ with $J\geq\max\{J_0,J_1\}$.
There exists $n_0$ such that
\begin{align}
\left\|\left(\widetilde{e}_{1,n}^J,\widetilde{e}_{2,n}^J\right)\right\|_{N_1([0,\infty))\times N_2([0,\infty))}\leq\frac{\varepsilon_1}{2} \label{034}
\end{align}
for any $n\geq n_0$.
By \eqref{032}, \eqref{033}, \eqref{034}, and Proposition \ref{Long time perturbation}, we deduce that a solution $(u_n,v_n)$ to \eqref{NLS} with a initial data $(u_{0,n},v_0)$ satisfies
\begin{align*}
\|(u_n,v_n)\|_{W_1([0,\infty))\times W_2([0,\infty))}
	\leq C(M,\varepsilon_1)
	<\infty
\end{align*}
for any $n\geq n_0$.
However, this contradicts with the definition of $(u_n,v_n)$.
Therefore, there exists $j_1\ge1$ such that
\begin{align*}
\|(\Phi_{j_1},\Psi_{j_1})\|_{W_1([0,T_{\max}))\times W_2([0,T_{\max}))}=\infty.
\end{align*}
By \eqref{031}, another characterization of $\ell_{v_0}^\dagger$ (Proposition \ref{Another characterization}), and Proposition \ref{P:l0lv0}, we have $\|\phi^{j_1}\|_{\mathcal{F}\dot{H}^\frac{1}{2}}=\ell_{v_0}^\dagger$ and $\phi^{j}=0$ for all $j\neq j_1$.
We encounter a dichotomy, $j_1=1$ or $j_1=2$.

Now, we suppose that $j_1=1$. %(This case corresponds with Theorem \ref{T:case1}.)
Since a solution $(\Phi_1,\Psi_1)$ to \eqref{NLS} with a initial data $(\phi^1,v_0)$ does not scatter, we have $\ell_{v_0}\leq\|\phi^1\|_{\mathcal{F}\dot{H}^\frac{1}{2}}=\ell_{v_0}^\dagger$ by the definition of $\ell_{v_0}$.
Combining this inequality and Lemma \ref{Comparison of ell and ell^dagger}, we obtain $\ell_{v_0}=\ell_{v_0}^\dagger=\|\phi^1\|_{\mathcal{F}\dot{H}^\frac{1}{2}}$.
This shows that $\phi^1$ is a minimizer to $\ell_{v_0}$.
% In addition, we have
% \begin{align*}
% u_{0,n}=\mathcal{G}_{n,1}^1\phi^1+R_n^1,\ \ \lim_{n\rightarrow\infty}\|u_{0,n}\|_{\mathcal{F}\dot{H}^\frac{1}{2}}=\ell_{v_0}^\dagger,\ \ \text{ and }\ \ \lim_{n\rightarrow\infty}\|R_n^1\|_{\mathcal{F}\dot{H}^\frac{1}{2}}=0
% \end{align*}
% by \eqref{023}, \eqref{026}, \eqref{027}, and $\|\phi^1\|_{\mathcal{F}\dot{H}^\frac{1}{2}}=\ell_{v_0}^\dagger$.
Moreover, it follows from Lemmas \ref{P:l0lv0} and Lemma \ref{Comparison of ell and ell^dagger} that $\ell_{v_0}=\ell_{v_0}^\dagger\leq\ell_0^\dagger\leq\ell_0$.
Therefore, we have the identity $\ell_{v_0}^\dagger=\min\{\ell_0,\ell_{v_0}\}$.
%Furthermore, this case does not take place under the assumption of Theorem \ref{T:case2}.

Let us move on to the case $j_1=2$. 
%(This case corresponds with Theorem \ref{T:case2}.)
In this case, it follows that $(\phi^1,\psi^1)=(0,v_0)$ and $(\phi^{2},\psi^{2})=(\phi^{2},0)$.
Since $(\Phi_{2},\Psi_{2})$ does not scatter, we have $\ell_0\leq\|\phi^{2}\|_{\mathcal{F}\dot{H}^\frac{1}{2}}=\ell_{v_0}^\dagger$ by the definition of $\ell_0$.
Using Lemma \ref{P:l0lv0} and Lemma \ref{Comparison of ell and ell^dagger}, we obtain
\begin{align*}
\ell_0\leq\|\phi^{2}\|_{\mathcal{F}\dot{H}^\frac{1}{2}}=\ell_{v_0}^\dagger\leq\ell_0^\dagger\leq\ell_0.
\end{align*}
In particular, we have $\ell_{v_0}^\dagger=\ell_0=\|\phi^{2}\|_{\mathcal{F}\dot{H}^\frac{1}{2}}$.
This shows that $\phi^2$ is a minimizer to $\ell_{0}$.
In addition, we have
\begin{gather*}
(u_{0,n},v_0)
	=\sum_{j=1,2}\mathcal{G}_n^j(\phi^j,\psi^j)+(R_n^2,0)
	=(0,v_0)+\mathcal{G}_n^{2}(\phi^{2},0)+(R_n^2,0),\\
\lim_{n\rightarrow\infty}\|u_{0,n}\|_{\mathcal{F}\dot{H}^\frac{1}{2}}
	=\ell_{v_0}^\dagger,\ \ \text{ and }\ \ 
\lim_{n\rightarrow\infty}\|R_n^2\|_{\mathcal{F}\dot{H}^\frac{1}{2}}=0
\end{gather*}
by \eqref{023}, \eqref{026}, \eqref{027}, and $\|\phi^{j_0}\|_{\mathcal{F}\dot{H}^\frac{1}{2}}=\ell_{v_0}^\dagger$.
Remark that we have the identity $\ell_{v_0}^\dagger=\min\{\ell_0,\ell_{v_0}\}$ also in this case.
% this case does not take place under the assumption of Theorem \ref{T:case1}.

In both cases, we have the identity $\ell_{v_0}^\dagger=\min\{\ell_0,\ell_{v_0}\}$, hence we have Theorem \ref{T:l0}.
If we assume that $\ell_0>\ell_{v_0}^\dagger$ then the second case is precluded. 
This is nothing but Theorem \ref{T:case1}.

Similarly, the assumption $\ell_{v_0}>\ell_{v_0}^\dagger$ precludes the case $j_1=1$.
This shows Theorem \ref{T:case2}.
Indeed, the above argument applies to the minimizing sequence satisfying the assumption of Theorem \ref{T:case2}
and leads us to the same conclusion in the case $j_1=2$.
% We consider the minimizer sequence $\{\mathcal{G}_n^{j_0}(\phi^{j_0},0)\}$ to $\ell_0$.
% Recalling the way of the construction of the profile, we have
% \begin{align*}
% \displaystyle \lim_{n\rightarrow\infty}(|\log h_n^{j_0}|+|\xi_n^{j_0}|)=\infty\ \ \text{ and }\ \ (\mathcal{G}_{n,1}^{j_0})^{-1}u_{0,n}\xrightharpoonup[]{\hspace{0.4cm}}\phi^{j_0}\ \text{ in }\ \mathcal{F}\dot{H}^\frac{1}{2}
% \end{align*}
% by Lemma \ref{Lemma for weak convergence}.
% Since $(\mathcal{G}_{n,1}^{j_0})^{-1}u_{0,n}$ converges weakly to $\phi^{j_0}$ in $\mathcal{F}\dot{H}^\frac{1}{2}$ and $\|(\mathcal{G}_{n,1}^{j_0})^{-1}u_{0,n}\|_{\mathcal{F}\dot{H}^\frac{1}{2}}$ converges to $\|\phi^{j_0}\|_{\mathcal{F}\dot{H}^\frac{1}{2}}$, we have
% \begin{align*}
% (\mathcal{G}_{n,1}^{j_0})^{-1}u_{0,n}\longrightarrow\phi^{j_0}\ \text{ in }\ \mathcal{F}\dot{H}^\frac{1}{2}.
% \end{align*}
% Moreover, if we use Lemma \ref{Comparison of ell and ell^dagger} again, then it follows that $\ell_0=\ell_{v_0}^\dagger\leq\ell_{v_0}$.
% Therefore, we get $\ell_{v_0}^\dagger=\min\{\ell_0,\ell_{v_0}\}$.
Let $T_{\max}$ denote the maximal existence time of a solution to \eqref{NLS} with a initial data $(\phi^{2},0)$.
Fix $0\leq \tau <T_{\max}$.
% We use the notations $(\Phi_j,\Psi_j)$ and $\Bigl((\Phi_j)_{[h_n^j,\xi_n^j]},(\Psi_j)_{[h_n^j,\xi_n^j]}\Bigr)$ again.
Recall that $(\Phi_j,\Psi_j)$ denotes the solution to \eqref{NLS} with a initial data $(\phi^j,\psi^j)$, and $\Bigl((\Phi_j)_{[h_n^j,\xi_n^j]},(\Psi_j)_{[h_n^j,\xi_n^j]}\Bigr)$ does the solution to \eqref{NLS} with a initial data $\mathcal{G}_n^j(\phi^j,\psi^j)$.
We set
\begin{align*}
	(\widetilde{u}_n,\widetilde{v}_n)
		:=\sum_{j=1,2}\left((\Phi_j)_{[h_n^j,\xi_n^j]},(\Psi_j)_{[h_n^j,\xi_n^j]}\right)
		=(0,e^{\frac{1}{2}it\Delta}v_0)+\left((\Phi_{2})_{[h_n^{2},\xi_n^{2}]},(\Psi_{2})_{[h_n^{2},\xi_n^{2}]}\right).
\end{align*}
Then, $(\widetilde{u}_n,\widetilde{v}_n)$ solves
\begin{align*}
	i\partial_t\widetilde{u}_n+\Delta\widetilde{u}_n
		&=\sum_{j=1,2}\left(i\partial_t(\Phi_j)_{[h_n^j,\xi_n^j]}+\Delta(\Phi_j)_{[h_n^j,\xi_n^j]}\right)
		%=-2\sum_{j=1,2}(\Psi_j)_{[h_n^j,\xi_n^j]}\overline{(\Phi_j)_{[h_n^j,\xi_n^j]}}
		=-2(\Psi_{2})_{[h_n^{2},\xi_n^{2}]}\overline{(\Phi_{2})_{[h_n^{2},\xi_n^{2}]}},\\
	i\partial_t\widetilde{v}_n+\frac{1}{2}\Delta\widetilde{v}_n
		&=\sum_{j=1,2}\left(i\partial_t(\Psi_j)_{[h_n^j,\xi_n^j]}+\frac{1}{2}\Delta(\Psi_j)_{[h_n^j,\xi_n^j]}\right)
		=-(\Phi_2)_{[h_n^2,\xi_n^2]}^2.
\end{align*}
We also set
\begin{gather*}
	\widetilde{e}_{1,n}
		:=i\partial_t\widetilde{u}_n+\Delta\widetilde{u}_n+2\widetilde{v}_n\overline{\widetilde{u}_n}
%		&=-2\sum_{j=1}^J(\Psi_j)_{[h_n^j,\xi_n^j]}\overline{(\Phi_j)_{[h_n^j,\xi_n^j]}}+2\sum_{j=1,\,j_0}(\Psi_j)_{[h_n^j,\xi_n^j]}\sum_{j=1,\,j_0}\overline{(\Phi_j)_{[h_n^j,\xi_n^j]}}\\
		=2(\Psi_{1})_{[h_n^1,\xi_n^1]}\overline{(\Phi_{2})_{[h_n^{2},\xi_n^{2}]}},\\
	\widetilde{e}_{2,n}
		:=i\partial_t\widetilde{v}_n+\frac{1}{2}\Delta\widetilde{v}_n+(\widetilde{u}_n)^2
%		=-\sum_{j=1,\,j_0}(\Phi_j)_{[h_n^j,\xi_n^j]}^2+\Bigl(\sum_{j=1,\,j_0}(\Phi_j)_{[h_n^j,\xi_n^j]}\Bigr)^2
		=0.
\end{gather*}
We check the assumptions of Proposition \ref{Long time perturbation}. One has
\begin{align*}
	&\|(\widetilde{u}_n,\widetilde{v}_n)\|_{W_1([0,\tau/(h_n^{2})^2))\times W_2([0,\tau/(h_n^{2})^2))}\\
		%\leq \sum_{j=1,\,j_0}\left\|\left((\Phi_j)_{[h_n^j,\xi_n^j]},(\Psi_j)_{[h_n^j,\xi_n^j]}\right)\right\|_{W_1([0,\tau/{h_n^{j_0}}^2))\times W_2([0,\tau/{h_n^{j_0}}^2))}\\
%		&\leq \|(0,e^{\frac{1}{2}it\Delta}v_0)\|_{W_1([0,\infty))\times W_2([0,\infty))}+\left\|\left((\Phi_{j_0})_{[h_n^{j_0},\xi_n^{j_0}]},(\Psi_{j_0})_{[h_n^{j_0},\xi_n^{j_0}]}\right)\right\|_{W_1([0,\tau/{h_n^{j_0}}^2))\times W_2([0,\tau/{h_n^{j_0}}^2))}\\
		& \qquad\leq \|(0,e^{\frac{1}{2}it\Delta}v_0)\|_{W_1([0,\infty))\times W_2([0,\infty))}+\|(\Phi_{j_0},\Psi_{j_0})\|_{W_1([0,\tau))\times W_2([0,\tau))} =:M<\infty,
\end{align*}
\begin{align*}
	\|(u_{0,n},v_0)-(\widetilde{u}_n(0),\widetilde{v}_n(0))\|_{\mathcal{F}\dot{H}^\frac{1}{2}\times \mathcal{F}\dot{H}^\frac{1}{2}}
		=\|(R_n^2,0)\|_{\mathcal{F}\dot{H}^\frac{1}{2}\times \mathcal{F}\dot{H}^\frac{1}{2}}
		\longrightarrow0\ \ \text{ as }\ \ n\rightarrow\infty,
\end{align*}
and
\begin{align*}
	\|(\widetilde{e}_{1,n},\widetilde{e}_{2,n})\|_{N_1([0,\tau/(h_n^{2})^2))\times N_2([0,\tau/(h_n^{2})^2))}
		= \|\widetilde{e}_{1,n}\|_{N_1([0,\tau/(h_n^{2})^2))}
		\longrightarrow0\ \ \text{ as }\ \ n\rightarrow\infty,
\end{align*}
where the last estimate is shown as in the same spirit of Lemma \ref{Lemma2 for main result} with a help of the first estimate.
Therefore, we obtain
\begin{align*}
	\left\|(u_n,v_n)-(0,e^{\frac{1}{2}it\Delta}v_0)-\Bigl((\Phi_{2})_{[h_n^{2},\xi_n^{2}]},(\Psi_{2})_{[h_n^{2},\xi_n^{2}]}\Bigr)\right\|_{L_t^\infty([0,\tau/(h_n^{2})^2);\dot{X}_{1/2}^{1/2})\times L_t^\infty([0,\tau/(h_n^{2})^2);\dot{X}_1^{1/2})}
		\longrightarrow0
\end{align*}
as $n\rightarrow\infty$.
\end{proof}

We next consider the optimizing problem $\ell_f$ defined in \eqref{E:definition of lf}.

\begin{theorem}\label{T:ellfminimizer}
Let $f(x,y)$ be a function on $[0,\infty)\times [0,\infty)$ satisfying the following three conditions:
\begin{itemize}
\item Strictly increasing with respect to the both variables, i.e.,
\[
	0\le x_1 \le x_2 ,\, 0\le y_1 \le y_2 \Longrightarrow f(x_1,y_1) \le f(x_2,y_2) 
\]
and the equality holds only if $x_1=x_2$ and $y_1=y_2$.
\item Continuous, i.e., for any $(x_0,y_0)\in [0,\infty)\times [0,\infty)$,
\[
	\lim_{[0,\infty)\times [0,\infty) \ni (x,y) \to (x_0,y_0) } f(x, y) = f(x_0,y_0).
\]
\item $f(0,0)=0$.
\end{itemize} 
Let $\ell_f$ be defined in \eqref{E:definition of lf}.
Define
\[
	\widetilde{\ell}_f := \inf_{v_0 \in \mathcal{F} \dot{H}^{\frac12}} f(\ell_{v_0}, \norm{v_0}_{\mathcal{F} \dot{H}^{\frac12}}).
\]
Then, it follows that
\[
	\ell_f= \widetilde{\ell}_f= \inf_{v_0 \in \mathcal{F} \dot{H}^{\frac12}} f(\ell_{v_0}^\dagger, \norm{v_0}_{\mathcal{F} \dot{H}^{\frac12}}).
\]
Furthermore, 
 there exists a minimizer $(u^{(f)}(t),v^{(f)}(t))$ to $\ell_{f}$ such that
\begin{enumerate}
\item $f(\norm{u^{(f)}(0)}_{\mathcal{F} \dot{H}^{\frac12}}, \norm{v^{(f)}(0)}_{\mathcal{F} \dot{H}^{\frac12}}) = \ell_{f}$;
\item $(u^{(f)}(t),v^{(f)}(t))$ does not scatter forward in time;
\item $\norm{u^{(f)}(0)}_{\mathcal{F} \dot{H}^{\frac12}}=\ell_{v^{(f)}(0)}$.
\end{enumerate}
The minimizer is not a ground state.
\end{theorem}

\begin{proof}
Let us first show $\ell_f \ge \widetilde{\ell}_f$.
By definition of $\ell_{v_0}$ and the fact that $f$ is increasing in $x$, one sees that
the inequality
\[
	f(\ell_{v_0}, \norm{v_0}_{\mathcal{F} \dot{H}^{\frac12}})  \le f(\norm{u_0}_{\mathcal{F} \dot{H}^{\frac12}}, \norm{v_0}_{\mathcal{F} \dot{H}^{\frac12}}) 
\]
is true for any $(u_0,v_0) \notin \mathcal{S}_+$. Taking the infimum over $(u_0,v_0) \notin \mathcal{S}_+$, we obtain 
$\widetilde{\ell}_{f} \le \ell_f$.
% On the other hand,
% For any $\varepsilon>0$, there exists $v_0 \in \mathcal{F} \dot{H}^{\frac12}$ such that
% \[
% 	f(\ell_{v_0}, \norm{v_0}_{\mathcal{F} \dot{H}^{\frac12}}) < \widetilde{\ell}_f +\varepsilon
% \]
% by definition of $\widetilde{\ell}_f$.
% Moreover, by the continuity of $f$,
% there exists $\delta=\delta(\ell_{v_0},\norm{v_0}_{\mathcal{F} \dot{H}^{\frac12}}) >0$ such that if $x \in [\ell_{v_0}, \ell_{v_0}+\delta)$ then
% \[
% 	f(x, \norm{v_0}_{\mathcal{F} \dot{H}^{\frac12}}) \le f(\ell_{v_0}, \norm{v_0}_{\mathcal{F} \dot{H}^{\frac12}}) + \varepsilon
% \]
% Furthermore, by definition of $\ell_{v_0}$, there exists $u_0 \in \mathcal{F} \dot{H}^{\frac12}$ such that
% $(u_0,v_0) \notin \mathcal{S}_+$ and
% \[
% 	\ell_{v_0} \le \norm{u_0}_{\mathcal{F} \dot{H}^{\frac12}} < \ell_{v_0}+\delta.
% \]
% Combining these inequalities and the fact that $f$ is increasing in $x$, one has
% \[
% 	f(\norm{u_0}_{\mathcal{F} \dot{H}^{\frac12}}, \norm{v_0}_{\mathcal{F} \dot{H}^{\frac12}}) \le f(\ell_{v_0}+\delta, \norm{v_0}_{\mathcal{F} \dot{H}^{\frac12}}) \le f(\ell_{v_0}, \norm{v_0}_{\mathcal{F} \dot{H}^{\frac12}}) +\varepsilon < \widetilde{\ell}_f + 2\varepsilon.
% \]
% This shows $\ell_f < \widetilde{\ell}_f + 2\varepsilon$. As $\varepsilon>0$ is arbitrary, we obtain $\ell_f \le \widetilde{\ell}_{f}$.
% Thus $\ell_f = \widetilde{\ell}_{f}$.

Let us introduce $\ell^\dagger_f$ as follows:
\begin{equation*}
L_{f}(\ell):=\sup\left\{\|(u,v)\|_{W_1([0,T_{\max}))\times W_2([0,T_{\max}))} :
\begin{array}{l}
(u,v)\text{ is the solution to \eqref{NLS} on }[0,T_{\max}),\\[0.1cm]
f( \|u(0)\|_{\mathcal{F}\dot{H}^\frac{1}{2}}, \|v(0)\|_{\mathcal{F}\dot{H}^\frac{1}{2}}) \leq\ell
\end{array}
\right\},
\end{equation*}
and
\begin{equation}\label{D:ellfd}
\ell_{f}^\dagger:=\sup\{\ell:L_{f}(\ell)<\infty\} \in (0,\infty].
\end{equation}
Our next goal is to show $\ell_f=\ell_f^\dagger$.
Since $f(0,0)=0$ and $f$ is increasing with respect to the both variables, we see that 
\[
	\{(u(0),v(0))\in \mathcal{F}\dot{H}^\frac{1}{2} \times \mathcal{F}\dot{H}^\frac{1}{2} : f(\|u(0)\|_{\mathcal{F}\dot{H}^\frac{1}{2}},\|v(0)\|_{\mathcal{F}\dot{H}^\frac{1}{2}}) \le \ell \}	
\]
is a small neighborhood of $(0,0)$ for small $\ell>0$. 
% Further, the set contains nontrivial solutions since $f$ is right continuous in $x$.
By the small data theory, we have $L_f(\ell)\lesssim_f 1$ for small $\ell>0$, showing that $\ell_f^\dagger>0$.
Mimicking the argument in Proposition \ref{Continuity of L}, we see that $L_f(\ell)$ is a non-decreasing continuous extended function 
defined on $[0,\infty)$, thanks to the strictly increasing property of $f$ in both valuables.
We also have
\[
	\ell_{f}^\dagger=\inf\{\ell:L_{f}(\ell)=\infty\}
\]
and $\ell_{f}^\dagger \le \ell_f$ as in the proof of Propositions \ref{Another characterization} and Lemma \ref{Comparison of ell and ell^dagger}.

We now show the other direction $\ell_{f}^\dagger \ge \ell_f$.
Take an optimizing sequence for $\ell_f^\dagger$.
Then, by a similar argument to the proof of Theorem \ref{T:l0}, we obtain a minimizer $(u_c(t),v_c(t))$ to $\ell_f^\dagger$,
which completes the proof of $\ell_f = \ell_f^\dagger$.
We omit the details of the proof but point out several different respects compared with an optimizing sequence for $\ell_{v_0}^\dagger$.
First of all, the second component $v_{0,n}$ of the optimizing sequence may vary in $n$.
As a result, we do not have a priori information about the second component in the profile decompositions, 
hence the decomposition takes the form
\begin{align*}
(u_{0,n},v_{0,n})
	=\sum_{j=1}^J\mathcal{G}_n^j(\phi^j,\psi^j)+(R_n^J,L_n^J).
% 	=(\phi^1,v_0)+\sum_{j=2}^J\mathcal{G}_n^j(\phi^j,0)+(R_n^J,0).
\end{align*}
Here we remark that since $(\norm{u_{0,n}}_{\mathcal{F} \dot{H}^\frac12},\norm{v_{0,n}}_{\mathcal{F} \dot{H}^\frac12})$ belongs
to a compact set, say $\{ f(x,y) \le \ell_f^\dagger +1 \}$, one may suppose that it converges to a point $(x_\infty,y_\infty)$ such that
$f(x_\infty,y_\infty)=\ell_f^\dagger$ along a subsequence.

A contradiction argument shows there exists at most one $j$ such that $(\phi^j,\psi^j) \notin S_+$.
We may let $j=1$.
Then, the second difference is that we are also able to show that the number of nonzero profile is at most one.
This is because if $(\phi^2,\psi^2)\neq(0,0)$ then we have
\[
	\norm{\phi^1}_{\mathcal{F} \dot{H}^\frac12}^2 + \norm{\phi^2}_{\mathcal{F} \dot{H}^\frac12}^2 \le \lim_{n\to\infty} \norm{u_{0,n}}_{\mathcal{F} \dot{H}^\frac12}^2 = x_\infty^2
\]
and
\[
	\norm{\psi^1}_{\mathcal{F} \dot{H}^\frac12}^2 + \norm{\psi^2}_{\mathcal{F} \dot{H}^\frac12}^2 \le \lim_{n\to\infty} \norm{v_{0,n}}_{\mathcal{F} \dot{H}^\frac12}^2 = y_\infty^2
\]
by the Pythagorean decomposition. This shows $\norm{\phi^1}_{\mathcal{F} \dot{H}^\frac12}\le x_\infty$ and
$\norm{\psi^1}_{\mathcal{F} \dot{H}^\frac12}\le y_\infty$.
Since one of the above equality fails when $(\phi^2,\psi^2)\neq(0,0)$, the strictly increasing property of $f$ shows
\[
	f(\norm{\phi^1}_{\mathcal{F} \dot{H}^\frac12},\norm{\psi^1}_{\mathcal{F} \dot{H}^\frac12})
	< f(x_\infty,y_\infty) = \ell_f^\dagger.
\]
However, as $(\phi^1,\psi^1) \notin S_+$, the left hand side is not less than $\ell_f$.
Thus, we obtain $\ell_f < \ell_f^\dagger$, a contradiction.
We conclude that $\ell_f^\dagger = \ell_f$ and a solution $(u^{(f)}(t),v^{(f)}(t))$
which satisfies the initial condition $(u^{(f)}(0),v^{(f)}(0))=(\phi^1,\psi^1)$
is a desired minimizer to $\ell_f$.

To complete the proof, it suffices to show that
\[
	\ell_f^\dagger \le \inf_{v_0 \in \mathcal{F} \dot{H}^{\frac12}} f(\ell_{v_0}^\dagger, \norm{v_0}_{\mathcal{F} \dot{H}^{\frac12}})
\]
because we have already shown $\ell_f^\dagger=\ell_f\ge \widetilde{\ell}_f$ and because $\widetilde{\ell}_f$ is obviously greater than 
or equal to the right hand side in view of Theorem \ref{T:l0} and the increasing property of $f$.
Fix $v_0 \in \mathcal{F} \dot{H}^\frac12$.
By \eqref{E:Lvellv}, we can pick a sequence $\{u_{0,n}\}_{n} \subset \mathcal{F} \dot{H}^\frac12$ so that $\norm{u_{0,n}}_{\mathcal{F} \dot{H}^\frac12} \le \ell_{v_0}^\dagger$ and
the corresponding solution $(u_n(t),v_n(t))$ with the data $(u_n(0),v_n(0))=(u_{0,n},v_0)$ satisfies
\[
	\norm{(u_n,v_n)}_{W_1([0,T_{\max}))\times W_1([0,T_{\max}))} \ge n .
\]
% By the other characterization of $\ell_{v_0}^\dagger$, we have
% \[
% 	\norm{u_{0,n}}_{\mathcal{F} \dot{H}^\frac12} \to \ell_{v_0}^\dagger
% \]
% as $n\to\infty$. 
Further, it follows from the increasing property of $f$ that
\[
	f(\norm{u_{0,n}}_{\mathcal{F} \dot{H}^\frac12}, \norm{v_0}_{\mathcal{F} \dot{H}^{\frac12}}) \le f(\ell_{v_0}^\dagger, \norm{v_0}_{\mathcal{F} \dot{H}^{\frac12}})
\]
for all $n\ge0$. The existence of the above $\{u_{0,n}\}_{n}$ implies that
\[
	L_f( f(\ell_{v_0}^\dagger, \norm{v_0}_{\mathcal{F} \dot{H}^{\frac12}}) )=\infty.
\]
The other characterization of $\ell_f^\dagger$ then gives us
\[
	\ell_f^\dagger \le f(\ell_{v_0}^\dagger, \norm{v_0}_{\mathcal{F} \dot{H}^{\frac12}}).
\]
Since $v_0$ is arbitrary, we obtain the result.
\end{proof}

\section{Proof of corollaries of Theorem \ref{Nonpositive energy}}\label{Proof of corollaries}
We have proven Theorem \ref{Nonpositive energy} in Subsection \ref{subsec:Nonpositive energy}.
Let us show its corollaries.

\begin{proof}[Proof of Corollary \ref{Large data nonscattering}]
For given $v_0 \in \mathcal{F} \dot{H}^{\frac12} \cap H^1$ with $v_0\neq0$, we take
\[
	u_0 =  v_0(x)^{\frac12} |v_0 (x)|^{\frac12} \in \mathcal{F} \dot{H}^{\frac12} \cap H^1.
\]
Then,  we have
\[
	E[ c^{\frac12} d u_0,c v_0] \le c d^2 \norm{\nabla v_0}_{L^2}^2 + \frac{c^2}2 \norm{\nabla v_0}_{L^2}^2
	-2 c^{2}d^2\norm{v_0}_{L^3}^3
\]
for $c>0$ and 
% $d>0$. We choose 
$d=\norm{\nabla v_0}_{L^2} \norm{v_0}_{L^3}^{-3/2}$.
 There exists $c_0=c_0(v_0)>0$ such that the right side is negative for any $c\ge c_0$.
For such $c$, the corresponding solution does not scatter by virtue of Theorem \ref{Nonpositive energy}. This also shows the bound
\[
	\ell_{c v_0} \le \norm{c^{\frac12}d u_0}_{\mathcal{F} \dot{H}^{\frac12}} =
	c^{\frac12} \norm{v_0}_{\mathcal{F} \dot{H}^{\frac12}} \norm{\nabla v_0}_{L^2} \norm{v_0}_{L^3}^{-3/2}.
\]
We have the desired result.
\end{proof}

\begin{proof}[Proof of Corollary \ref{Sufficient condition of finite ell}]
% Let $v_0$ be real-valued function for simplicity.
% et $-\Delta-2\Re e^{i\theta }v_0$ have a negative eigenvalue such that a eigenfunction $\varphi\in\mathcal{F}\dot{H}^\frac{1}{2}\cap H^2$. 
% We take a eigenvalue $\lambda<0$ and a real-valued eigenfunction $\phi\in\mathcal{F}\dot{H}^\frac{1}{2}\cap H^1$, that is,
We have
\begin{align*}
-\Delta\varphi-2(\Re e^{i\theta} v_0)\varphi= \tilde{e}\varphi.
\end{align*}
Remark that $\varphi$ is real-valued.
Multiplying this identity by $\varphi$, and integrating, we have
\begin{align*}
(-\Delta\varphi,\varphi)_{L^2}-(2(\Re e^{i\theta} v_0)\varphi,\varphi)_{L^2}=\tilde{e}(\varphi,\varphi)_{L^2}.
\end{align*}
This can be rearranged as
\begin{align*}
\|\nabla\varphi\|_{L^2}^2 + 2\Re \int (-e^{-i\theta} \varphi^2) \overline{v_0} dx =\tilde{e} \|\varphi\|_{L^2}^2.
\end{align*}
Here, we take $u_0= e^{-i\theta/2} \varphi$. Then, 
\begin{align*}
E[cu_0,v_0]=c^2\|\nabla\varphi\|_{L^2}^2+\frac{1}{2}\|\nabla v_0\|_{L^2}^2- 2\Re \int (c^2e^{-i\theta} \varphi^2) \overline{v_0} dx
=c^2\tilde{e} \|\varphi\|_{L^2}^2+\frac{1}{2}\|\nabla v_0\|_{L^2}^2.
\end{align*}
From $\tilde{e}<0$, the choice $c^2=\frac{\|\nabla v_0\|_{L^2}^2}{2|\tilde{e}|\|\varphi\|_{L^2}^2}$ gives us $E(u_0,v_0)=0$. Therefore, $(cu_0,v_0)\notin S_+$ by  Theorem \ref{Nonpositive energy}. This also implies the bound
\begin{align*}
\ell_{v_0}\leq \|cu_0\|_{\mathcal{F}\dot{H}^\frac{1}{2}}= \frac{\|\varphi\|_{\mathcal{F}\dot{H}^\frac{1}{2}}}{\sqrt{2|\tilde{e}|}\|\varphi\|_{L^2}}\|\nabla v_0\|_{L^2}.
\end{align*}
We complete the proof.
\end{proof}

\subsection*{Acknowledgments}
M.H. was supported by JSPS KAKENHI Grant Number JP19J13300.
S.M. was supported by JSPS KAKENHI Grant Numbers JP17K14219, JP17H02854, JP17H02851, and JP18KK0386.

\end{document}